\documentclass[leqno]{amsart}
\usepackage{amssymb}
\usepackage{mathrsfs}
\usepackage{amsmath, amsfonts, vmargin, enumerate}
\usepackage{graphics}
\usepackage{verbatim}
\usepackage{color}
\usepackage{amsthm}
\usepackage{latexsym,bm}
\usepackage{euscript}
\usepackage{color}
\usepackage[colorlinks=true,citecolor=blue,linkcolor=blue]{hyperref}

\input xypic
\xyoption {all}

 \makeatletter

\newtheorem{thm}{Theorem}[section]
\newtheorem{prop}[thm]{Proposition}
\newtheorem{cor}[thm]{Corollary}
\newtheorem{lem}[thm]{Lemma}
\newtheorem{defi}[thm]{Definition}
\newtheorem{remark}[thm]{Remark}
\newtheorem{example}[thm]{Example}
\newtheorem{pb}[thm]{Problem}

\newenvironment{rk}{\begin{remark}\rm}{\end{remark}}
\newenvironment{definition}{\begin{defi}\rm}{\end{defi}}

\numberwithin{equation}{section}

\newcommand{\ra}{{\rightarrow}}
\newcommand{\be}{\begin{eqnarray*}}
\newcommand{\ee}{\end{eqnarray*}}
\newcommand{\beq}{\begin{equation}}
\newcommand{\eeq}{\end{equation}}
\newcommand{\beqn}{\begin{equation*}}
\newcommand{\eeqn}{\end{equation*}}

\newcommand{\na}{{\nabla}}
\newcommand{\p}{{\psi}}
\newcommand{\ri}{{\rm{i}}}

\newcommand{\dom}{{\mathrm{dom}}}
\newcommand{\vf}{{\varphi}}

\newcommand{\qe}{{\mathbb{R}_\theta^d}}
\newcommand{\ot}{{\otimes}}
\newcommand{\sgn}{{{\rm{sgn}}}}
\newcommand{\g}{{\gamma}}
\newcommand{\al}{{\alpha}}
\newcommand{\la}{{\lambda}}
\newcommand{\bt}{{\beta}}

\newcommand{\qcS}{{\mathcal{S}(\qe)}}
\newcommand{\qDist}{{\mathcal{S}'(\qe)}}
\newcommand{\cA}{{\mathcal{A}(\qe)}}

\newcommand{\rank}{\mathrm{rank}}
\newcommand{\interspace}{{\bigcup_{d\leq p <\infty} L_p(\qe)}}

\def\qd{\,{\mathchar'26\mkern-12mu d}}

\begin{document}

\title{Quantum differentiability on noncommutative Euclidean spaces}

\author{Edward MCDONALD}

\address{School of Mathematics and Statistics, UNSW, Kensington, NSW 2052, Australia}
\email{edward.mcdonald@unsw.edu.au}



\author{Fedor SUKOCHEV}

\address{School of Mathematics and Statistics, UNSW, Kensington, NSW 2052, Australia}
\email{f.sukochev@unsw.edu.au}

\author{Xiao XIONG}

\address{Institute for Advanced Study in Mathematics, Harbin Institute of Technology, Harbin 150001, China, and School of Mathematics and Statistics, UNSW, Kensington, NSW 2052, Australia}
\email{xxiong@hit.edu.cn}

\date{}
\maketitle

\markboth{E. Mcdonald, F. Sukochev and X. Xiong}%
{Quantum differentiability on noncommutative Euclidean spaces}

\begin{abstract}
    We study the topic of quantum differentiability on quantum Euclidean $d$-dimensional spaces (otherwise known as Moyal $d$-spaces), and we find conditions
    that are necessary and sufficient for the singular values of the quantised differential $\qd x$ to have decay $O(n^{-\alpha})$ for $0 < \alpha \leq \frac{1}{d}$.
    This result is substantially more difficult than the analogous problems for Euclidean space and for quantum $d$-tori.

\end{abstract}


\section{Introduction}

Quantum Euclidean spaces were first introduced by a number of authors, including Groenewold \cite{Groen} and Moyal \cite{Moyal}, for the study of quantum mechanics in phase space. The constructions of Groenewold and Moyal
were later abstracted into more general canonical commutation relation (CCR) algebras, and have since become fundamental in mathematical physics.
Under the names Moyal planes or Moyal-Groenewold planes, these algebras play the role of a central and motivating example in noncommutative geometry \cite{GGISV2004,CGRS2014}. 
As geometrical spaces with noncommutating spatial coordinates, noncommutative Euclidean spaces have appeared frequently in the mathematical physics literature \cite{DouglasNekrasov}, in the contexts of string theory \cite{SeibergWitten}
and noncommutative field theory \cite{NekrasovSchwarz}. 

Quantum Euclidean spaces have also been studied as an interesting noncommutative setting for classical and harmonic analysis, and for this we refer the reader to recent work such as \cite{GJP2017,LeSZ2017,MSZ2018,SZ2018a}. 

\medskip

    Connes introduced the quantised calculus in \cite{Connes-ncdg-1985} as a replacement for the algebra of differential forms for applications in a noncommutative setting, and afterwards this point of view found application to mathematical physics \cite{Connes1988}. Connes successfully applied his quantised calculus in providing a formula for the Hausdorff measure of Julia sets and for limit sets of Quasi-Fuchsian groups in the plane 
    \cite[Chapter 4, Section 3.$\gamma$]{Connes1994} (for a more recent exposition see \cite{CSZ,CMSZ}).
    
    Following \cite{Connes-ncdg-1985}, quantised calculus may be defined defined in terms of a Fredholm module. The idea behind a Fredholm module has its origins with Atiyah's work on $K$-homology \cite{Atiyah1969},
    and further details can be found in, for example, \cite[Chapter 8]{Higson-Roe-2000}.
    
    A Fredholm module can be defined with the following data: a separable Hilbert space $H$, a unitary self-adjoint operator $F$ on $H$ and a $C^*$-algebra $\mathcal{A}$ represented on $H$ such that the commutator $[F,a]$ is a compact operator on $H$ for all $a$ in $A$. The quantised differential of $a \in \mathcal{A}$ is then defined to be the operator $\qd a = \ri[F,a]$. 
    
    It is suggestive to think of the compact operators on $H$ as being analogous to ``infinitesimals", and one can measure the ``size" or ``order" of an infinitesimal $T$
    in terms of its singular value sequence:
    \begin{equation*}
        \mu(n,T) := \inf\{\|T-R\|\;:\;\mathrm{rank}(R)\leq n\}
    \end{equation*}
    where $\|\cdot\|$ is the operator norm.
    
    A problem of particular interest in quantised calculus is to precisely quantify the asymptotics of the sequence $\{\mu(n,\qd a)\}_{n=0}^\infty$ in terms of $a$. In operator theoretic language, we seek conditions under 
    which the operator $\qd a$ is in some ideal of the algebra of bounded operators on $H$. Of the greatest importance are Schatten-von Neumann ${\mathcal L}_p$ ideals, the Schatten-Lorentz ${\mathcal L}_{p,\infty}$ spaces
    and the Macaev-Dixmier ideal ${\mathcal M}_{1,\infty}$ (c.f. Section \ref{operator notation subsection} and \cite[Section 2.6]{LSZ2012}).
%
%
    
    The link between quantised calculus and geometry is discussed by Connes in \cite{Connes1988}. A model example for quantised calculus is to take a compact $d$-dimensional Riemannian spin manifold $M$ (with $d \geq 2$) with Dirac operator $D$,
    and define $H$ to be the Hilbert space of pointwise almost-everywhere equivalence classes of square integrable sections of the spinor bundle. The algebra $\mathcal{A} = C(M)$ of continuous functions on $M$ acts by pointwise multiplication on $H$, and one defines $F$ as a difference of spectral projections:
    \begin{equation*}
        F := \chi_{[0,\infty)}(D)-\chi_{(-\infty,0)}(D).
    \end{equation*}
    One then has $\qd f = \ri[F,M_f]$, where $M_f$ is the operator on $H$ of pointwise multiplication by $f \in C(M)$. 
    In quantised calculus the immediate question is to determine the relationship between the degree of differentiability of $f \in C(M)$ and 
    the rate of decay of the singular values of $\qd f$. In general, we have the following inclusion \cite[Theorem 3.1]{Connes1988}:
    \begin{equation*}
        f \in C^\infty(M) \Rightarrow |\qd f|^d \in \mathcal{M}_{1,\infty}.
    \end{equation*}
    This corresponds to the implication:
    \begin{equation*}
        f \in C^\infty(M) \Rightarrow \sup_{n\geq 0} \frac{1}{\log(2+n)} \sum_{j=0}^n \mu(j,\qd f)^d < \infty.
    \end{equation*}
    It is possible to specify even more precise details about the asymptotics of $\{\mu(j,\qd f)\}_{j\geq 0}$. Suppose that $\omega$ is an extended limit (a continuous linear functional on the space of bounded sequences $\ell_\infty({\mathbb N})$ which
    extends the limit functional). If $\omega$ is invariant under dilations (in the sense of \cite[Definition 6.2.4]{LSZ2012}) then \cite[Theorem 3.3]{Connes1988} states that:
    \begin{equation}\label{connes_limit_formula}
        \omega\left(\left\{\frac{1}{\log(2+n)}\sum_{j=0}^n \mu(j,\qd f)^d\right\}_{n= 0}^\infty\right) = c_d\int_{M} |df \wedge \star df|^{d/2}
    \end{equation}
     where $c_d$ is a known constant, $d$ is the exterior differential and $\star$ denotes the Hodge star operator associated to the orientiation of $M$. The quantity on the left hand side of \eqref{connes_limit_formula} is precisely the Dixmier trace ${\mathrm{tr}}_{\omega}(|\qd f|^d)$. According to Connes, this formula ``shows how to pass from quantized $1$-forms to ordinary forms, not by a classical limit, but by a direct application
    of the Dixmier trace" \cite[Page 676]{Connes1988}.

    \medskip
    
    When working with particular manifolds, rather than general compact manifolds, it is possible to specify with even greater precision the relationship between $f$
    and the singular values of $\qd f$. In the one dimensional cases of the circle and the line, the appropriate choice for $F$ turns out to be the Hilbert transform (see \cite[Chapter 4, Section 3.$\alpha$]{Connes1994})
    and the commutators of pointwise multiplication operators and the Hilbert transform are very well understood. If $f$ is a function on either the line ${\mathbb R}$ or the circle ${\mathbb{T}}$, necessary and sufficient 
    conditions for $\qd f$ to be in virtually every named operator ideal are known (see the discussion at the end of Chapter 6 of \cite{Peller2003}).
    
    In higher dimensions (in particular ${\mathbb T}^d$ and ${\mathbb R}^d$ for $d \geq 2$), an appropriate choice for $F$ is given by a linear combination of Riesz transforms \cite{CST,LMSZ2017}. Commutators of pointwise multiplication
    operators and Riesz transforms are well studied in classical harmonic analysis, and Janson and Wolff \cite{JW1982} determined necessary and sufficient conditions for such a commutator to be in ${\mathcal L}_p$ for all $p \in (0,\infty)$.
    An even more precise characterisation was obtained by Rochberg and Semmes \cite{RS1989}.

    If $f \in C^\infty({\mathbb T}^d)$, let $\nabla f = (\partial_1f,\partial_2f,\ldots,\partial_df)$ be the gradient vector of $f$, and let $\|\nabla f\|_2 = \left(\sum_{j=1}^d |\partial_j f|^2\right)^{\frac{1}{2}}$. Then as a special case of \eqref{connes_limit_formula}, we have the following:
    \begin{equation}\label{torus trace formula}
        {\mathrm{tr}}_\omega(|\qd f|^d) = k_d\int_{{\mathbb T}^d} \|\nabla f(t)\|_{2}^ddm(t),
    \end{equation}
    where $k_d > 0$ is a known constant, and $m$ denotes the flat (Haar) measure on ${\mathbb T}^d$. A similar integral formula can also be obtained in the non-compact setting of ${\mathbb R}^d$ \cite[Theorem 2]{LMSZ2017}.
%
%
         
    Despite having been heavily studied in the commutative setting, quantum differentiability in the strictly noncommutative setting is still largely unexplored. Recently the authors have established a characterisation of the ${\mathcal L}_{d,\infty}$-ideal membership of quantised differentials for noncommutative tori \cite{MSX2018}.
    The primary result of \cite{MSX2018} is as follows. Let $\theta$ be an antisymmetric real $d\times d $ matrix with $d>2$, and consider the noncommutative tori ${\mathbb T}_\theta^d$. In this setting, there is a conventional
    choice of Fredholm module and an associated quantised calculus \cite[Section 12.3]{GVF}.
    An element $x \in L_2({\mathbb T}^d_{\theta})$ belongs to the (noncommutative) homogeneous Sobolev space $\dot{W}_d^1({\mathbb T}_\theta^d)$ if and only if its quantised differential $\qd x $ has bounded extension in ${\mathcal L}_{d, \infty}$. 
    The quantum torus analogue of \eqref{torus trace formula} is also obtained as \cite[Theorem 1.2]{MSX2018}: for $x\in \dot{W}_d^1({\mathbb T}_\theta^d)$, there is a certain constant $c_d$ such that for any continuous normalised trace $\varphi$ on ${\mathcal L}_{1, \infty}$ we have 
    \begin{equation}\label{qt trace formula}
        \varphi(|\qd x|^d) = c_d\,\int_{\mathbb{S}^{d-1}} \tau\Bigg(\Big(\sum_{j=1}^d |\partial_j x-s_j\sum_{k=1}^d s_k\partial_k x|^2\Big)^{\frac{d}{2}}\Bigg)\,ds,
    \end{equation}
    where $\tau$ is the standard trace on the algebra $L_\infty({\mathbb T}^d_{\theta})$, and the integral is over the $d-1$-sphere $\mathbb{S}^{d-1}$ with respect to its rotation invariant measure $ds$.
    To the best of our knowledge, these results were the first concerning quantum differentiability in the strictly noncommutative setting.

        \medskip
    
    The primary task of this paper is to determine similar results for noncommutative Euclidean spaces. A number of major obstacles make this task far more difficult than for noncommutative tori. 
    In particular, the methods of \cite{MSX2018} were facilitated by a well-developed theory of pseudodifferential operators on noncommutative tori \cite{HLP2018,HLP2018a}.
    However, despite recent advances \cite{LM2016,GJP2017,MSZ2018}, the theory of pseudodifferential operators for noncommutative Euclidean spaces is still in its infancy and it is not clear how to directly adapt the existing theory to this problem. It has therefore been necessary for us to introduce new arguments based on operator theory rather than pseudodifferential operator theory (see Section \ref{commutator-estimates}). 
    
    Another difficulty with ${\mathbb R}^d_{\theta}$ compared to ${\mathbb T}^d_\theta$ is that the nature of the required analysis changes dramatically with $\theta$. For example, the range of the canonical trace $\tau$ on the algebra $L_\infty({\mathbb T}^d_\theta)$
    on projections is $[0,1]$, while for the canonical trace on $L_\infty({\mathbb R}^d_\theta)$ the range of the trace on projections is either $[0,\infty]$ if $\det(\theta)=0$ or instead ranges over integral multiples of $(2\pi)^{d/2}|\det(\theta)|^{1/2}$ if 
    $\det(\theta)\neq 0$.

        A noteworthy side effect of our self-contained approach is that we obtain in an abstract manner the following commutator estimates for quantum Euclidean spaces: 
        Let $\Delta_\theta$ be the Laplace operator associated to the noncommutative Euclidean space $\qe$ (see Section \ref{Sec-QES} for complete definitions).
        For an appropriate class of smooth elements $x\in L_\infty(\qe)$, if $\alpha, \beta\in {\mathbb R}$ are such that $\alpha < \beta +1 $, then we have
        \begin{equation*}
            [(1-\Delta_\theta)^{\alpha/2}, x](1-\Delta_\theta)^{-\beta/2} \in {\mathcal L}_{\frac{d}{\beta-\alpha+1},\infty}.
        \end{equation*}    
        In the classical (commutative) case, this estimate follows almost immediately from the calculus and the mapping properties of pseudodifferential operators (see \cite[Lemma 13]{LMSZ2017}).

         
         \medskip


\subsection{Main results on quantum differentiability}
    In this section we state the main results of this paper. Heretofore unexplained notation which we use will be defined in Section \ref{notation_section}.
    
    Let $\theta$ be an antisymmetric real $d\times d $ matrix, where $d\geq 2$.
     
    Our first main result provides sufficient conditions for $\qd x \in {\mathcal L}_{d,\infty}$:
    \begin{thm}\label{sufficiency}
        If $x \in L_p(\qe)\cap \dot W^1_d(\qe)$ for some $d\leq p < \infty$, then $\qd x $ has bounded extension, and the extension is in ${\mathcal L}_{d,\infty}$.
    \end{thm}
    The space $\dot W^{1}_d(\qe)$ is a noncommutative homogeneous Sobolev space defined with respect to the partial derivatives $\partial_j$, $j=1,\ldots,d$ (these notions will be defined and discussed in Subsection \ref{calculus definition subsubsection}).
    The {\emph{a priori}} assumption $x \in L_p(\qe)$ for some $d\leq p < \infty$ may not be necessary, however we have been unable to remove it. One reason for this difficulty is that there is no clear replacement for the use of the Poincar\'e inequality in the noncommutative situation. See Proposition \ref{dense-Schwartz}. 


    With Theorem \ref{sufficiency}, we can prove our second main result, the following trace formula:    
    \begin{thm}\label{trace formula}
    Let $x\in L_p(\qe)\cap \dot W^1_d(\qe)$ for some $d\leq p < \infty$. Then there is a constant $c_d$ depending only on the dimension $d$ such that for any continuous normalised trace $\varphi$ on ${\mathcal L}_{1,\infty}$ we have:
        \begin{equation*}
            \varphi(|\qd x|^d) = c_d\,\int_{\mathbb{S}^{d-1}} \tau_{\theta}\Bigg(\Big(\sum_{j=1}^d |\partial_j x-s_j\sum_{k=1}^d s_k\partial_k x|^2\Big)^{\frac{d}{2}}\Bigg)\,ds.
        \end{equation*}
        Here, the integral over $\mathbb{S}^{d-1}$ is taken with respect to the rotation-invariant measure $ds$ on $\mathbb{S}^{d-1}$, and $s = (s_1,\ldots,s_d)$.
    \end{thm}
    Here $\tau_{\theta}$ is the canonical trace on the algebra $L_\infty(\qe)$ (see Subsection \ref{Sec-QES}).
    Although the above integral formula is identical in appearance to \eqref{qt trace formula}, the proof involves different techniques.
    

    The next corollary is a direct application of Theorem \ref{trace formula}. The proof is the same as \cite[Corollary 1.3]{MSX2018}, so we omit the details.
    
    \begin{cor}\label{trace formula-bound}
        Let $x\in L_p(\qe)\cap \dot W^1_d(\qe)$ for some $d\leq p < \infty$. Then there are constants $c_d$ and $C_d$ depending only on $d$ such that for any continuous normalised trace $\varphi$ on ${\mathcal L}_{1,\infty}$ we have 
        $$c_d \| x\|_{\dot{W}_d^1}^d \leq  \varphi(|\qd x|^d) \leq C_d \| x\|_{\dot{W}_d^1}^d .$$
    \end{cor}  
    
    Since $\varphi$ vanishes on the trace class ${\mathcal L}_1$, Corollary \ref{trace formula-bound} immediately yields the following noncommutative version of the $p\leq d$ component of \cite[Theorem 1]{JW1982}:
    \begin{cor}\label{triviality}
        If $x \in L_{p}(\qe)$ for some $d\leq p < \infty$ and $\qd x $ has bounded extension in  $ {\mathcal L}_{q}$ for some $q \leq d$, then $x$ is a constant.
    \end{cor}
    
    As a converse to Theorem \ref{sufficiency}, we prove our third main result: the necessity of the condition $x \in \dot{W}^1_d(\qe)$ for $\qd x \in {\mathcal L}_{d,\infty}$.
    
    \begin{thm}\label{necessity}
        Suppose that $d > 2$, and let $x \in L_{d}(\qe)+L_\infty(\qe)$. If $\qd x$ has bounded extension in  ${\mathcal L}_{d,\infty}$, then $x \in \dot{ W}^1_d(\qe)$, and there is a constant $c_d>0$ depending only on $d$ such that
        $$c_d  \| x\| _{\dot{ W}^1_d} \leq  \|\qd x \|_{{\mathcal L}_{d,\infty}}.$$
        
        For $d = 2$, the same conclusion holds under the assumption that $x \in L_\infty({\mathbb R}^2_\theta)$.
    \end{thm}
    Note that in the strictly noncommutative $\det(\theta)\neq 0$ case, the assumed conditions on $x$ in Theorem \ref{necessity} are the same for $d=2$ and $d>2$, since $L_d(\qe)\subset L_\infty(\qe)$ in that case.

    It is worth noting that one may consider the commutative ($\theta = 0$) case in Theorems \ref{sufficiency}, \ref{trace formula} and \ref{necessity} and in this case the results obtained are very similar to those 
    of \cite{LMSZ2017}. The only difference being in the integrability assumptions: in \cite{LMSZ2017}, boundedness was assumed, and here we assume $p$-integrability for some $d\leq p < \infty$.
    Nonetheless the proofs we give here are independent to those of \cite{LMSZ2017}.

\subsection{Main commutator estimate}
    As a byproduct of the proof of Theorem \ref{trace formula}, we obtain a commutator estimate on quantum Euclidean spaces.
    In Section \ref{Sec-QES} we will introduce a certain smooth subalgebra $\cA$ of $L_\infty(\qe)$ (see Proposition \ref{factorisation}), and let $J_\theta = (1-\Delta_\theta)^{1/2}$ denote the quantum Bessel potential defined in Section \ref{calculus definition subsubsection}.
    \begin{thm}\label{main-commutator}
        Let $\alpha,\beta \in {\mathbb R}$, and let $x \in \cA$. Then if $ \alpha< \beta+1$:
        \begin{equation*}
            [J_\theta^{\alpha}, x]J^{-\beta}_\theta \in {\mathcal L}_{\frac{d}{\beta-\alpha+1},\infty}.
        \end{equation*}
        On the other hand if $\alpha = \beta+1$, then the operator
        \begin{equation*}
            [J^{\alpha}_\theta, x]J_\theta^{-\beta}
        \end{equation*}
        has bounded extension.        
    \end{thm}
    This estimate is to be compared with the Cwikel type estimates provided in \cite{LeSZ2017}. Using the latter estimates, one can deduce that  $ J_\theta^{\alpha}  xJ^{-\beta}_\theta \in {\mathcal L}_{\frac{d}{\beta-\alpha},\infty}$ and $x J_\theta^{\alpha-\beta} \in {\mathcal L}_{\frac{d}{\beta-\alpha},\infty}$, however showing that the difference of these two operators is in the smaller ideal ${\mathcal L}_{\frac{d}{\beta-\alpha+1},\infty}$ requires additional argument.

    If we consider the classical (commutative) setting, the result of Theorem \ref{main-commutator} would follow from a standard application of pseudodifferential operator calculus: $x$ is viewed as an order $0$ pseudo-differential operator, while $J
    _\theta^{\alpha}$ is of order $\alpha$. It follows that the commutator $[J_\theta^\al,x]$ is of order $\alpha -1 $, and thus $[J_\theta^{\alpha}, x]J^{-\beta}_\theta $ is of order $\alpha-\beta-1$. From there, a short argument can be used to show that the result of
    Theorem \ref{main-commutator} holds (an argument of precisely this nature was used in \cite[Lemma 13]{LMSZ2017}). It likely is possible to carry out a similar argument in the noncommutative setting using the
    quantum pseudodifferential operator theory of \cite{GJP2017}, however we have found the direct argument to be insightful.

  \medskip 
  
  The layout of this paper is the following. In the following section we introduce notation, terminology and required background material concerning operator ideals and analysis on quantum Euclidean spaces, and we also recount some elementary properties such as the dilation action and Cwikel type estimates. Section \ref{section-sufficiency} is devoted to the proof of Theorem \ref{sufficiency}. Section \ref{commutator-estimates} concerns our proof of Theorem \ref{main-commutator}, and
  is the most technical component of the paper. The final section, Section \ref{proofs of tf n}, completes the proofs of Theorems \ref{trace formula} and \ref{necessity}.
    
\section{Notation and preliminary results}\label{notation_section}
    We will occasionally use the notation $A\lesssim B$ to indicate that $A\leq CB$ for some $0 \leq C < \infty$, and use subscripts to indicate dependence on constants. E.g., $A\lesssim_d B$ means that $A\leq C_dB$ for a constant $C_d$ depending on $d$.

\subsection{Operators, Ideals and traces}\label{operator notation subsection}
    The following material is standard; for more details we refer the reader to \cite{Simon1979, LSZ2012}.
    Let $H$ be a complex separable Hilbert space, and let ${\mathcal B}(H)$ denote the set of all bounded operators on $H$, and let ${\mathcal K}(H)$ denote the ideal of compact operators on $H$. Given $T\in {\mathcal K}(H)$, the sequence of singular values $\mu(T) = \{\mu(k,T)\}_{k=0}^\infty$ is defined as:
    \begin{equation*}
        \mu(k,T) = \inf\{\|T-R\|\;:\;\mathrm{rank}(R) \leq k\}.
    \end{equation*}
    Equivalently, $\mu(T)$ is the sequence of eigenvalues of $|T|$ arranged in non-increasing order with multiplicities.
    
    Let $p \in (0,\infty).$ The Schatten class ${\mathcal L}_p$ is the set of operators $T$ in ${\mathcal K}(H)$ such that $\mu(T)$ is $p$-summable, i.e. in the sequence space $\ell_p$. If $p \geq 1$ then the ${\mathcal L}_p$
    norm is defined as:
    \begin{equation*}
        \|T\|_p := \|\mu(T)\|_{\ell_p} = \left(\sum_{k=0}^\infty \mu(k,T)^p\right)^{1/p}.
    \end{equation*}
    With this norm ${\mathcal L}_p$ is a Banach space, and an ideal of ${\mathcal B}(H)$.
    
    The weak Schatten class ${\mathcal L}_{p,\infty}$ is the set of operators $T$ such that $\mu(T)$ is in the weak $L_p$-space $\ell_{p,\infty}$, with quasi-norm:
    \begin{equation*}
        \|T\|_{p,\infty} = \sup_{k\geq 0}\, (k+1)^{1/p}\mu(k,T) < \infty.
    \end{equation*}
    As with the ${\mathcal L}_p$ spaces, ${\mathcal L}_{p,\infty}$ is an ideal of ${\mathcal B}(H)$. We also have the following form
    of H\"older's inequality,
    \begin{equation*}
        \|TS\|_{r,\infty} \leq c_{p,q}\|T\|_{p,\infty}\|S\|_{q,\infty}
    \end{equation*}
    where $\frac{1}{r}=\frac{1}{p}+\frac{1}{q}$, for some constant $c_{p,q}$.
    
    An operator theoretic result which will be useful is the Araki-Lieb-Thirring inequality \cite[Page 169]{Araki-1990} (see also \cite[Theorem 2]{Kosaki-alt-1992}) which
    states that if $A$ and $B$ are bounded operators and $r\geq 1$, then:
    \begin{equation*}
        |AB|^r \prec\prec_{\log} |A|^r|B|^r
    \end{equation*}
    where $\prec\prec_{\log}$ denotes logarithmic submajorisation. In particular this implies the following inequality for the ${\mathcal L}_{r,\infty}$ quasinorm, when $r\geq 1$:
    \begin{equation}\label{ALT_inequality}
        \|AB\|_{r,\infty} \leq e\||A|^r|B|^r\|_{1,\infty} \leq e\|A\|_{\infty}^{r-1}\|A|B|^r\|_{1,\infty}.
    \end{equation}
    
    Among ideals of particular interest is ${\mathcal L}_{1,\infty}$, and we are concerned with traces on this ideal. For more details, see \cite[Section 5.7]{LSZ2012} and \cite{SSUZ2015}. A functional $\varphi:{\mathcal L}_{1,\infty}\to {\mathbb C}$ is called a trace if it is unitarily invariant. That is, for all unitary operators
    $U$ and $T\in {\mathcal L}_{1,\infty}$ we have that $\varphi(U^*TU) = \varphi(T)$. It follows that for all bounded operators $B$ we have $\varphi(BT)=\varphi(TB).$
    
    An important fact about traces is that any trace $\varphi$ on ${\mathcal L}_{1,\infty}$ vanishes on ${\mathcal L}_1$ \cite[Theorem~5.7.8]{LSZ2012}. A trace $\varphi$ is called continuous if it is continuous with respect to the ${\mathcal L}_{1,\infty}$ quasi-norm. It is known that not all traces on ${\mathcal L}_{1,\infty}$ are continuous \cite[Remark~3.1(3)]{LSZ2018}. 
    Within the class of continuous traces on ${\mathcal L}_{1,\infty}$ there are the well-known Dixmier traces \cite[Chapter 6]{LSZ2012}.
    
    Finally, we say that a trace $\varphi$ on ${\mathcal L}_{1,\infty}$ is normalised if $\varphi$ takes the value $1$ on any compact positive operator with eigenvalue sequence $\{\frac{1}{n+1}\}_{n=0}^\infty$ (any two such
    operators are unitarily equivalent, and so the particular choice of operator is inessential).

\subsection{Quantum Euclidean spaces}\label{Sec-QES}
\subsubsection{Heuristic motivation}
    The original motivation for noncommutative Euclidean spaces begins with the canonical commutation relations of quantum mechanics.
    Let $\theta$ be a fixed antisymmetric $d\times d$ matrix. We consider the associative $*$-algebra with $d$ self-adjoint generators $\{x_1,\ldots,x_d\}$ satisfying the relation:
    \begin{equation}\label{ccr_heisenberg}
        [x_j,x_k] = \ri\theta_{j,k},\quad 1\leq j,k\leq d.
    \end{equation}
    These operators may be thought of as coordinates of some fictitious noncommutative $d$-dimensional space.
    
    At a purely formal level, if one defines:
    \begin{equation*}
        U(t) := \exp(\ri (t_1x_1+t_2x_2+\cdots+t_dx_d)),\quad t \in \mathbb{R}^d,
    \end{equation*}
    and formally applies the Baker-Campbell-Hausdorff formula, one is led to the following identity:
    \begin{equation}\label{ccr_weyl}
        U(t)U(s) = \exp(\frac{\ri}{2}( t,\theta s))U(t+s),\quad t,s \in \mathbb{R}^d.
    \end{equation}
    The above relation is often called the Weyl form of the canonical commutation relations, and its representation theory is summarised by the well-known Stone-von Neumann theorem: 
    provided that $\det(\theta)\neq 0$, any two $C^*$-algebras generated by a strongly continuous unitary family $\{U(t)\}_{t\in {\mathbb R}^d}$ satisfying \eqref{ccr_weyl} are $*$-isomorphic \cite[Section 5.2.2.2]{BR1997}, \cite[Theorem 14.8]{Hall2013}, \cite[Chapter 2, Theorem 3.1]{Takhtajan2008}. 
    
    After fixing a concrete Hilbert space representation of \eqref{ccr_weyl}, we will define $L_\infty(\qe)$ as the von Neumann algebra generated by $\{U(t)\}_{t \in {\mathbb R}^d}$.
    
%

\subsubsection{Formal definition and elementary properties} 
    Noncommutative Euclidean spaces admit several equivalent definitions; here we follow the approach in \cite{LeSZ2017}, where the authors define $L_\infty(\qe)$ as twisted group von Neumann algebra and then define function spaces on $\qe$ as being operator spaces associated to that algebra. We refer the reader to \cite{LeSZ2017} for more details on this approach, and give a brief introduction here. Alternative yet unitarily equivalent approaches
    to the definition of noncommutative Euclidean space may also be found in the literature, see \cite{GGISV2004}, \cite[Section 5.2.2.2]{BR1997}, \cite{GJP2017} and \cite[Chapter 14]{Hall2013}.

    Define the following family of unitary operators on $L_2({\mathbb R}^d)$:
    \beq \label{rep-on-L2}
    (U(t) \xi )(r)  =  e^{-\frac{\ri}{2}    ( t, \theta r )  } \xi (r-t), \quad \xi \in L_2({\mathbb R}^d) , \; r, t \in {\mathbb R}^d. 
    \eeq
    It is easily verified that the family $\{U(t)\}_{t\in {\mathbb R}^d}$ is strongly continuous, and satisfies the Weyl relation \eqref{ccr_weyl}. We will write $U_\theta$ when there is need to refer to the dependence on the matrix $\theta$.

    \begin{defi}\label{nc_plane_defi}
        Let $d\in {\mathbb N}$ and $\theta$ be a fixed antisymmetric real $d\times d$ matrix. The von Neumann subalgebra of ${\mathcal B}(L_2({\mathbb R}^d))$ generated by $\{U(t)\}_{t\in {\mathbb R}^d}$ given in \eqref{rep-on-L2} is called a noncommutative Euclidean space, denoted by $L_\infty (\qe)$.
    \end{defi}

    Taking $\theta =0 $, this definition states that $L_\infty ({\mathbb R}^d_0)$ is the von Neumann algebra generated by the unitary group of translations on ${\mathbb R}^d$, and this is $*$-isomorphic to $L_\infty({\mathbb R}^d)$. Therefore the algebra of essentially
    bounded functions on Euclidean space is recovered as a special case of Definition \ref{nc_plane_defi}. 
    
    \begin{remark}
        We caution the reader that the approach taken here is the ``Fourier dual" of the approach in \cite{GGISV2004}. In the commutative case, $U(t)$ is the operator on $L_2({\mathbb R}^d)$ of translation by $t \in {\mathbb R}^d$, and 
        the Fourier transform provides an isomorphism with the algebra $L_\infty({\mathbb R}^d)$ of essentially bounded functions acting by pointwise multiplication.
    \end{remark}

    The algebraic structure of $L_\infty(\qe)$ is determined by the dimension of the kernel of $\theta$. If $d=2$, then up to an orthogonal conjugation $\theta$ may be written as
    \begin{equation}\label{symplectic_matrix}
        \theta = \hbar\begin{pmatrix}
                        0 & -1 \\ 
                        1 & 0
                    \end{pmatrix} 
    \end{equation}
    for some constant $\hbar > 0$. With $\theta$ given as above, the algebra $L_\infty(\qe)$ is $*$-isomorphic to the algebra of bounded linear operators on $L_2({\mathbb R})$. A $*$-isomorphism can be given explicitly by:
    \begin{equation*}
        U(t) \mapsto \exp(\ri t_1M_x+\ri t_2\hbar\partial_x),
    \end{equation*}
    where $M_x\xi(t) = t\xi(t)$ for $\xi \in L_2({\mathbb R})$ and $\partial_x\xi = \xi'$ is the differentiation operator.

    When $d \geq 2$, we may up to orthogonal conjugation express an arbitrary $d\times d$ antisymmetric real matrix as a direct sum of a zero matrix and matrices of the form \eqref{symplectic_matrix} (see Section 6 of \cite{LeSZ2017}), ultimately leading to the following $*$-isomorphism:
    \begin{equation}\label{algebraic_structure}
        L_\infty(\qe) \cong L_\infty({\mathbb R}^{\dim(\ker(\theta))})\overline{\ot} {\mathcal B}(L_2({\mathbb R}^{\mathrm{rank}(\theta)/2}))
    \end{equation}
    where $\overline{\ot}$ is the von Neumann algebra tensor product\footnote{It is meaningful to write $\mathrm{rank}(\theta)/2$, since the rank of an antisymmetric matrix is always even}. See \cite{GV1988b} for detailed information about this
    isomorphism.
    
    In the case where $\det(\theta)\neq 0$, \eqref{algebraic_structure} reduces to:
    \begin{equation}\label{key_isomorphism}
        L_\infty(\qe) \cong {\mathcal B}(L_2({\mathbb R}^{d/2})).
    \end{equation}
    
%



\subsubsection{Weyl quantisation}
    Let $f \in L_1({\mathbb R}^d)$. We will define $U(f) \in L_\infty(\qe)$ as the operator given by the absolutely convergent Bochner integral:
    \begin{equation*}
        U(f)\xi = \int_{{\mathbb R}^d} f(t)U(t)\xi\,dt,\quad \xi \in L_2({\mathbb R}^d).
    \end{equation*}
    It should be verified first that the above integral indeed exists in the Bochner sense, and secondly that $U(f) \in L_\infty(\qe)$ as claimed.
    
    \begin{lem}
        For $f \in L_1({\mathbb R}^d)$, the integral:
        \begin{equation*}
            U(f)\xi = \int_{{\mathbb R}^d} f(t)U(t)\xi\,dt,\quad \xi \in L_2({\mathbb R}^d)
        \end{equation*}
        is absolutely convergent in the Bochner sense, and defines a bounded linear operator $U(f):L_2({\mathbb R}^d)\to L_2({\mathbb R}^d)$ such that $U(f) \in L_\infty(\qe)$.
    \end{lem}
    \begin{proof}
        Recall that $t\mapsto U(t)$ is strongly continuous. It follows that for all $\eta,\xi\in L_2({\mathbb R}^d)$, the scalar-valued function $t\mapsto f(t)\langle \eta,U(t)\xi\rangle$ is measurable. Since $L_2({\mathbb R}^d)$ is separable, the Pettis measurability theorem \cite[Theorem II.1.2]{DU1977}, \cite[Theorem 1.19]{HvNVW2016} implies that for all $\xi\in L_2({\mathbb R}^d)$ the function $t\mapsto f(t)U(t)\xi$ is measurable in the $L_2({\mathbb R}^d)$-valued Bochner sense.
        
        Since $\|f(t)U(t)\xi\|_{L_2({\mathbb R}^2)} \leq |f(t)|\|\xi\|_{L_2({\mathbb R}^d)}$ and $f \in L_1({\mathbb R}^d)$, the integrand is absolutely integrable, and this proves the claim that the integral is absolutely convergent in the Bochner sense.
        
        To see that $\xi \mapsto U(f)\xi$ is a bounded operator, one simply applies the triangle inequality for the Bochner integral to obtain
        \begin{equation*}
            \|U(f)\xi\|_{L_2({\mathbb R}^d)} \leq \|f\|_{L_1({\mathbb R}^d)}\|\xi\|_{L_2({\mathbb R}^d)}
        \end{equation*}
        so that $U(f) \in {\mathcal B}(L_2({\mathbb R}^d))$. Finally, to see that $U(f) \in L_\infty(\qe)$ we will use von Neumann's bicommutant theorem. Suppose that $X \in {\mathcal B}(L_2({\mathbb R}^d))$ is a bounded
        linear operator which commutes with every $\{U(t)\}_{t \in {\mathbb R}^d}$. Since $X$ is bounded, it can be moved under the integration sign:
        \begin{equation*}
            XU(f)\xi = X\int_{{\mathbb R}^d}f(t)U(t)\xi\,dt = \int_{{\mathbb R}^d} f(t)XU(t)\xi\,dt = U(f)X\xi.
        \end{equation*}
        Hence $X$ commutes with $U(f)$, and thus $U(f)$ commutes with every operator which commutes with every $\{U(t)\}_{t \in {\mathbb R}^d}$ so it follows that $U(f)\in L_\infty(\qe)$.
    \end{proof}   
    We will denote $U = U_\theta$ when there is a need to refer to the dependence on $\theta$.
    The map $U$ has other names and notations in the literature: for example composing $U$ with the Fourier transform determines a mapping ${\mathcal S}({\mathbb R}^d)\to {\mathcal B}(L_2({\mathbb R}^{d/2}))$ which is also known as the Weyl quantisation map \cite[Section 13.3]{Hall2013}.
    In the $\det(\theta) \neq 0$ case, the map $U$ is also essentially the same as the so-called Weyl transform \cite[Page 138]{Takhtajan2008}. In \cite{GJP2017}, the map denoted there $\lambda_\theta$ is very similar to $U$, the only difference
    being that $U(t_1e_1)U(t_2e_2)\cdots U(t_de_d)$ is used in place of $U(t)$.
    
    Assume now that $f \in {\mathcal S}({\mathbb R}^d)$. For $\xi \in {\mathcal S}({\mathbb R}^d)$, by the definition of $U(t)$ we have:
    \begin{align}
        (U(f)\xi)(s) = \int_{{\mathbb R}^d} f(t)e^{-\frac{\ri}{2}(t,\theta s)}\xi(s-t)\,dt.
    \end{align}
    Since $\xi$ is continuous, it is easy to see that $(U(f)\xi)(s)$ is continuous as a function of $s$. Evaluating $U(f)\xi(s)$ at $s = 0$ yields:
    \begin{equation*}
        (U(f)\xi)(0) = \int_{{\mathbb R}^d} f(t)\xi(-t)\,dt.
    \end{equation*}
    Hence, if $U(f) = U(g)$ for $f,g \in L_1({\mathbb R}^d)$, it follows that:
    \begin{equation*}
        \int_{{\mathbb R}^d} (f(t)-g(t))\xi(-t)\,dt = 0
    \end{equation*}
    for all $\xi \in  {\mathcal S}({\mathbb R}^d)$, and thus $f = g$ pointwise almost everywhere. It follows that $U$ is injective.
       
    The class of Schwartz functions on $\qe$ is defined as the image of ${\mathcal S}({\mathbb R}^d)$ under $U$. That is,
    \beq\label{Schwartz}
        \qcS := \{ x\in L_\infty(\qe): x=    \int _{{\mathbb R}^d}  f(s) U(s) ds ,\;\;\mbox{for some}\; f\in {\mathcal S}({\mathbb R}^d)\},\eeq
    The Schwartz space $\qcS$ is equipped with the topology induced by the isomorphism $U:{\mathcal S}({\mathbb R}^d)\to \qcS$, where ${\mathcal S}({\mathbb R}^d)$ is equipped with its canonical Fr\'echet topology. It is important to note
    that the Fr\'echet topology of $\qcS$ is finer than the $L_p(\qe)$ topology for every $1\leq p \leq \infty$. This follows, for example, from Proposition \ref{HY-ineq} below.

    It is worth emphasising that in the nondegenerate case ($\det(\theta)\neq 0$), the noncommutativity of $L_\infty(\qe)$ implies that $\qcS$ has a number of properties quite unlike the classical Schwartz space ${\mathcal S}({\mathbb R}^d)$ (for example, see Theorem \ref{smooth_projections} below). In terms of the isomorphism \eqref{key_isomorphism}, it is possible to select a specific basis such that $\qcS$ is an algebra of infinite matrices whose entries have rapid decay (\cite[Theorem 6]{GV1988} and \cite[Theorem 6.11]{PSVZ2018}). 
    While we will not need the specific details of the matrix description, we do make use of the following result, which is \cite[Lemma 2.4]{GGISV2004}. 
    \begin{thm}\label{smooth_projections}
        Assume that $\det(\theta)\neq 0$. There exists a sequence $\{p_n\}_{n\geq 0} \subset \qcS$ such that:
        \begin{enumerate}[{\rm (i)}]
            \item{} Each $p_n$ is a projection of rank $n$ (considered as an operator on $L_2({\mathbb R}^{d/2})$, via \eqref{key_isomorphism}).
            \item{} We have that $p_n\uparrow 1$, where $1$ is the identity operator in $L_\infty(\qe)$.
            \item{}\label{finite_rank_dense} $\bigcup_{n\geq 0} p_n L_\infty(\qe) p_n$ is dense in $\qcS$ in its Fr\'echet topology.
        \end{enumerate}
    \end{thm}
    The presence of smooth projections is a feature of analysis on quantum Euclidean spaces in the $\det(\theta)\neq 0$ case entirely distinct from analysis on Euclidean space. For our purposes we do not need to know the precise form of 
    the sequence $\{p_n\}_{n\geq 0}$, however a description using the map $U$ may be found in \cite[Section 2]{GGISV2004}.
    
    One feature of the Schwartz class ${\mathcal S}({\mathbb R}^d)$ is \emph{factorisability}: that is, every $f \in {\mathcal S}({\mathbb R}^d)$ can be obtained as a product $f = gh$ for $g,h \in {\mathcal S}({\mathbb R}^d)$ (see e.g. \cite{Voigt1984}).
    There is a similar result in the case of $\qcS$ when $\det(\theta)\neq 0$. For the mixed case, where $\theta \neq 0$ but $\det(\theta)=0$, the situation is less clear. We have found it more convenient to pass
    to a subalgebra of $\qcS$ for which we can verify (a very minor weakening of) the factorisation property.
    \begin{prop}\label{factorisation}
        There is a dense $*$-subalgebra $\cA\subseteq \qcS$ such that every $x\in \cA$ can be expressed as a finite linear combination of products of elements of $\cA$. That is, $x= \sum_{j=1}^n y_jz_j $ where each $y_j, z_j \in \cA$.
    \end{prop} 
    \begin{proof}
         In the case $\det(\theta)\neq 0$, this result is provided by \cite[pg. 877]{GV1988}. In the commutative ($\theta=0$) case, this is a classical result of harmonic analysis (see e.g. \cite{Voigt1984}).
         
         Performing a change of variables if necessary, we assume that $\theta$ is of the form:
         \begin{equation*}
            \theta = \begin{pmatrix} 0 & 0 \\ 0 & \theta' \end{pmatrix}
         \end{equation*}
         where $\det(\theta')\neq 0$. Let $d_1 = \dim(\ker(\theta))$. If $\det(\theta)\neq 0$, then we do not need to change variables.
         
         Let $f\in {\mathcal S}({\mathbb R}^{d_1})$ and $g \in {\mathcal S}({\mathbb R}^{d-d_1})$, and let $f\otimes g$ denote the function on ${\mathbb R}^d$ given by:
         \begin{equation*}
            (f\otimes g)(t_1,\ldots,t_d) = f(t_1,\ldots,t_{d_1})g(t_{d_1+1},\ldots,t_{d}),\quad t \in {\mathbb R}^d.
         \end{equation*}
         Then it follows readily from the definition that:
         \begin{equation*}  
            U_{\theta}(f\otimes g) = U_0(f)U_{\theta'}(g).
         \end{equation*}
         
         Every Schwartz class function $\phi \in {\mathcal S}({\mathbb R}^d)$ can be written as an infinite linear combination:
         \begin{equation*}
            \phi = \sum_{j=0}^\infty \lambda_j f_j\otimes g_j
         \end{equation*}
         where $\{f_j\}_{j=0}^\infty$ and $\{g_j\}_{j=0}^\infty$ are vanishing sequences in ${\mathcal S}({\mathbb R}^{d_1})$ and ${\mathcal S}({\mathbb R}^{d-d_1})$ respectively, and $\sum_{j=0}^\infty |\lambda_j| < \infty$ (see
         \cite[Theorem 45.1, Theorem 51.6]{Treves1967}).
         
         It follows that every $x \in \qcS$ can be written as a convergent series
         \begin{equation}\label{splitting}
            x = \sum_{j=0}^\infty \lambda_jU_0(f_j)U_{\theta'}(g_j)
         \end{equation}
         for a summable sequence $\{\lambda_j\}_{j=0}^\infty$.
         
         We will define $\cA$ as the algebraic tensor product:
         \begin{equation*}
            \cA = {\mathcal S}({\mathbb R}^{d_1})\otimes {\mathcal S}({\mathbb R}^{d-d_1}_{\theta'}).
         \end{equation*}
         That is, we define $\cA$ to be the algebra of finite linear combinations of elements of the form $U_0(f)U_{\theta'}(g)$, where $f \in {\mathcal S}({\mathbb R}^{d_1})$ and $g \in {\mathcal S}({\mathbb R}^{d-d_1})$.
         
         Then $\cA$ clearly has the desired factorisation property, as ${\mathcal S}({\mathbb R}^{d_1})$ and ${\mathcal S}({\mathbb R}^{d-d_1}_{\theta'})$ do.

    \end{proof}
    
    From now on, we fix $\cA$ to be the dense subalgebra of $\qcS$ constructed in the proof of Proposition \ref{factorisation}.
    
    For $f, g \in {\mathcal S}({\mathbb R}^d)$, we compute
    \be\begin{split}
    U(f)^* &=     \int _{{\mathbb R}^d} \overline{ f(s)}  U(s)^* ds =    \int _{{\mathbb R}^d} \overline{ f(s)} \,    U(-s) ds \\
    &=   \int _{{\mathbb R}^d}  \overline{ f(-s)} \, U(s) ds\,,
    \end{split}
    \ee
    and 
    \be\begin{split}
    U(f) U(g) & =   \int _{{\mathbb R}^d}  f(s) U(s) ds\,\cdot   \int _{{\mathbb R}^d}  g(t) U(t) dt \\
    &=    \int _{{\mathbb R}^d} \int _{{\mathbb R}^d}   f(s) g(t) \,  e^{\frac{\ri}{2} ( s , \theta t) }  U(s+t) dt ds  \\
    &=  \int _{{\mathbb R}^d} \int _{{\mathbb R}^d}   f(s-t) g(t) \,  e^{\frac{\ri}{2} ( s  , \theta t) } U(s) dt ds \\
    &=    \int _{{\mathbb R}^d}    \int _{{\mathbb R}^d}   e^{\frac{\ri}{2} ( s  , \theta t) } f(s-t) g(t)     dt\, U(s) ds .
    \end{split}
    \ee
    For this reason, we define the $\theta $-involution as 
    \beq\label{invo-twist}
    f^\theta (s)=   \overline{ f(-s)} ,
    \eeq
    and the $\theta $-convolution as 
    \beq\label{convo-twist}
    f*_\theta g(s)=    \int _{{\mathbb R}^d}   e^{\frac{\ri}{2} ( s  , \theta t) } f(s-t) g(t)     dt.
    \eeq
    Then, the above calculation shows immediately $U(f)^* = U(f^\theta )$, and 
    \begin{equation}\label{twisted_convolution}
        U(f) \, U(g) =  U(f*_\theta g).
    \end{equation}
    It is straightforward to verify that ${\mathcal S}({\mathbb R}^d)\ast_\theta {\mathcal S}({\mathbb R}^d)\subseteq {\mathcal S}({\mathbb R}^d)$.
    The $\theta$-convolution $*_\theta$ is essentially the same as the twisted convolution of \cite[Definition 1]{GV1988},
    where it was the basis for an alternative definition of ${\mathcal S}(\qe)$ (as was done in \cite{GGISV2004}). 

    
\subsubsection{Measure and integration for $\qe$}
    There is a canonical semifinite normal trace $\tau_{\theta}$ on $L_\infty(\qe)$, essentially defined so that in the isomorphism \eqref{algebraic_structure}, $\tau_{\theta}$ corresponds to integration with respect
    to the Lebesgue measure on the commutative part and is the canonical operator trace ${\mathrm{tr}}$ on the noncommutative part.  
    \begin{defi}\label{qe-integral}
        If $x \in {\mathcal S}(\qe)$ is given by $x = U(f)$ for $f \in {\mathcal S}({\mathbb R}^d)$, we define $\tau_{\theta}(x)$ as:
        \begin{equation*}
            \tau_{\theta}(x) = (2\pi)^df(0).
        \end{equation*}
    \end{defi}
    Since $U$ is injective, $\tau_\theta$ is indeed well-defined.
    The factor of $(2\pi)^d$ is inserted so that $\tau_{\theta}$ recovers the Lebesgue integral when $\theta = 0$ in the following sense:
    let $\iota$ denote the map:
    \begin{equation*}
        {\mathcal S}({\mathbb R}^d_0)\to {\mathcal S}({\mathbb R}^d)
    \end{equation*}
    given by:
    \begin{equation*}
        U(f) \mapsto \big(s\mapsto \int_{{\mathbb R}^d} f(\xi)\exp(\ri(s,\xi))\,d\xi \big).
    \end{equation*}
    Then if $\widehat{f}$ denotes the Fourier transform of $f \in {\mathcal S}({\mathbb R}^d)$, we have
    $$\iota(U(\widehat{f} )) = (2\pi)^{d/2}f.$$
    However $\int_{{\mathbb R}^d} f(s)\,ds = (2\pi)^{d/2}\widehat{f}(0)$, and so the integral of $\iota(U(\widehat{f}))$
    is $(2\pi)^d\widehat{f}(0)$.
    
    \begin{lem}
        The functional $\tau_\theta:\qcS\to {\mathbb C}$ admits an extension to a semifinite normal trace on $L_\infty(\qe)$. If $\theta = \begin{pmatrix}0 & 0 \\ 0 & \theta'\end{pmatrix}$ where $\det(\theta')\neq 0$ then
        in terms of the isomorphism \eqref{algebraic_structure} we have:
        \begin{equation*}
            \tau_{\theta} = \left(\int_{{\mathbb R}^{\dim(\ker(\theta))}}dt\right)\otimes (2\pi)^{\rank(\theta)/2}|\det(\theta')|^{1/2}{\mathrm{tr}}
        \end{equation*}
        where ${\mathrm{tr}}$ is the classical trace on ${\mathcal B}(L_2({\mathbb R}^{\dim(\ker(\theta'))/2}))$.
    \end{lem}

    
    When $\det(\theta)\neq 0$, we have:
    \begin{equation}\label{scaling_factor}
        \tau_{\theta}(U(f)) = (2\pi)^{d/2}|\det(\theta)|^{1/2}{\mathrm{tr}}(U(f)),\quad f \in {\mathcal S}({\mathbb R}^d).
    \end{equation}
    Hence in the $\det(\theta)\neq 0$ case the range of $\tau_{\theta}$ on projections consists of integer multiples of $(2\pi)^{d/2}|\det(\theta)|^{1/2}$. On the other hand, when $\det(\theta)=0$
    then the range of $\tau_\theta$ on projections is $[0,\infty]$.
    
    For $0 < p <\infty$, the space $L_p(\qe)$ is defined to be the noncommutative ${\mathcal L}_p$-space associated to the von Neumann algebra $L_\infty(\qe)$. If we define:
    $$N_p := \{    x\in L_\infty (\qe) :  \tau _\theta (|x|^p)  <\infty \}$$
    then the $L_p$ space $L_p(\qe)$ is defined as the completion of $N_p$ with the (quasi)norm $\|x\|_p =   \tau _\theta (|x|^p) ^{1/p}$. This is a norm when $p\geq 1$.
    
    When $\det(\theta)\neq 0$, since $L_\infty(\qe)$ is $*$-isomorphic to the algebra ${\mathcal B}(L_2({\mathbb R}^{d/2}))$ and $\tau_\theta$ is a rescaling of the classical trace, the spaces $L_p(\qe)$ are precisely the Schatten ${\mathcal L}_p$-classes.
    Then in the nondegenerate case we have immediately $L_p (\qe) \subset L_q(\qe)$ when $p<q$, i.e., 
    \beq\label{norm-p>q}
    c_{\theta}\|x\|_q \leq \|x\|_p,\quad x \in L_p(\qe).
    \eeq
    for some constant $c_{\theta}$.
    This is in great contrast to the classical case, where $L_p({\mathbb R}^d)$ is not contained in $L_q({\mathbb R}^d)$ for $p\neq q$.

    The preceding computations immediately yield that the mapping $(2\pi)^{-d/2}U$ extends to an isometry from $L_2({\mathbb R}^d)$ to $L_2(\qe)$ \cite[Chapter 2, Lemma 3.1]{Takhtajan2008}.
    \begin{prop}\label{trace-qe}
        Let $f\in {\mathcal S}({\mathbb R}^d)$. Then we have
        $$\|U(f)\|_2 = (2\pi)^{d/2}\|f\|_2\,.$$
    \end{prop}    
    Proposition \ref{trace-qe} permits us to extend the domain of $U$ from $L_1({\mathbb R}^d)$ to $L_1({\mathbb R}^d)+L_2({\mathbb R}^d)$.

    \begin{rk}\label{norm_dense_remark}
        It follows from Proposition \ref{trace-qe} that the Schwartz class $\qcS$ is dense in $L_2(\qe)$. Indeed, $(2\pi)^{-d/2}U$ effects an isometric
        isomorphism between $L_2({\mathbb R}^d)$ and $L_2(\qe)$, and since the classical Schwartz space ${\mathcal S}({\mathbb R}^d)$ is dense in $L_2({\mathbb R}^d)$ the 
        density of $\qcS$ in $L_2(\qe)$ follows.
    \end{rk}
    %
    %
    %


The following inequality may be thought of as the quantum Euclidean analogue of the Hausdorff-Young inequality.
\begin{prop}\label{HY-ineq}
    Let $1\leq p \leq 2$ with $\frac  1 p + \frac 1 q =1$. Then for every $f \in L_p({\mathbb R}^d)\cap L_1({\mathbb R}^d)$, we have $U(f)\in L_q(\qe)$, and 
    $$\|U(f)\|_q \leq (2\pi)^{d/2}\|f \|_p$$
    and hence $U$ has continuous extension from $L_p({\mathbb R}^d)$ to $L_q(\qe)$.
\end{prop}
\begin{proof}
    First consider the case $p=1$ and $q = \infty$. If $\xi \in L_2({\mathbb R}^d)$, the triangle inequality for the Bochner integral gives us:
    $$\|U(f)\xi\|_2  = \| \int_{{\mathbb R}^d}  f(s) U(s)\xi ds \|_2 \leq \|f\|_1\|\xi\|_2 \leq (2\pi)^{d/2}\|f\|_1\|\xi\|_2$$
    for all $f \in L_1({\mathbb R}^d)$, and therefore,
    \begin{equation*}
        \|U(f)\|_\infty\leq (2\pi)^{d/2}\|f\|_1.
    \end{equation*}

    The case $p=2$ is provided by Proposition \ref{trace-qe}:
    $$\|U(f)\|_2 = (2\pi)^{d/2}\|f\|_2.$$
    We may deduce the result for all $1\leq p\leq 2$ by using complex interpolation for the couples $(L_1({\mathbb R}^d),L_2({\mathbb R}^d))$ and $(L_\infty(\qe),L_2(\qe))$. The complex interpolation method
    for the latter couple is covered by the standard theory of interpolation of noncommutative $L_p$-spaces (see e.g. \cite{PX2003}).
\end{proof}

\section{Calculus on $\qe$}\label{calculus definition subsubsection}
    Now let us recall the differential structure on $\qe$. Let ${\mathcal D}_k$, $1\leq k \leq d $ be the multiplication operators 
    $$({\mathcal D}_k \xi)(r)= r_k  \xi(r),\, r \in {\mathbb R}^d$$
    defined on the domain ${\rm{dom\,}} {\mathcal D}_k = \{\xi \in L_2({\mathbb R}^d): \xi \in L_2({\mathbb R}^d , r_k^2 dr)\}$.
    Fixing $s\in {\mathbb R}^d$, it is easy to see that the unitary generator $U(s)$ preserves $\dom({\mathcal D}_k)$, and 
    we may compute:
    $$[{\mathcal D}_k , U(s)] =  s_k U(s),\quad \mbox{and} \quad  e^{\ri t {\mathcal D}_k } U(s) e^{-\ri t {\mathcal D}_k } = e^{\ri t s_k  }   U(s) \in L_\infty (\qe), \quad t>0.$$
    For general $x\in  L_\infty(\qe)$, if $[{\mathcal D}_k , x ] $ extends to a bounded operator on $L_2({\mathbb R}^d)$, then we can write
    $$[{\mathcal D}_k , x ]  =\lim_{t\ra 0 }     \frac{e^{\ri t {\mathcal D}_k } x e^{-\ri t {\mathcal D}_k }- x  }{\ri t }$$
    with respect to the strong operator topology, and therefore $[{\mathcal D}_k , x] \in L_\infty(\qe)$ (see \cite[Proposition 6.12]{LeSZ2017}). This operator $[{\mathcal D}_k , x]$ is then defined to be the derivative $\partial_k x$ of $x\in L_\infty (\qe)$. Evidently, $\partial_k$ anti-commutes with the adjoint operation:
    $$ \partial_k   x^* = {\mathcal D}_k x^* - x^* {\mathcal D}_k = - [{\mathcal D}_k , x ]^* = -(\partial_k   x)^* .$$
    For a multi-index $\al \in {\mathbb N}_0^d$ and $x\in L_\infty(\qe)$, if every repeated commutator $[{\mathcal D}_j^{\al_j}, [{\mathcal D}_{j+1}^{\al_{j+1}}, \cdots, [{\mathcal D}_d^{\al_d}, x]]]$, $j= 1 ,\cdots, d$ extends to a bounded operator on $L_2({\mathbb R}^d)$, then the mixed partial derivative $\partial^\al x $ is defined as 
    $$\partial^\al x  =[{\mathcal D}_1^{\al_1}, [{\mathcal D}_2^{\al_2}, \cdots, [{\mathcal D}_d^{\al_d}, x]]] .$$
    
    If $\partial^{\al}x$ is bounded for all $\al$, we say that $x$ is smooth.

    Note that the space of Schwartz functions ${\mathcal S}({\mathbb R}^d)$ is a core for every operator ${\mathcal D}_k$, $k=1, \cdots, d$, and we may show that if $x=  \int _{{\mathbb R}^d}  f(s) U(s) ds \in \qcS$, then we have 
    $$[{\mathcal D}_k , x]=  \int _{{\mathbb R}^d}  s_k f(s) U(s) ds \in \qcS.$$
    Inductively, for any $\al \in {\mathbb N}_0^d$, $\partial^\al x  \in \qcS$, and so by our definition the elements of $\qcS$ are smooth.
    
    In terms of the isomorphism $U:{\mathcal S}({\mathbb R}^d)\to \qcS$, we can compute derivatives easily:
    \begin{equation}\label{derivative_formula}
        \partial^{\alpha}U(\phi) = U(t_1^{\alpha_1}\cdots t_d^{\alpha_d}\phi(t)).
    \end{equation}

    We now define the space $\qDist$ of tempered distributions, and the associated operations.
    \begin{definition}
        Let $\qDist$ be the space of continuous linear functionals on $\qcS$, which may be called the space of quantum tempered distributions. 
        
        As in the classical case, denote the pairing of $T \in \qDist$ with $\phi$ in $\qcS$ by $(T,\phi)$, and $L_1(\qe)+L_\infty(\qe)$ is embedded into $\qDist$ by:
        \begin{equation*}
            (x,\phi) := \tau_\theta(x\phi), \quad x \in L_1(\qe)+L_\infty(\qe),\,\phi\in \qcS.
        \end{equation*}
        For a multi-index $\alpha \in {\mathbb N}_0^d$ and $T \in \qcS$, define $\partial^\alpha T$ as the distribution $(\partial^{\alpha}T,\phi) = (-1)^{|\alpha|}(T,\partial^{\alpha}\phi)$.
    \end{definition}
    It is not hard to verify that $\partial^{\alpha}$ on distributions extends $\partial^{\alpha}$ on $L_\infty(\qe)$, so there is no conflict of notation.
        
        By duality, we can extend the derivatives ${\mathcal D}_k$ to operators on $\qDist$.
    With these generalised derivatives, we are able to introduce the Sobolev spaces $W_p^m (\qe)$ associated to noncommutative Euclidean space.
    \begin{defi}
        For a positive integer $m$ and $1\leq p \leq \infty$, the space $W_p^m (\qe)$ is the space of $x\in \qDist $ such that every partial derivative of $x$ up to order $m$ is in $L_p(\qe)$, equipped with the norm
        $$\|x\|_{W_p^m}   = \sum_{|\al| \leq m }   \|\partial^\al x \|_p\,. $$
        The homogeneous Sobolev space $\dot{W}_p^m (\qe)$ consists of those $x\in \qDist$ such that every partial derivative of $x$ of order $m$ is in $L_p(\qe)$, equipped with the norm:
        $$\|x\|_{\dot W_p^m}   = \sum_{|\al| = m }   \|\partial^\al x \|_p\,. $$
    \end{defi}

    We shall now record a proof that $W_p^m(\qe)$ is a Banach space. The proof given here largely replicates well-known arguments in the classical setting, so is only included for the sake of completeness.
    \begin{prop}\label{Sob-Banach}
        Equipped with the above norm, $W_p^m (\qe)$ is a Banach space for any $1\leq p \leq \infty$ and $m\in {\mathbb N}_0$.
    \end{prop}
    \begin{proof}    
        It suffices to show that $W_p^m (\qe)$ is complete. Assume that $\{x_n\}_{n=0}^\infty \subset W_p^m (\qe)$ is a Cauchy sequence. Then for every $|\al| \leq m$, $\{\partial^\al x_n\}_n $ is a Cauchy sequence in $L_p (\qe)$, and so is convergent in the $L_p$-norm, so for each $\al$ there exists some $y_\al\in L_p(\qe)$ such that $\partial^\al x_n\ra y_\al $ in $L_p (\qe)$. In particular $  x_n\ra y_0 $ in $L_p (\qe)$. Let us show that $y_\al = \partial^\al y_0$ for all $|\al|\leq m$, and this will complete the proof.
        
        Let $\phi \in \qcS$. Then by the definition of $\partial^\al$ on $\qDist$ we have:
        \begin{equation*}
            (\partial^\al x_n,\phi) = (-1)^{|\al|}(x_n,\partial^\al \phi).
        \end{equation*} 
        Since $x_n\to y_0$ and $\partial^\al x_n\to y_\al$ in the $L_p$-sense it follows that:
        \begin{equation*}
            (-1)^{|\al|}(y_0,\partial^\al \phi) = \lim_{n\to\infty} (\partial^\al x_n,\phi) = (y_\al,\phi).
        \end{equation*}
        Thus by definition, $y_\al = \partial^\al y_0$.         
%
    \end{proof}

\medskip

The Laplacian $\Delta_\theta $ associated with $L_\infty(\qe)$ is defined on the domain $\dom(\Delta_\theta) = L_2({\mathbb R}^d, |t|^4 dt)$ by 
$$(-\Delta_\theta \xi)(t) =  |t|^2 \xi(t).$$
The gradient $\na_\theta $ associated with $L_\infty(\qe)$ is the operator
$$\na_\theta =  (-\ri {\mathcal D}_1, \cdots ,  -\ri {\mathcal D}_d ), $$
with the domain $L_2({\mathbb R}^d, t_1^2 dt)\cap \cdots \cap L_2({\mathbb R}^d, t_d^2 dt )$. 

We can see that if $t \in {\mathbb R}^d$, then $\exp((t,\na_\theta))$ is the operator on $L_2({\mathbb R}^d)$ given by:
\begin{equation*}
    (\exp((t,\na_\theta))\xi)(r) = \exp(\ri(t,r))\xi(r),\quad r \in {\mathbb R}^d,\,\xi \in L_2({\mathbb R}^d).
\end{equation*}

Strictly speaking, the operators $\Delta_\theta $ and $\na_\theta $ do not depend on the matrix $\theta$. However, we prefer to use notation with $\theta$ to emphasise that these operators are associated with $L_\infty(\qe)$.
We will have frequent need to refer to the operator $(1-\Delta_\theta)^{1/2}$, which we abbreviate as $J_\theta$,
$$J_\theta := (1- {\mathcal D}elta_\theta)^{1/2}.$$
That is, $J_{\theta}$ is the operator on $L_2({\mathbb R}^d)$ of pointwise multiplication by $(1+|t|^2)^{1/2}$, with domain
$L_2({\mathbb R}^d,(1+|t|^2)dt)$. Classically, the operator $J_{\theta}$ is called the Bessel potential.

\begin{defi}\label{Dirac}
Let $N= 2^{\lfloor d/2\rfloor }$ and $\{\gamma_j\}_{1\leq j \leq d }$ be self-adjoint $N\times N$ matrices satisfying $\gamma _j \gamma_k +\gamma _k \gamma_j= 2 \delta_{j,k}$. The Dirac operator ${\mathcal D}$ associated with $L_\infty(\qe)$ is the operator on ${\mathbb C}^N \otimes L_2({\mathbb R}^d) $ defined by 
$${\mathcal D}:=  \sum_{j=1} ^d   \gamma_j \otimes {\mathcal D}_j .  $$
\end{defi}
In noncommutative geometric terms, the Dirac operator ${\mathcal D}$ may be used to define a spectral triple for $L_\infty(\qe)$ given by 
$\big(   1\otimes W_1^\infty (\qe),   {\mathbb C}^N \otimes L_2({\mathbb R}^d)   ,  {\mathcal D} \big)$. 
We refer the reader to \cite{GGISV2004,SZ2018a} for more details. 

The main object in this note is the commutator
\beq\label{quantum-diff}
\qd x =  \ri [\sgn({\mathcal D}) ,  1 \ot x ],\quad  x\in L_\infty(\qe),
\eeq
which denotes the quantised differential on quantum Euclidean spaces.

More generally, if $x$ is not necessarily bounded we may still define $\qd x $ on the dense subspace ${\mathbb C}^N\otimes C_c^\infty  ({\mathbb R}^d)$. Suppose that $x\in L_p(\qe)$ for some $2\leq  p <\infty$. Then if $\eta\in {\mathbb C}^N\otimes C_c^\infty  ({\mathbb R}^d)$ with compact support $K$, we will have from Theorem \ref{Cwikel-type} that $(1 \otimes x) \eta= (1\otimes  x M_{\chi_K}) \eta \in L_2({\mathbb R}^d) \otimes {\mathbb C}^N$, where $\chi_K$ is the characteristic function of $K$. It follows that $\sgn({\mathcal D})(1 \otimes x) \eta  \in {\mathbb C}^N\otimes L_2({\mathbb R}^d)$. on the other hand, since $\sgn({\mathcal D}) \eta$ is still a compactly supported function in ${\mathbb C}^N\otimes L_2({\mathbb R}^d)$, using the same argument we have $(1 \otimes x)\sgn({\mathcal D}) \eta  \in {\mathbb C}^N\otimes L_2({\mathbb R}^d)$. Thus $(\qd x)  \eta $ is a well-defined element in ${\mathbb C}^N\otimes L_2({\mathbb R}^d)$.

\subsection{Dilation and translation}
    Since our quantum Euclidean spaces are generated by noncommutating operators, we cannot realise $L_\infty(\qe)$ as an algebra of functions on a space. 
    While there are no underlying points, there are still natural actions of translation by $t \in {\mathbb R}^d$
    and dilation by $\lambda \in (0,\infty)$. 
    
    Of the two, translation is simplest.
    \begin{defi}\label{translation_definition}
        Suppose that $x \in L_\infty(\qe)$. For $t \in {\mathbb R}^d$, define $T_t(x)$ as:
        \begin{equation*}
            T_t(x) = \exp((t,\nabla_\theta))x\exp(-(t,\nabla_{\theta}))
        \end{equation*}
        More generally, if $x \in \qDist$, define $T_t(f)$ as the distribution given by
        \begin{equation*}
            (T_t(f),\phi) = (f,T_{-t}\phi),\quad \phi \in \qcS.
        \end{equation*}
    \end{defi}
    
    That $T_t(f)$ is a well-defined distribution for all $f \in \qDist$ is a straightforward consequence of the observation that $T_t$ is continuous in every seminorm which defines the topology of $\qcS$.
    Moreover, it is a trivial matter to verify that $T_t$ is an isometry in every $L_p(\qe)$, for $0 < p \leq \infty$.
    
    In terms of the map $U$, we have:
    \begin{equation*}
        T_tU(\phi) = U(e^{\ri(t,\cdot)}\phi(\cdot))
    \end{equation*}
    for all $\phi\in {\mathcal S}({\mathbb R}^d)$.

    As we would expect from the classical case, $\{T_t\}_{t \in {\mathbb R}^d}$ is continuous in the $L_p$ norm for $1\leq p < \infty$.
    \begin{thm}\label{continuity_of_translation}
        If $x \in L_p(\qe)$ for $1 \leq  p < \infty$, then $T_t(x)\rightarrow x$ in the $L_p$-norm as $t\to 0$.
    \end{thm}
    \begin{proof}
        Initially consider the case when $x = U(f) \in L_2(\qe)$. It is straightforward to see that $T_t(U(f)) = U(\exp(\ri(t,\cdot)f(\cdot)))$ for all $f \in L_2({\mathbb R}^d)$, and 
        using Proposition \ref{trace-qe} and the dominated convergence theorem:
        \begin{equation*}
            \|T_t(U(f))-U(f)\|_2 = (2\pi)^{d/2}\|e^{\ri(t,\cdot)}f(\cdot)-f(\cdot)\|_2 \rightarrow 0
        \end{equation*}
        as $t\to 0$. 
%
%
        
        Suppose that $2< p < \infty$ and $x \in L_2(\qe)\cap L_\infty(\qe)$. Using the H\"older inequality, it follows that:
        \begin{align*}
            \lim_{t\to 0}\|T_t(x)-x\|_{p} &\leq \lim_{t\to 0} \|T_t(x)-x\|_\infty^{1-\frac{2}{p}}\|T_t(x)-x\|_2^{\frac{2}{p}}\\
                                          &\leq (2\|x\|_\infty)^{1-\frac{2}{p}}\lim_{t\to \infty}\|T_t(x)-x\|_2^{\frac{2}{p}}\\
                                          &= 0.
        \end{align*}
        
        We can extend from $x \in L_2(\qe)\cap L_\infty(\qe)$ to all $x \in L_p(\qe)$ by using the norm-density of $L_2(\qe)\cap L_\infty(\qe)$ in $L_p(\qe)$. Namely,
        let $\varepsilon>0$ and select $y \in L_2(\qe)\cap L_\infty(\qe)$ such that $\|x-y\|_p<\varepsilon$. Then:
        \begin{align*}
            \lim_{t\to0} \|T_t(x)-x\|_p &\leq \lim_{t\to 0} \|T_t(x-y)\|_p+\lim_{t\to 0}\|T_ty-y\|_p+\|y-x\|_p\\
                           &\leq 2\varepsilon+\lim_{t\to 0}\|T_ty-y\|_p\\
                           &= 2\varepsilon.
        \end{align*}
        Hence, $T_tx\to x$ in the $L_p$ norm.
        
        On the other hand, if $1\leq p < 2$, consider $x \in L_2(\qe)\cap L_{2p/(4-p)}(\qe)$, then H\"older's inequality and the fact that $T_t$ is an isometry in every $L_p(\qe)$ implies that:
        \begin{align*}
             \|T_t(x)-x\|_p &\leq \|T_t(x)-x\|_2^{1/2}. \|T_t(x)-x\|_{2p/(4-p)} ^{1/2}\\
                            &\lesssim_p \|T_t(x)-x\|_2^{1/2}\|x\|_{2p/(4-p)} ^{1/2}.
%
        \end{align*}
        Thus $\lim_{t\to 0}\|T_t(x)-x\|_p = 0$ for $x \in L_2(\qe)\cap L_{2p/(4-p)}(\qe)$, and this may be extended to all $x \in L_p(\qe)$ by a density argument similar to the $p>2$ case.

  \end{proof}
  Theorem \ref{continuity_of_translation} only discusses the cases $1 \leq p < \infty$ since we are not aware of any characterisation of the subspace of $x \in L_\infty(\qe)$ such that $\lim_{t\to 0} \|T_tx-x\|_\infty = 0$.
  In the classical case, this corresponds to the space of bounded uniformly continuous functions. Using Theorem \ref{HY-ineq}, it is possible to prove that $\lim_{t\to 0}\|T_tx-x\|_\infty=0$
  for all $x \in \qcS$, and for all $x$ in the closure of $\qcS$ in $L_\infty(\qe)$.

    We now describe the ``dilation" action of ${\mathbb R}^+$ on a quantum Euclidean space. A peculiarity of the noncommutative situation is
    that the natural dilation semigroup does not define an automorphism of $L_\infty(\qe)$ to itself, but instead the value of $\theta$ varies.

The heuristic motivation for the dilation mapping is as follows. Recall that we consider $\qe$ as being generated by elements $\{x_1,\ldots,x_d\}$ satisfying
the commutation relation 
$$[x_j,x_k]=\ri\theta_{j,k}.$$ 

However this relation is not invariant under rescaling. That is, if we let $\lambda > 0$ then the family $\{\lambda x_1,\ldots,\lambda x_d\}$ satisfies the relation:
\begin{equation*}
    [\lambda x_j,\lambda x_k] = \ri \lambda^2\theta_{j,k}.
\end{equation*}
It therefore becomes clear that if we wish to define a ``dilation by $\lambda$" map on $\qe$, we should instead consider dilation as mapping between two different noncommutative spaces. That is,
from $\qe$ to ${\mathbb R}^d_{\lambda^2\theta}$.

The following rigorous definition of the ``dilation by $\la$" map follows \cite{GJP2017}. Given $\lambda >0$, define the map $\Psi_\la$ from $L_\infty(\qe)$ to $L_\infty({\mathbb R}_{\la ^2 \theta }^d)$ as 
\beq\label{def-dilation}
\Psi_\la:  U_\theta (s) \mapsto U_{\la^2 \theta } (\frac s \la ).\eeq 
Recall that we include a subscript $\theta$ (or $\la^2\theta$) to indicate the dependence on the matrix.

Denote by $\sigma_\la $ the usual $L_2$-norm preserving dilation on Euclidean space:
$$\sigma_\la \xi(t)    =  \la^{d/ 2 } \xi(\la  t ), \quad \xi \in  L_2({\mathbb R}^d). $$
We have $ \sigma _\la ^* =  \sigma_{\la^{-1}}$. 
It is standard to verify that 
\beq\label{dilation}
U_\theta (s) =  \sigma^*_{\la}   \,  U_{\la^2 \theta } (\frac s \la )\,  \sigma_\la .\eeq
Moreover, by \eqref{dilation}, it is evident that for every $\la>0$, $\Psi_\la$ is a $*$-isomorphism from $L_\infty(\qe)$ to $L_\infty({\mathbb R}_{\la ^2 \theta }^d)$.

%
%
%
    The following proposition shows how the dilation $\Psi_\la $ affects the $L_p$ norms for quantum Euclidean spaces.
    \begin{prop}\label{dilation-Ld}
        Let $\la>0 $ and $x\in L_p (\qe)$, and denote $\xi= \la^2\theta $. Then for all $2 \leq p < \infty$, we have:
        \begin{equation*}
            \|\Psi_\la x\|_{L_p({\mathbb R}^d_\xi)} \leq \lambda^{d/p}\|x\|_{L_p(\qe)}
        \end{equation*}
        and $\Psi_{\la}$ is an isometry from $L_\infty(\qe)$ to $L_\infty({\mathbb R}^d_\xi)$.
        
        If in addition $x \in W^{1}_p(\qe)$, then:
        \beq\label{dilation-Wd-eq}
        \|\partial_j\Psi_\la (x)  \|_{L_p({\mathbb R}^d_{\xi})}  \leq  \lambda^{d/p-1}\|\partial_jx\|_{L_p({\mathbb R}^d_{ \theta})} \,,\quad  j=1,\cdots , d. \eeq
    \end{prop}
\begin{proof}
    As was already mentioned, $\Psi_{\la}$ is a $*$-isomorphism between $L_\infty(\qe)$ and $L_\infty({\mathbb R}^d_\xi)$, and since a $*$-isomorphism of $C^*$-algebras is an isometry, it follows
    immediately that $\Psi_\la:L_\infty(\qe)\to L_\infty({\mathbb R}^d_\xi)$ is an isometry.
    
    For $p=2$, recall from Proposition \ref{trace-qe} that the mapping $(2\pi)^{-d/2}U_\theta$ (resp. $(2\pi)^{-d/2}U_\xi$) defines an isometry from $L_2(\qe)$ (resp. $L_2({\mathbb R}^d_\xi)$) to $L_2({\mathbb R}^d)$. Denoting $d_{\lambda}$ for the map $d_\lambda f(t) = f(t/\lambda)$, we have:
    \begin{equation*}
        \Psi_{\la} \circ U_\theta = U_\xi\circ d_\la, \quad \lambda > 0.
    \end{equation*}
    Hence $\Psi_\la$ has the same norm betweeen $L_2(\qe)$ and $L_2({\mathbb R}^d_\xi)$ as $d_\la$ does on $L_2({\mathbb R}^d)$. This is easily computed to be $\la^{d/2}$. 
    
    Finally, the result for $2 < p < \infty$ follows from complex interpolation of the couples $(L_2(\qe),L_\infty(\qe))$ and $(L_2({\mathbb R}^d_\xi),L_\infty({\mathbb R}^d_\xi))$.
    
    We recall that the complex interpolation space $(L_2(\qe),L_\infty(\qe))_{\eta}$ is $L_{2/\eta}(\qe)$, where $\eta \in (0,1)$, and that we have:
    \begin{equation*}
        \|\Psi_\la\|_{L_{2/\eta}\to L_{2/\eta}} \leq \|\Psi_{\la}\|_{L_2\to L_2}^\eta \|\Psi_{\la}\|_{L_\infty\to L_\infty}^{1-\eta} \leq \la^{d\eta/2}.
    \end{equation*}
    Taking $\eta = \frac{2}{p}$ yields the desired norm bound.
    
    The second claim follows from the easily-verified identity:
    \begin{equation*}
        \partial_j(\Psi_{\la}(x)) = \lambda^{-1}\Psi_{\la}\partial_j(x).
    \end{equation*}    
\end{proof}

%
        

\subsection{Approximation by smooth functions for $\qe$}\label{approximation_subsection}
    For this section, we fix $\psi \in {\mathcal S}({\mathbb R}^d)$ such that $\int_{{\mathbb R}^d} \psi(s)\,ds = 1$.
    We do not assume that $\psi$ is necessarily compactly supported or positive, since it will be convenient to have some freedom in choosing $\psi$. 
    For $\varepsilon > 0$, define:
    \begin{equation}\label{psi-dilation}
        \psi_{\varepsilon}(t) = \varepsilon^{-d}\psi(\frac{t}{\varepsilon}).
    \end{equation}
    By construction, $\int_{{\mathbb R}^d}\psi_{\varepsilon}(t)\,dt = 1$. Moreover since $\psi$ in particular has rapid decay at infinity, the $L_1$-norm of $\psi_{\varepsilon}$ is primarily concentrated in the ball of radius $\varepsilon^{1/2}$ around zero. That is, for each $N\geq 1$, there exists a constant $C_N$ depending on $\psi$ such that:
    \begin{equation}\label{approximate_concentration}
        \int_{|t|>\varepsilon^{1/2}} |\psi_{\varepsilon}(t)|\,dt \leq C_N \varepsilon^{N}.
    \end{equation}
    
    \begin{thm}\label{spatial_approximation}
        Let $1\leq p <\infty$. For all $x \in L_p(\qe)$, we have that $U(\psi_{\varepsilon})x \to x$ in the $L_p(\qe)$ norm as $\varepsilon\to 0$.
    \end{thm}
    \begin{proof}

        Let us first prove the result for $p=2$ and $x \in \qcS$. Thanks to Proposition \ref{trace-qe} and \eqref{twisted_convolution}, it suffices to show that for all $f \in {\mathcal S}({\mathbb R}^d)$:
        \begin{equation*}
            \psi_{\varepsilon}\ast_\theta f\rightarrow f
        \end{equation*}
        in the norm of $L_2({\mathbb R}^d)$, where $\ast_\theta$ is the deformed convolution \eqref{convo-twist}.
        
        By definition \eqref{convo-twist}, we have that:
        \begin{equation}\label{mass_one}
            \psi_{\varepsilon}\ast_{\theta} f (t) = \int_{{\mathbb R}^d} e^{-\frac{\ri}{2}(t,\theta s)}\psi_{\varepsilon}(s)f(t-s)\,ds,\quad t \in {\mathbb R}^d.
        \end{equation}

        Since by definition $\int_{{\mathbb R}^d} \psi_{\varepsilon}(s)\,ds = 1$, we have:
        \begin{equation*}
             \psi_{\varepsilon}\ast_{\theta} f (t)-f(t) = \int_{{\mathbb R}^d} e^{-\frac{\ri}{2}(t,\theta s)}\psi_{\varepsilon}(s)f(t-s)-\psi_{\varepsilon}(s)f(t)\,ds
        \end{equation*}
        for all $t \in {\mathbb R}^d$. Hence,
        \begin{equation*}
             \psi_{\varepsilon}\ast_{\theta} f  (t)-f(t) = \int_{{\mathbb R}^d} e^{-\frac{\ri}{2}(t,\theta s)}\psi_{\varepsilon}(s)(f(t-s)-e^{\frac{\ri}{2}(t,\theta s)}f(t))\,ds.
        \end{equation*}
        Split the integral into the set $|s|\leq \varepsilon^{1/2}$ and $|s| > \varepsilon^{1/2}$. Let $N\geq 1$. Using \eqref{approximate_concentration} there is a constant $C_N$ such that
        \begin{align*}
            |\psi_{\varepsilon}\ast_{\theta} f(t)-f(t)| &\leq \int_{|s|\leq \varepsilon^{1/2}} |\psi_{\varepsilon}(s)| \,|f(t-s)-e^{\frac{\ri}{2}(t,\theta s)}f(t)|\,ds\\
                                                     &\quad   + \int_{|s|> \varepsilon^{1/2}} |\psi_{\varepsilon}(s)|\, |f(t-s)-e^{\frac{\ri}{2}(t,\theta s)}f(t)|\,ds\\
                                                     &\leq  \|\psi\|_1\sup_{|s|\leq \varepsilon^{1/2}} |f(t-s)-e^{\frac{\ri}{2}(t,\theta s)}f(t)|\\
                                                     &\quad   + C_N\varepsilon^N\|f\|_\infty.
        \end{align*}
        Since $f$ is in Schwartz class (and in particular uniformly continuous and bounded), it follows that
        \begin{equation}\label{pointwise_convergence}
            \lim_{\varepsilon\to 0} |\psi_{\varepsilon}\ast_{\theta} f(t)-f(t)| = 0
        \end{equation}
        uniformly for $t \in {\mathbb R}^d$.
        
        Returning to \eqref{mass_one}, we can use the triangle inequality to deduce that:
        \begin{equation*}
            |\psi_{\varepsilon}\ast_{\theta} f(t)| \leq \int_{{\mathbb R}^d} |\psi_{\varepsilon}(s)||f(t-s)|\,ds.
        \end{equation*}
        That is, $|\psi_{\varepsilon}\ast_{\theta} f| \leq |\psi_{\varepsilon}|\ast |f|$. Using Young's convolution inequality, this implies that:
        \begin{equation*}
            \|\psi_{\varepsilon}\ast_{\theta}f\|_2\leq \|\psi\|_1\|f\|_2.
        \end{equation*}
         Thus $\psi_{\varepsilon}\ast_{\theta} f -f\in L_2({\mathbb R}^d)$. Let $\delta>0$ and select a compact set $K\subset {\mathbb R}^d$ such that $\|(\psi_{\varepsilon}\ast_{\theta}f-f)\chi_{{\mathbb R}^d\setminus K}\|_2< \delta$.
        Since we have uniform pointwise convergence \eqref{pointwise_convergence}, it follows that:
        \begin{equation*}
            \lim_{\varepsilon\to 0} \|\psi_{\varepsilon}\ast_{\theta}f-f\|_2 \leq \lim_{\varepsilon\to 0} \|(\psi_{\varepsilon}\ast_{\theta}f-f)\chi_K\|_2 + \delta = \delta.
        \end{equation*}
        However $\delta>0$ is arbitrary and therefore:
        \begin{equation*}
            \lim_{\varepsilon\to 0} \|\psi_{\varepsilon}\ast_{\theta} f-f\|_2 = 0.
        \end{equation*}
        This completes the proof for $x \in \qcS$.        
        
        Now we may complete the proof for $p=2$ by using the density of $\qcS$ in $L_2(\qe)$ (Remark \ref{norm_dense_remark}). Suppose that $x \in L_2(\qe)$ and $y \in \qcS$
        is chosen such that $\|y-x\|_2 < \varepsilon$. Note that we have $\|U(\psi_{\varepsilon})\|_\infty \leq \|\psi_\varepsilon\|_1 = \|\psi_1\|_1 < \infty$. Thus,
        \begin{align*}
            \|U(\psi_{\varepsilon})x-x\|_2 &\leq \|U(\psi_{\varepsilon})(x-y)\|_2 + \|U(\psi_{\varepsilon})y-y\|_2 + \|x-y\|_2\\
                                           &\leq (\|U(\psi_{\varepsilon})\|_\infty+1)\varepsilon + \|U(\psi_{\varepsilon})y-y\|_2\\
                                           &\rightarrow 0
        \end{align*}
        as $\varepsilon\to 0$. This completes the proof for $p = 2$.
        
        Now we may complete the proof for $p \neq 2$ by following an identical argument to the proof of Theorem \ref{continuity_of_translation}.
    \end{proof} 
    The $p=2$ component of Theorem \ref{spatial_approximation} may be equivalently, stated as $U(\psi_{\varepsilon}) \rightarrow 1$ in the strong operator topology of $L_\infty(\qe)$ in its representation on $L_2(\qe)$.
    
    There is another way in which we can approximate an element $x \in L_p(\qe)$ using $\psi_{\varepsilon}$. This uses the notion of convolution:
    \begin{defi}
        Let $x \in L_p(\qe)$ for $1\leq p <\infty$. For $\psi\in L_1({\mathbb R}^d)$ define:
        \begin{equation*}
            \psi \ast x := \int_{{\mathbb R}^d} \psi (s)T_{-s}(x)\,ds
        \end{equation*}
        as an absolutely convergent Bochner integral.
    \end{defi}
    Some remarks are in order: First, Theorem \ref{continuity_of_translation} implies that the mapping $s\mapsto T_{-s}(x)$ is continuous from ${\mathbb R}^d$
    to $L_p(\qe)$ with its norm topology, so for each $y \in L_q(\qe)$, for $\frac{1}{p}+\frac{1}{q}=1$, we have that $s\mapsto \tau_{\theta}(yT_{-s}(x))$ is continuous
    and so the integrand is weakly measurable. Since $L_p(\qe)$ is separable for $p < \infty$, the Pettis measurability
    theorem ensures the Bochner measurability of the integrand. The triangle inequality then implies:
    \begin{equation}\label{young_inequality}
        \|\psi  \ast x\|_p \leq  \|\psi\|_1 \|x\|_p.
    \end{equation}
    If we instead consider $p=\infty$, there may be issues with Bochner measurability of the integrand, however we will not need to be concerned with that case.
    
    Another fact about convolution worth noting is that if $x \in L_2(\qe)$ is given by $x = U(f)$ for $f \in L_2({\mathbb R}^d)$, then:
    \begin{equation}\label{convolution_theorem}
        \psi \ast U(f) = U(\widehat{\psi} f)
    \end{equation}
    where $\widehat{\psi} $ is the Fourier transform of $\psi$.
    
    Note at this stage that convolution with $\psi $ commutes with each $\partial_j$.
    
    \begin{thm}\label{frequency_approximation}
        Let $x \in L_p(\qe)$ for $1\leq p < \infty$, and let $\psi$ and $\psi_\varepsilon$ be as in \eqref{psi-dilation}. Then:
        \begin{equation*}
            \psi_\varepsilon \ast x \rightarrow x
        \end{equation*}
        in the $L_p$-norm, as $\varepsilon\to 0$.
    \end{thm}
    \begin{proof}
        By definition, and the fact that $\int_{{\mathbb R}^d}\psi_{\varepsilon}(s)\,ds = 1$, we have:
        \begin{equation*}
            \psi_{\varepsilon}\ast x-x = \int_{{\mathbb R}^d} \psi_{\varepsilon}(s)(T_{-s}(x)-x)\,ds.
        \end{equation*}
        Using \eqref{approximate_concentration}, let $N\geq 1$ and split the integral into regions $|s|\leq \varepsilon^{1/2}$ and $|s|> \varepsilon^{1/2}$ to obtain:
        \begin{equation*}
            \|\psi_{\varepsilon}\ast x-x\|_p \leq \|\psi\|_1\sup_{|s|< \varepsilon^{1/2}}\|T_s(x)-x\|_p+2C_N\varepsilon^N\|x\|_p \,.
        \end{equation*}
        The result now follows from Theorem \ref{continuity_of_translation}.
    \end{proof}
    
    We can now combine Theorems \ref{spatial_approximation} and \ref{frequency_approximation} to simultaneously approximate $x\in L_p(\qe)$ with convolution and left multiplication by mollifying functions.
    The proof of the following is a straightforward consequence of the fact that $\|U(\phi_{\varepsilon})\|_\infty$ is uniformly bounded in $\varepsilon$, and also the inequality \eqref{young_inequality}.
    \begin{cor}\label{space_frequency_approximation}
        Let $x \in L_p(\qe)$, and suppose that we have a family $\{x_{\varepsilon}\}_{\varepsilon>0} \subseteq L_p(\qe)$ such that $x_{\varepsilon}\to x$ in the $L_p$ sense as $\varepsilon\to 0$. Then:
        \begin{equation*}
            U(\psi_{\varepsilon})x_{\varepsilon}\to x,\quad \psi_{\varepsilon}\ast x_{\varepsilon}\to x
        \end{equation*}
        in $L_p(\qe)$, as $\varepsilon\to 0$.


    \end{cor}
    \begin{proof}
        Both estimates follow from the fact that the $L_1$-norm of $\psi_{\varepsilon}$ is uniformly bounded in $\varepsilon$. Indeed, we have:
        \begin{align*}
            \|U(\psi_{\varepsilon})x_{\varepsilon}-x\|_p &\leq \|U(\psi_{\varepsilon})\|_\infty\|x_{\varepsilon}-x\|_{p} + \|U(\psi_{\varepsilon})x-x\|_p\\
                                                         &\leq \|\psi\|_1\|x_{\varepsilon}-x\|_p + \|U(\psi_{\varepsilon})-x\|_p
        \end{align*}
        which vanishes as $\varepsilon\to 0$ thanks to Lemma \ref{spatial_approximation}. Similarly \eqref{young_inequality} implies:
        \begin{equation*}
            \|\psi_{\varepsilon}\ast x_{\varepsilon}-x\|_p\leq \|\psi_{\varepsilon}\|_1\|x_{\varepsilon}-x\|_p + \|\psi_{\varepsilon}\ast x-x\|_p
        \end{equation*}
        which again vanishes as $\varepsilon\to 0$, due to Lemma \ref{frequency_approximation}.
    \end{proof}
    
    Corollary \ref{space_frequency_approximation} suffices to show that, for example, $\psi_{\varepsilon}\ast (U(\phi_{\varepsilon})x)\to x$ as $\varepsilon\to 0$
    in the $L_p$ sense, where $\phi_{\varepsilon}\in {\mathcal S}({\mathbb R}^d)$ is defined similarly to $\psi_{\varepsilon}$.
    
    
   It is shown in \cite{GJP2017} that $ \qcS$ is weak-$*$ dense in $L_\infty(\qe)$, and norm dense in $L_p(\qe)$ for $1\leq p <\infty$. Corollary \ref{space_frequency_approximation} combined with the following lemma gives us a specific sequence which approximates an arbitrary $x\in L_p(\qe)$ by a sequence in $\qcS$. 
   
   
    \begin{lem}\label{Schwartz_approximation_is_possible}
         There exist choices of $\psi$, $\phi$ and $\chi$ in ${\mathcal S}({\mathbb R}^d)$ with integral equal to $1$ such that for all $x \in L_p(\qe)$ ($2\leq p \leq \infty$) and $\varepsilon>0$ the element $\psi_{\varepsilon}\ast \big(U(\phi_{\varepsilon})U(\chi_{\varepsilon})x \big)$ is in the Schwartz space $\qcS$.
    \end{lem}
    \begin{proof}
        Let us first prove that we can select $\chi \in {\mathcal S}({\mathbb R}^d)$ such that $U(\chi_{\varepsilon})x \in L_2(\qe)$ for all $x \in L_p(\qe)$. 
        
        We refer to the isomorphism \eqref{algebraic_structure}. By a change of variables if necessary, we assume that $\theta$ 
        is of the form:
        \begin{equation*}
            \theta = \begin{pmatrix} 0_{d_1} & 0 \\ 0 & \widetilde{\theta}\end{pmatrix},
        \end{equation*}
        where $d_1 = \dim(\ker(\theta))$ and $\det(\widetilde{\theta}) \neq 0$. Let $H = L_2({\mathbb R}^{\mathrm{rank}(\theta)/2})$, then $L_p(\qe)$
        can be identified with the Bochner space:
        \begin{equation*}
            L_p(\qe) = L_p({\mathbb R}^{d_1}; {\mathcal L}_p(H)).
        \end{equation*}
        (see, e.g. \cite[Chapter 3]{Pisier1998}).
        
            Since $\widetilde{\theta}$ has trivial kernel, the corresponding Schwartz space ${\mathcal S}({\mathbb R}^{d-d_1}_{\widetilde{\theta}})$ has a dense subspace of finite rank elements
            as in Theorem \ref{smooth_projections}. Select $n > 0$ and $z \in L_\infty({\mathbb R}^{d-d_1}_{\widetilde{\theta}})$ such that $p_nzp_n$ (which is in ${\mathcal S}({\mathbb R}^{d-d_1}_{\widetilde{\theta}})$) is given by $U_{\widetilde{\theta}}(\zeta)$, where $\zeta \in {\mathcal S}({\mathbb R}^{d-d_1})$. We may choose $p_nzp_n$
            such that $\zeta$ has nonzero integral, thanks to part \eqref{finite_rank_dense} of Theorem \ref{smooth_projections}.
            
            Now select $\eta \in C^\infty_c({\mathbb R}^{d_1})$ with $\eta(0) \neq 0$. We select $\chi \in {\mathcal S}({\mathbb R}^d)$ such that:
            \begin{equation*}
                U_\theta(\chi) = M_{\eta}\otimes U_{\widetilde{\theta}}(\zeta) = M_\eta\otimes p_nzp_n.
            \end{equation*}

        Since $\eta$ and $p_nzp_n$
         are in the Schwartz spaces for ${\mathbb R}^{d_1}$ and ${\mathbb R}^{d-d_1}_{\widetilde{\theta}}$ respectively, we may indeed choose $\chi$ such that $U_{\theta}(\chi) = M_{\eta}\otimes p_nzp_n$. We will have $\int_{{\mathbb R}^d}\chi(t)\,dt = \eta(0)\int_{{\mathbb R}^d}\zeta(t)\,dt$, which by construction is not zero.
%
        Thus, rescaling $\eta$ if necessary, we may assume that $\int_{{\mathbb R}^d}\chi(t)\,dt = 1$. 
        
        Then, if $x \in L_p({\mathbb R}^{d_1},{\mathcal L}_p(H))$, it follows that $U(\chi)x$ is compactly supported on ${\mathbb R}^{d_1}$, and takes values in $P{\mathcal L}_p(H)$. Therefore,
        \begin{equation*}
            U(\chi)x \in L_2({\mathbb R}^{d_1}; {\mathcal L}_2(H)) = L_2(\qe).
        \end{equation*}
        One can then deduce that $U(\chi_{\varepsilon})x \in L_2(\qe)$ via the dilation maps $\Psi_{\varepsilon}$ and $\Psi_{\varepsilon^{-1}}$, since we have:
        \begin{equation*}
            U_\theta(\chi_{\varepsilon}) = \varepsilon^{-d}\Psi_{\varepsilon^{-1}}U_{\varepsilon^2\theta}(\chi)\Psi_{\varepsilon}.
        \end{equation*}
%
        Since $U(\chi_{\varepsilon})x \in L_2(\qe)$, from Theorem \ref{trace-qe} there exists $f \in L_2({\mathbb R}^d)$ such that $U(\chi_{\varepsilon})x = U(f)$. Using \eqref{convolution_theorem} we have:
        \begin{equation*}
            \psi_{\varepsilon}\ast \big(U(\phi_{\varepsilon})U(f)\big) = U\big(\widehat{\psi}_{\varepsilon}(\phi_{\varepsilon}\ast_{\theta} f)\big).
        \end{equation*}
        It is easily shown that $\phi_{\varepsilon}\ast_{\theta} f$ is smooth, and we may select $\psi$ such that $\widehat{\psi}_{\varepsilon}$ is compactly supported,
        and thus $\widehat{\psi}_{\varepsilon}(\phi_{\varepsilon}\ast_{\theta} f)$ is smooth and compactly supported, and thus by definition it follows that $U(\widehat{\psi}_{\varepsilon}(\phi_{\varepsilon}\ast_{\theta} f)) = \psi_{\varepsilon}\ast (U(\phi_{\varepsilon})U(f))$
        is in $\qcS$. That is,
        \begin{equation*}
            \psi_{\varepsilon}\ast (U(\phi_{\varepsilon})U(\chi_{\varepsilon})x) \in \qcS.
        \end{equation*}        
    \end{proof}
    Note that in the proof of Lemma \ref{Schwartz_approximation_is_possible}, the function $\zeta$ was chosen such that $U_{\widetilde{\theta}}(\zeta)$ satisfies certain conditions. It is for this reason that we avoided making the assumption that the function $\psi$ appearing in the preceding lemmas is positive or compactly supported; the proof of Lemma \ref{Schwartz_approximation_is_possible} is simplified if we do not need to prove that $\zeta$ has those properties.
    
\subsection{Density of $\qcS$ and $\cA$ in Sobolev spaces}
    We now use the machinery of the previous subsection to prove that $\cA$ (and by extension, $\qcS$) is dense in $W^{m}_{p}(\qe)$ for an appropriate range of indices $(m,p)$.
 Proving the density of $\cA$ in the homogeneous Sobolev space $\dot{W}^{m}_{p}(\qe)$, however, presents difficulties and we have been unable to achieve this for the full range of indices $(m,p)$.
    
    As in Subsection \ref{approximation_subsection}, select a Schwartz class function $\psi$ with $\int_{{\mathbb R}^d} \psi(t)\,dt = 1$, and denote $\psi_{\varepsilon}(t) = \varepsilon^{-d}\psi(t/\varepsilon)$. We note one further property of $U(\psi_{\varepsilon})$:
    \begin{lem}\label{cancellation_lemma}
        Let $1\leq j\leq d$. Then for all $2\leq p \leq \infty$, we have:
        \begin{equation*}
            \|\partial_j U(\psi_{\varepsilon})\|_p \leq \varepsilon^{1-\frac{d}{p}}\|\psi_1\|_q.
        \end{equation*}
        where $q$ satisfies $\frac{1}{p}+\frac{1}{q}=1$.
    \end{lem}
    \begin{proof}
        Recall (from \eqref{derivative_formula}) that:
        \begin{equation*}
            \partial_j U(\psi_{\varepsilon}) = U(t_j\psi_{\varepsilon}(t))
        \end{equation*}
        so that we may apply Proposition \ref{HY-ineq} to bound $\|\partial_j U(\psi_{\varepsilon})\|_p$ by:
        \begin{equation*}
            \left(\int_{{\mathbb R}^d} t_j^q\varepsilon^{-dq}|\psi(\frac{t}{\varepsilon})|^qdt\right)^{1/q}
        \end{equation*}
        where $q$ is H\"older conjugate to $p$.
        
        Applying the change of variable $s = \frac{t}{\varepsilon}$, we get the norm bound:
        \begin{equation*}
            \|\partial_j U(\psi_{\varepsilon})\|_p \leq \varepsilon^{1-d+\frac{d}{q}}\|\psi_1\|_q. \qedhere
        \end{equation*} 
    \end{proof}

%
%
%
%
    Lemma \ref{cancellation_lemma} allows us to prove the density of $\cA$ in the Sobolev spaces associated to $\qe$. We achieve this by first using Lemma \ref{Schwartz_approximation_is_possible} to prove that $\qcS$ is dense in $W^{m,p}(\qe)$.    
    \begin{prop}
        Let $m\geq 0$ and $1\leq p < \infty$, and $x \in W^{m}_p(\qe)$, and let $\{\phi_{\varepsilon}\}_{\varepsilon>0}$, $\{\psi_{\varepsilon}\}_{\varepsilon>0}$ and $\{\chi_{\varepsilon}\}_{\varepsilon>0}$ be chosen as in Subsection \ref{approximation_subsection}. Then
        \begin{equation*}
            \lim_{\varepsilon\to 0} \|\psi_{\varepsilon}\ast \big(U(\phi_{\varepsilon})U(\chi_{\varepsilon})x\big)-x\|_{W^{m}_p} = 0.
        \end{equation*}
        In particular, $\qcS$ is norm-dense in $W_p^m(\qe)$.
     \end{prop}
     \begin{proof}
        For $m=0$, this is already implied by Corollary \ref{space_frequency_approximation}.
        
        For $m=1$, we use the Leibniz rule, recalling that differentiation commutes with convolution:
        \begin{align*}
            \partial_j\Big(\psi_{\varepsilon}\ast \big(U(\phi_{\varepsilon})U(\chi_{\varepsilon})x \big)  \Big)-\partial_j x &= \psi_{\varepsilon}\ast \Big( \big(\partial_j U(\phi_{\varepsilon})\big)\, U(\chi_{\varepsilon})x\Big) + \psi_{\varepsilon}\ast \Big( U(\phi_{\varepsilon})\,\big(\partial_j U(\chi_{\varepsilon})\big) \,x\Big)\\                                                                       
            &\quad + \Big(\psi_{\varepsilon}\ast\big( U(\phi_{\varepsilon})U(\chi_{\varepsilon})\partial_j x\big)-\partial_j x\Big).
        \end{align*}
        Due to Corollary \ref{space_frequency_approximation}, the latter term vanishes in the $L_p$-norm as $\varepsilon\to 0$.
        
        For the first two terms, we apply H\"older's inequality and Lemma \ref{cancellation_lemma}. For the first summand, we apply \eqref{young_inequality},
        \begin{equation*}
            \|\psi_{\varepsilon}\ast \Big(  \big(\partial_j U(\phi_{\varepsilon})\big)\, U(\chi_{\varepsilon})x \Big)\|_p \leq \|\psi_{\varepsilon}\|_1\|\chi_{\varepsilon}\|_1\|\partial_j U(\phi_{\varepsilon})\|_\infty \|x\|_p\lesssim \varepsilon\|\chi\|_1\|\psi\|_1\|\phi\|_1 \|x\|_p
        \end{equation*}
        and this vanishes as $\varepsilon\to 0$. The second summand also vanishes as $\varepsilon\to 0$ due to an identical argument, and this completes the case $m=1$.
        
        The cases $m\geq 2$ follow similarly.
     \end{proof}
    At the time of this writing, we are unable to prove that the inclusion $\cA \subset   \dot{W}_p ^m (\qe)$ is dense. In the classical (commutative) setting or on quantum tori, this can be achieved by an application of a Poincar\'e inequality (see, e.g., \cite[Theorem 7]{HK1995}).
    To the best of our knowledge, no adequate replacement is known in the noncommutative setting. 
    In the following proposition, to obtain the desired convergence in $\dot{W}^1_d(\qe)$ norm, we have to assume additionally that $x\in L_p(\qe)$ for some $d \leq  p <\infty$. This is the ultimate cause of the \emph{a priori} assumption
    in the statements of Theorems \ref{sufficiency}, \ref{trace formula} and \ref{necessity} that $x \in L_p(\qe)$ for some $d\leq p < \infty$.
    \begin{prop}\label{dense-Schwartz}
        If $x\in \dot{W}_d^1 (\qe)\cap L_p(\qe)$ for some $d\leq p<\infty$, then the sequence $\psi_{\varepsilon}\ast (U(\phi_\varepsilon)U(\chi_{\varepsilon})x)$ converges to $x$ in $\dot{W}_d^1  $-seminorm when $\varepsilon \rightarrow 0^+$.
    \end{prop}
    \begin{proof}
        Let $1\leq j\leq d$. Applying the Leibniz rule:
        \begin{align*}
        \partial_j\Big(\psi_{\varepsilon}\ast \big(U(\phi_{\varepsilon})U(\chi_{\varepsilon})x \big) \Big)-\partial_jx &= \psi_{\varepsilon}\ast \Big( \big(\partial_j U(\phi_{\varepsilon})\big)\, U(\chi_{\varepsilon})x\Big)
                                                                                        + \psi_{\varepsilon}\ast  \Big( U(\phi_{\varepsilon})\,\big(\partial_j U(\chi_{\varepsilon})\big)x \Big)\\
                                                                                        &\quad + \Big(\psi_{\varepsilon}\ast \big(U(\phi_{\varepsilon})U(\chi_{\varepsilon})\partial_j x \big)-\partial_j x \Big).
        \end{align*}
The latter term vanishes as $\varepsilon\to 0$, as a consequence of Theorem \ref{frequency_approximation}. 
        
        For the first two terms, since $x\in L_p(\qe)$ for some $p \geq d$ we can apply H\"older's inequality. E.g. for the first term we have:
        $$\Big\| \psi_{\varepsilon}\ast \Big( \big(\partial_j U(\phi_{\varepsilon})\big)\, U(\chi_{\varepsilon})x\Big) \Big\|_d \lesssim  \|\partial_j U( \phi_\varepsilon) \|_q\|x\|_p\,,$$
        where $\frac 1 d  = \frac 1 p + \frac 1 q $. Using Lemma \ref{cancellation_lemma}, $\|\partial_j U(\phi_{\varepsilon})\|_q\rightarrow 0$ as $\varepsilon\to 0$. 
        The second term is handled similarly.
        Therefore, $\Big\| \partial_j\Big(\psi_{\varepsilon}\ast \big(U(\phi_{\varepsilon})U(\chi_{\varepsilon})x \big) \Big) -\partial_jx \Big\|_d  \ra 0$ and this completes the proof.
    \end{proof} 
    
    Now using the density of $\cA$ in $\qcS$ in its Fr\'echet topology, we may conclude the following key result:
    \begin{cor}\label{final_approximation_corollary}
        Let $x \in L_p(\qe)\cap \dot{W}^1_d(\qe)$ for some $d\leq p < \infty$. There exists a sequence $\{x_n\}_{n\geq 0}\subset \cA$ such that for all $1\leq j\leq d$:
        \begin{equation*}
            \lim_{n\to\infty} \|\partial_jx_n-\partial_j x\|_d = 0.
        \end{equation*}
    \end{cor}

\subsection{Cwikel type estimates}

Let $x\in L_\infty(\qe)$, then by definition, $x$ is a bounded operator in $B(L_2({\mathbb R}^d))$. On the other hand, for a (Borel) function $g$ on ${\mathbb R}^d$, we may define:
$$M_g = g(  {\mathcal D}_1,\cdots,   {\mathcal D}_d  )= g( \ri\nabla_\theta)$$
via functional calculus. As ${\mathcal D}_k$ is merely the operator $\xi(t) \mapsto t_k\xi(t)$, it follows that $M_g$ is the multiplication operator:
\begin{equation}\label{multiplication}
    M_g\xi(t) = g(t)\xi(t),\quad \dom(M_g) = L_2({\mathbb R}^d,|g(t)|^2\,dt).
\end{equation}
We call operators of the form $M_g$ Fourier multipliers of $\qe$.

Note that if $x \in L_2(\qe)$, we may still consider $x$ as a (potentially unbounded) operator on $L_2({\mathbb R}^d)$, with initial domain ${\mathcal S}({\mathbb R}^d)$.

The following theorem, quoted from \cite{LeSZ2017}, gives sufficient conditions for operators of the form $x  M_g$ to be in the Schatten class ${\mathcal L}_p (L_2({\mathbb R}^d))$ or the corresponding weak Schatten classes.

    \begin{thm}\label{Cwikel-type}
    Let $x\in L_p(\qe)$ with $2\leq p <\infty$.
        \begin{enumerate}[\rm (i)]
            \item\label{L_p cwikel} If $g\in L_p({\mathbb R}^d)$, then $ x  M_g$ is in ${\mathcal L}_p(L_2({\mathbb R}^d))$ and 
                $$\| x  M_g\|_{{\mathcal L}_p} \lesssim_p \|x\|_p \|g\|_p.$$
            \item\label{weak L_p cwikel} If $g\in L_{p,\infty}({\mathbb R}^d)$ with $p > 2$, then $ x  M_g$ is in ${\mathcal L}_{p,\infty}(L_2({\mathbb R}^d)) $ and 
                $$\| x  M_g\|_{{\mathcal L}_{p,\infty}} \lesssim_p \|x\|_p \|g\|_{p,\infty}.$$
            \item\label{cwikel_estimate_Sob}  Let $x \in  W^{d}_1(\qe)$. Then $ x J_\theta ^{-d} \in {\mathcal L}_{1,\infty}$ and
                 $$\| x  J_\theta ^{-d}\|_{{\mathcal L}_{1,\infty}} \lesssim_p C_d \|x\|_{W_1^d}.$$
        \end{enumerate}
    \end{thm}
    \begin{proof}
        Theorem 7.2 in \cite{LeSZ2017} says that 
        \begin{equation}\label{interpolated cwikel estimate}
            \|xM_g\|_{E({\mathcal B}(L_2({\mathbb R}^d)))} \lesssim_E \|x\otimes g\|_{E(L_\infty(\qe)\otimes L_\infty({\mathbb R}^d))}
        \end{equation}
        for any interpolation space $E$ of the couple $(L_2,L_\infty)$. Taking $E  =L_p$ in \eqref{interpolated cwikel estimate}, we get (i). For (ii), we take $E  =L_{p,\infty}$ and use the estimate
        \begin{equation*}
            \|x\otimes g\|_{L_{p,\infty}(L_\infty(\qe)\otimes L_\infty({\mathbb R}^d))}\leq \|x\|_p \|g\|_{p,\infty}
        \end{equation*}
        to immediately conclude the proof.   
        
        \eqref{cwikel_estimate_Sob} is merely an application of \cite[Theorem 7.6]{LeSZ2017}.
        Since the function $  (1+|t|^2)^{-d/2}$ is in $\ell_{1,\infty} (L_\infty({\mathbb R}^d))$\footnote{see \cite[pp.~38]{Simon1979} for the definition of this function space.}, it follows that $ x J_\theta^{-d} \in {\mathcal L}_{1,\infty}$.
    \end{proof}

\section{Proof of Theorem \ref{sufficiency}}\label{section-sufficiency}

This section is devoted to the proof of Theorem \ref{sufficiency}, that is, that the condition $x \in \interspace\cap W^{1}_d(\qe)$ is sufficient
for $\qd x \in {\mathcal L}_{d,\infty}$, and with an explicit norm bound: 
$$\|\qd x\|_{d,\infty} \lesssim_d \|x\|_{\dot{W}^{1}_d(\qe)}.$$
The proof given here is similar to the corresponding result on quantum tori \cite{MSX2018}, relying heavily
on the Cwikel type estimate stated in the last section.

The following two lemmas are easily deduced from Theorem \ref{Cwikel-type}.  

Consider the function on ${\mathbb R}^d$, $\xi\mapsto (1+|\xi|^2)^{-\frac{d}{2}}$. When $|\xi| > 1 $, we have $(1+|\xi| ^2)^{-\frac{d}{2}} \leq |\xi|^{-d}$. For $|\xi| \leq 1$, $(1+|\xi| ^2)^{-\frac{d}{2}}$ is bounded from above by $1$. Hence $\xi \mapsto (1+|\xi| ^2)^{-\frac{d}{2}} \in L_{1,\infty} ({\mathbb R}^d)$, and so $\xi\mapsto (1+|\xi| ^2)^{-\frac{\bt}{2}}\in L_{\frac{d}{\bt} , \infty} ({\mathbb R}^d)$. Recall $J_\theta = (1- {\mathcal D}elta_\theta)^{1/2}$. Then we have:

        \begin{lem}\label{Cwikel-type-lem}
        Consider the linear operator $x J_\theta ^{-  \beta   } $ on ${\mathbb C}^N \ot L_2({\mathbb R}^d)$. If $x\in L_{\frac d \beta}(\qe)$ with $\frac d \beta  >  2$, then $x J_\theta ^{-  \beta   } \in {\mathcal L}_{\frac d  \beta, \infty} ,$ and
        $$\|x J_\theta ^{-  \beta   }\|_{{\mathcal L}_{\frac d \beta, \infty}} \lesssim_{d,\beta} \|x\|_{\frac d \beta}.$$
        \end{lem}

\begin{lem}\label{Cwikel-xp}
Suppose that $p>\frac d 2$ and $x\in L_p(\qe)$. If $p\geq 2$, then:
$$ \left\|\left[\sgn({\mathcal D})  - \frac{{\mathcal D}}{\sqrt{1+{\mathcal D}^2}}, 1\ot  x  \right]\right\|_{{\mathcal L}_p}  \lesssim_{p,d}  \|x\|_p .        $$

\end{lem}

\begin{proof}
Let $1\leq j \leq d$, and for $\xi\in {\mathbb R}^d$ define
$$h_j(\xi) := \frac{\xi_j}{|\xi|} -\frac{\xi_j}{(1+|\xi|^2)^{\frac{1}{2}}}.$$
Thus,
\begin{equation*}
   M_{h_j} = h_j(\ri\nabla_\theta) = \frac{{\mathcal D}_j}{\sqrt{-{\mathcal D}elta_\theta}}-\frac{{\mathcal D}_j}{(1-{\mathcal D}elta_\theta)^{\frac{1}{2}}}
\end{equation*}
Note that there is no ambiguity in writing $\frac{{\mathcal D}_j}{\sqrt{-{\mathcal D}elta_{\theta}}}$, as this is simply $M_g$ for $g(\xi) = \frac{\xi_j}{|\xi|}$.
and so,
\begin{equation*}
    \sgn({\mathcal D}) - \frac{{\mathcal D}}{\sqrt{1+{\mathcal D}^2}} = \sum_{j=1}^d \gamma_j\otimes \Big(\frac{{\mathcal D}_j}{\sqrt{-{\mathcal D}elta_\theta}}-\frac{{\mathcal D}_j}{(1-{\mathcal D}elta_\theta)^{\frac{1}{2}}}\Big) = \sum_{j=1}^d \gamma_j\otimes  M_{h_j}.
\end{equation*}
One can easily check that $h_j \in L_p({\mathbb R}^d)$ as $p> \frac d 2$. Expanding out the commutator,
\be 
\left[\sgn({\mathcal D})  - \frac{{\mathcal D}}{\sqrt{1+{\mathcal D}^2}}, 1\ot  x  \right] = \left[\sum_{j=1}^d \g _j  \ot M_{h_j}, 1\ot   x  \right]
=\sum_{j=1}^d  \g_j \ot  [M_{h_j},    x].
\ee
Hence, 
\be 
\begin{split}
\left\|\left[\sgn({\mathcal D})  - \frac{{\mathcal D}}{\sqrt{1+{\mathcal D}^2}}, 1\ot   x  \right]\right\|_{{\mathcal L}_p} &\leq d \max_{1\leq j \leq d} \left\|\left[ M_{h_j},    x  \right]\right\|_{{\mathcal L}_p}\\
&\leq d \max_{1\leq j \leq d} \left(\left\| M_{h_j}   x   \right\|_{{\mathcal L}_p}+\left\|   x   M_{h_j}\right\|_{{\mathcal L}_p}\right)\\
&= d \max_{1\leq j \leq d} \left(\|  x ^*  M_{h_j}\|_{{\mathcal L}_p}+   \|   x  M_{h_j}\|_{{\mathcal L}_p}\right).
\end{split}
\ee
The desired conclusion follows then from Theorem \ref{Cwikel-type}.(i).
\end{proof}

The proof of the next lemma is modeled on that of \cite[Lemma~4.2]{MSX2018} and \cite[Lemma~10]{LMSZ2017}, via the technique of double operator integrals (see \cite{PSW2002} and \cite{PS2009} and references therein). For the convenience of the reader, let us give an brief introduction of double operator integrals, and sketch the proof of the next lemma.

Let $H$ be a (complex) separable Hilbert space. Let $D_0$ and $D_1$ be self-adjoint (potentially unbounded) operators on $H$, and $E^0$ and $E^1$ be the associated spectral measures. For all $x, y \in {\mathcal L}_2(H)$, the measure $(\lambda, \mu)  \mapsto {\mathrm{tr}}(x\, dE^0(\lambda) \, y \, dE^1(\mu)  )$ is a countably additive complex valued measure on ${\mathbb R}^2$. We say that $\phi \in L_\infty({\mathbb R}^2)$ is $E^0 \otimes E^1$ integrable if there exists an operator $T_\phi ^{D_0, D_1} \in \mathcal{B}  ({\mathcal L}_2(H))$ such that for all $x, y \in {\mathcal L}_2(H)$,
$${\mathrm{tr}} (x\,T_\phi ^{D_0, D_1} y  ) =\int _{{\mathbb R}^2}    \phi(\lambda, \mu )   {\mathrm{tr}}(x\, dE^0(\lambda) \, y \, dE^1(\mu)  ). $$
The operator $T_\phi ^{D_0, D_1} $ is called the transformer. For $A\in {\mathcal L}_2(H) $, we define
\beq\label{doi-def}
T_\phi ^{D_0, D_1}(A)=\int _{{\mathbb R}^2}    \phi(\lambda, \mu )   dE^0(\lambda) \, A \, dE^1(\mu)  . \eeq
This is called a double operator integral.

\begin{lem}\label{commutator-Sob}
Let $x\in \qcS$. Then 
$$ \big\|\big[\frac{{\mathcal D}}{\sqrt{1+{\mathcal D}^2}}, 1\ot  x  \big]\big\|_{{\mathcal L}_{d,\infty} }  \lesssim_d  \|x\|_{\dot W_d^1}.$$
\end{lem}

\begin{proof}
Set $g(t)= t (1+t^2)^{-\frac{1}{2}}$ for $t\in {\mathbb R}$. Since all of the derivatives of $x$ are bounded, we may apply \cite[Theorem~4.1]{BS1989}, which asserts that:
\beq\label{repr-commutator}
[g({\mathcal D}), 1\ot  x] =  T_{g^{[1]}}^{{\mathcal D},{\mathcal D}} ([{\mathcal D}, 1\ot  x]),
\eeq
where $g^{[1]}(\lambda,\mu )  := \frac{g(\lambda)-g(\mu )}{\lambda-\mu }  =\p_1(\lambda,\mu)\p_2(\lambda,\mu) \p_3(\lambda,\mu)$, with 
$$\psi_1 = 1+ \frac{1-\lambda \mu }{(1+ \lambda^2)^{\frac 1 2 } (1+\mu ^2) ^{\frac 1 2 } },\;\; \psi_2 = \frac{(1+\lambda^2)^{\frac 1 4} (1+\mu ^2)^{\frac 1 4} }{(1+ \lambda^2)^{\frac 1 2 } + (1+\mu ^2) ^{\frac 1 2 } },\;\;\psi_3 =  \frac{1   }{(1+ \lambda^2)^{\frac 1 4 } (1+\mu ^2) ^{\frac 1 4 } }.$$ 
It follows that 
\beq\label{TDDg}
T_{g^{[1]}}^{{\mathcal D},{\mathcal D}}  = T_{\p_1}^{{\mathcal D},{\mathcal D}}  T_{\p_2}^{{\mathcal D},{\mathcal D}}  T_{\p_3}^{{\mathcal D},{\mathcal D}} .
\eeq
\cite[Lemma 8]{LMSZ2017} ensures the boundedness of the transformer $T_{\p_2}^{{\mathcal D},{\mathcal D}} $, on both ${\mathcal L}_1$ and ${\mathcal L}_\infty$. For $k=1,3$ the function $\psi_k$ can be written as a linear combination of products of bounded functions of $\lambda$ and of $\mu$, and from this it follows that $T_{\p_k}^{{\mathcal D},{\mathcal D}}$ is also a bounded linear map on ${\mathcal L}_{1}$
and ${\mathcal L}_\infty$; see e.g. \cite[Corollary 2]{PS2009} and \cite[Corollary 2.4]{RX2011}. Then by real interpolation of $({\mathcal L}_1, {\mathcal L}_\infty)$ (see \cite{DDP1992}), the transformers $T_{\p_k}^{{\mathcal D},{\mathcal D}} $ with $k=1,2,3$ are bounded linear transformations from ${\mathcal L}_{d,\infty}$ to ${\mathcal L}_{d,\infty}$. Using \eqref{repr-commutator} and the product representation of $g$ in \eqref{TDDg}, we have
\be\begin{split}
\|[g({\mathcal D}), 1\ot  x]\|_{{\mathcal L}_{d,\infty}}&\leq \| T_{\p_1}^{{\mathcal D},{\mathcal D}} \|_{{\mathcal L}_{d,\infty}\ra {\mathcal L}_{d,\infty}}  \| T_{\p_2}^{{\mathcal D},{\mathcal D}} \|_{{\mathcal L}_{d,\infty}\ra {\mathcal L}_{d,\infty}}\\
&\;\;\;\;\;\;\;\;    \times  \| T_{\p_3}^{{\mathcal D},{\mathcal D}} ([{\mathcal D}, 1\ot  x])\|_{ {\mathcal L}_{d,\infty}}  \\
&\lesssim_d \| T_{\p_3}^{{\mathcal D},{\mathcal D}} ([{\mathcal D}, 1\ot  x])\|_{ {\mathcal L}_{d,\infty}}.
\end{split}\ee

 Since $\p_3(\lambda,\mu)=(1+\lambda^2)^{-1/4} (1+\mu^2)^{-1/4}$, by \eqref{doi-def}, we have
$$T_{\p_3}^{{\mathcal D},{\mathcal D}} ([{\mathcal D}, 1\ot  x]) = (1+{\mathcal D}^2)^{-1/4} [{\mathcal D}, 1\ot   x](1+{\mathcal D}^2)^{-1/4}. $$
Recalling that ${\mathcal D} = \sum_{j=1}^d \g_j\ot {\mathcal D}_j$,
\be\begin{split}
\|[g({\mathcal D}), 1\ot  x]\|_{{\mathcal L}_{d,\infty}} &\lesssim_d  \| (1+{\mathcal D}^2)^{-1/4} [{\mathcal D}, 1\ot  x](1+{\mathcal D}^2)^{-1/4}\|_{{\mathcal L}_{d,\infty}} \\
&\lesssim_d  \sum_{j=1}^d \| (1+{\mathcal D}^2)^{-1/4} [\g_j\ot {\mathcal D}_j, 1\ot  x](1+{\mathcal D}^2)^{-1/4}\|_{{\mathcal L}_{d,\infty}}.
\end{split}\ee
But by definition, $[\g_j\ot {\mathcal D}_j, 1\ot  x] =  \g_j \ot  \partial_j x $, thus we obtain
$$\| (1+{\mathcal D}^2)^{-1/4} [\g_j\ot {\mathcal D}_j, 1\ot  x](1+{\mathcal D}^2)^{-1/4}\|_{{\mathcal L}_{d,\infty}} = \| J_\theta^{-1/2} \, \partial_j x \, J_\theta^{-1/2}\|_{{\mathcal L}_{d,\infty}}.$$
Here the first norm $\|\cdot\|_{{\mathcal L}_{d,\infty}}$ is the norm of ${\mathcal L}_{d,\infty}({\mathbb C}^N\otimes L_2({\mathbb R}^d))$, and the second one is the norm of ${\mathcal L}_{d,\infty}(L_2({\mathbb R}^d))$, and $J_\theta=  (1-{\mathcal D}elta_\theta)^{1/2} $.
We are reduced to estimating the quantity $\| J_\theta^{-1/2} \, \partial_j x \, J_\theta^{-1/2}\|_{{\mathcal L}_{d,\infty}}$. By polar decomposition, for every $j$, there is a partial isometry $V_j$ on $L_2({\mathbb R}^d)$ such that  
$$\partial_j x  =  V_j |\partial_j x | =  V_j |\partial_j x |^{\frac 1 2 } |\partial_j x |^{\frac 1 2 }.$$
Taking $\beta=\frac 1 2 $, and recalling that $x$ is such that $\|V_j |\partial_j x|^{\frac 1 2 } \|_{2d}\leq \|\, |\partial_j x|^{\frac 1 2 } \|_{2d} = \|\partial_j x \|_d^{\frac 1 2 }<\infty$, since $2d > 2$, we may apply Lemma \ref{Cwikel-type-lem}.\eqref{weak L_p cwikel} to get
$$\|  |\partial_j x|^{\frac 1 2} J_\theta^{-1/2}\|_{{\mathcal L}_{2d,\infty}}= \|  J_\theta^{-1/2}  |\partial_j x|^{\frac 1 2} \|_{{\mathcal L}_{2d,\infty}}\lesssim_d \|\, |\partial_j x|^{\frac 1 2 } \|_{2d}$$
and 
$$\|  J_\theta^{-1/2} V_j |\partial_j x|^{\frac 1 2}\|_{{\mathcal L}_{2d,\infty}}\lesssim_d \|V_j |\partial_j x|^{\frac 1 2 } \|_{2d}\lesssim_d \|\, |\partial_j x|^{\frac 1 2 } \|_{2d}.$$
Thus, by H\"{o}lder's inequality for weak Schatten classes, 
$$\|  J_\theta^{-1/2}  \partial_j x \,  J_\theta^{-1/2}\|_{{\mathcal L}_{d,\infty}}\lesssim_d \|\, |\partial_j x|^{\frac 1 2 } \|_{2d}^2\lesssim_d \|\partial_j x \|_d.$$
Combining the preceding estimates, we arrive at
$$\|[g({\mathcal D}), 1\ot   x]\|_{{\mathcal L}_{d,\infty}} \lesssim_d \sum_{j=1}^d\|\partial_j x \|_d\lesssim_d \|x\|_{\dot W_d^1},$$
which completes the proof.
\end{proof}

Now we are able to complete the proof of Theorem \ref{sufficiency}. 
\begin{proof}[Proof of Theorem \ref{sufficiency}]
Lemmas \ref{Cwikel-xp}, \ref{commutator-Sob}  and the inequality $\|T\|_{d,\infty} \leq \|T\|_d$ yield:
\beq\label{qd-Sobolev}
 \|\qd x    \|_{{\mathcal L}_{d,\infty}} \lesssim_d \|[g({\mathcal D}),1\otimes x]\|_{d,\infty}+\|[\sgn({\mathcal D})-g({\mathcal D}),1\otimes x]\|_{d,\infty} \lesssim_d \|x\|_d + \|x\|_{\dot W_d^1} ,
\eeq

for all $x \in \qcS$, and with constants independent of $\theta$. We are going to get rid of the dependence on  $\|x\|_d $ by a dilation argument as follows. Let $\la >0 $ and $\Psi_\la : L_\infty (\qe) \rightarrow L_\infty ({\mathbb R}^d_{\la^2\theta})$ be the $*$-isomorphism defined in \eqref{def-dilation}. By \eqref{dilation}, for $x\in L_\infty(\qe)$, we have $\Psi_\la (x) = \sigma_\la x \sigma_\la ^*$. Since the operator $ \frac{{\mathcal D}_j}{\sqrt{-{\mathcal D}elta_\theta}}$, viewed as a Fourier multiplier on ${\mathbb R}^d$, commutes with $\sigma_\la$ (and $\sigma_\la^*$), we have
\be\begin{split}
\qd \big(   \Psi_\la(x) \big) &  = \ri [\sgn({\mathcal D}) , 1\ot \Psi_\la (x)]= \ri [\sgn({\mathcal D}) , 1\ot \sigma_\la  x \sigma_\la^*] \\
&=    \ri   \sigma_\la  [\sgn({\mathcal D}) , 1\ot    x ] \sigma_\la^*  =  \sigma_\la \qd x \, \sigma_\la^*.
\end{split}\ee
Whence, $\|\qd \big(   \Psi_\la(x) \big)    \|_{{\mathcal L}_{d,\infty}} =\|\qd x    \|_{{\mathcal L}_{d,\infty}} $. Applying \eqref{qd-Sobolev} to $\Psi_\la (x) \in  L_\infty ({\mathbb R}^d_{\la^2\theta})$, we obtain
\be 
 \|\qd \big(   \Psi_\la(x) \big)    \|_{{\mathcal L}_{d,\infty}}  \lesssim_d   \|  \Psi_\la(x)\|_d +B_{d}  \|  \Psi_\la(x)\|_{\dot W_d^1} .
\ee
By virtue of Proposition \ref{dilation-Ld}, we return back to $x\in L_\infty(\qe)$:
\be 
\|\qd x    \|_{{\mathcal L}_{d,\infty}} =  \|\qd \big(   \Psi_\la(x) \big)    \|_{{\mathcal L}_{d,\infty}}  \lesssim_d \la   \|  x\|_d +\|  x\|_{\dot W_d^1} .  
\ee
Letting $\la \rightarrow 0$ completes the proof of Theorem \ref{sufficiency} for  $x\in \qcS$.

The general case $x\in  \dot{W}_d^1 (\qe) \cap  \interspace$ is achieved by approximation. By Proposition \ref{dense-Schwartz}, select a sequence $\{x_n\}$ in $\qcS$ such that $x_n \ra  x$ in $\dot{W}_d^1$ seminorm. Corollary \ref{space_frequency_approximation} implies that we can choose this sequence such that we also have that $x_n\ra x$ in the $L_p(\qe)$-sense. For these Schwartz elements $x_n$, we have $\|\qd x_m - \qd x_n\|_{  {\mathcal L}_{d, \infty}} \lesssim_d \|x_m -x_n \|_{\dot{W}_d^1}$, so $\{\qd x_n\}$ is Cauchy in ${\mathcal L}_{d, \infty}$, and thus converges to some limit (say, $L$) in the ${\mathcal L}_{d, \infty} $ quasinorm. 

Let $\eta\in L_2({\mathbb R}^d)$ be compactly supported, and let $K\subset {\mathbb R}^d$ be a compact set containing the support of $\eta$. Then $(x_n-x)\eta = (x_n-x)M_{\chi_K}\eta$. We have:
\begin{equation*}
    \|(x_n-x)\eta\|_2 = \|(x_n-x)\chi_K\eta\|_2 \leq \|(x_n-x)M_{\chi_K}\|_{\infty}\|\eta\|_2 \leq \|(x_n-x)\chi_K\|_{{\mathcal L}_p}\|\eta\|_2.
\end{equation*}
Theorem \ref{Cwikel-type} implies that $\|(x_n-x)M_{\chi_K}\|_{{\mathcal L}_p} \lesssim_{p,K} \|x_n-x\|_p$, and since we have selected the sequence to converge in the $L_p(\qe)$ sense:
\begin{equation}\label{strong_one_side}
    \lim_{n\to\infty} \|(x_n-x)\eta\|_2 = 0.
\end{equation}
Similarly, if $\xi \in {\mathbb C}^N\otimes L_2({\mathbb R}^d)$ is compactly supported, then $\sgn({\mathcal D})\xi$ is still compactly supported and we have:
\begin{equation}\label{strong_other_side}
    \lim_{n\to\infty} \|1\otimes (x_n-x)\sgn(D)\xi\|_2 = 0.
\end{equation}
Combining \eqref{strong_one_side} and \eqref{strong_other_side} implies that $(\qd x_n)\xi \rightarrow (\qd x)\xi$ for all compactly supported $\xi \in {\mathbb C}^N\otimes L_2({\mathbb R}^d)$. Since we know that $\qd x_n\to L$
in the ${\mathcal L}_{d,\infty}$ topology, it follows that $\qd x = L$, and therefore $\qd x \in {\mathcal L}_{d,\infty}$.





To complete the proof, we note that for these Schwartz elements $x_n$,
$$\|\qd x_n \|_{{\mathcal L}_{d, \infty}}\lesssim_d   \|x_n\|_{\dot{W}^{1}_d}.$$
Upon taking the limit $n \ra  \infty $ we arrive at:
\begin{equation*}
    \|\qd x\|_{{\mathcal L}_{d,\infty}}\lesssim_{d} \|x\|_{\dot{W}^1_d}.
\end{equation*}
\end{proof}

\section{Commutator estimates for $\qe$}\label{commutator-estimates}

This section is devoted to the proof of Theorem \ref{main-commutator}, which is an essential ingredient for our proof of Theorem \ref{trace formula}
i.e., the computation of $\vf(|\qd x|^d)$ when $x\in L_\infty(\qe) \cap \dot  W _d^1 (\qe)$ and $\vf$ is a continuous normalised trace on ${\mathcal L}_{1,\infty}$. One powerful tool used in \cite{MSX2018} for quantum tori is the theory of noncommutative pseudodifferential operators. The proof in \cite{MSX2018} proceeds by viewing the quantised differential $\qd x =  \ri [\sgn({\mathcal D}), 1\ot x ]$ as a pseudodifferential operator, then determining its (principal) symbol and order, and finally appealing to Connes' trace formula as obtained in \cite{MSZ2018}. 

Despite the development of pseudodifferential operators on quantum Euclidean spaces in \cite{GJP2017, LM2016}, we have found it instructive to attempt a direct proof of Theorem \ref{main-commutator}. This has two main advantages: first, it makes
the present text self-contained, and more importantly the methods presented below are based only on operator theory and can be generalised to settings where no pseudodifferential calculus is available. 

For potential future utility we will prove Theorem \ref{main-commutator} for the full range of parameters $(\alpha,\beta)$, although ultimately we will only need certain specific choices of $\alpha$ and $\beta$.

Let $\cA\subseteq \qcS$ be a factorisable subalgebra as in Proposition \ref{factorisation}. 

\medskip

The main target of this section is to give the proof of Theorem \ref{main-commutator}, which is technical and somewhat tedious, and so is divided into several steps presented in the following subsections.

\subsection{Commutator identities}

    The following integral formula will be useful: let $\zeta < 1$ and $\eta > 1-\zeta$. Then for all $t> 0$ we have
    \beq\label{best_integral}
        \int_0^\infty \frac{1}{\lambda^{\zeta}(t+\lambda)^{\eta}}\, d\lambda = t^{1-\zeta-\eta}\, {\mathrm B}(\eta+\zeta-1,1-\zeta).
    \eeq
    where ${\mathrm B}(\cdot,\cdot)$ is the Beta function.
        
    For an operator $T\in {\mathcal B}(L_2({\mathbb R}^d))$, let $L_\theta(T) := J_\theta^{-1}[J_\theta^2,T]$ whenever it is defined, and define $\delta_\theta(T) := [J_\theta,T]$ similarly. Inductively, for $k\in {\mathbb N}$ we define $L_\theta^k (T)= L_\theta(L_\theta^{k-1}(T))$ and $\delta^k _\theta(T) =\delta_\theta(\delta^{k-1}_\theta(T))$.
    We also make the convention that $L_\theta^0(T)=T$ and $\delta^0_\theta(T) =T$. Note that $L_{\theta}(T)J_{\theta}^{-1} = L_{\theta}(TJ_{\theta}^{-1})$.
    
    The following theorem states that to prove that $\delta_{\theta}(T)$ is in a certain ideal, it suffices to show that $L^k_{\theta}(T)$ is in that ideal for all $k\geq 0$. The essential idea behind the proof
    goes back to \cite[Appendix B]{CM1995}. Here some extra care is needed for the quasi-Banach cases $0 <p\leq 1$.
    We make use of the theory of integration of functions valued in quasi-Banach developed by Turpin and Waelbroeck \cite{TW1,TW2,Kalton1985}. We will refer the reader to \cite{LeSZ2018} for results in the precise form we need.
    \begin{thm}\label{L_to_delta}
        Let $T$ be an operator on $L_2({\mathbb R}^d)$ which maps the Schwartz class ${\mathcal S}({\mathbb R}^d)$ into ${\mathcal S}({\mathbb R}^d)$. Assume that $L_\theta^k(T)$ is defined for all $k\geq 0$. 
        \begin{enumerate}[{\rm (i)}]
            \item{}\label{L_to_delta_bdd} If $L_\theta^k(T)$ has bounded extension for all $k\geq 0$, then $\delta_\theta^k(T)$ has bounded extension for all $k\geq 0$.
            \item{}\label{L_to_delta_cpt} Similarly if $p > 0$ and $L_\theta^k(T) \in {\mathcal L}_{p,\infty}$ for all $k\geq 0$, then $\delta_\theta^k(T)\in{\mathcal L}_{p,\infty}$ for all $k\geq 0$.
        \end{enumerate}
    \end{thm}    
    \begin{proof}
        Taking $\eta = 1$ and $\zeta = 1/2$ in \eqref{best_integral} yields
        \begin{equation}\label{best_integral_op}
            J_\theta^{-1} = \frac{1}{\pi} \int_0^\infty \frac{1}{\lambda^{1/2}(\lambda+ J_\theta^2)}\,d\lambda.
        \end{equation}
        Here since  $(\lambda + J_\theta^2) ^{-1}$ has bounded extension for all $\lambda\geq 0$, the integrand is a norm-continuous function of $\la$ and the integral converges in operator norm; see e.g. \cite[pp~701]{CP1998}. Since by assumption $T$ has bounded extension and maps ${\mathcal S}({\mathbb R}^d)$ to ${\mathcal S}({\mathbb R}^d)$, for any $\xi \in {\mathcal S}({\mathbb R}^d)\subset  \dom(J^2_\theta)$, multiplying by $J_\theta^2$ and taking the commutator with $T$ gives us
        \begin{equation*}
            [J_\theta,T] \xi = \frac{1}{\pi}\int_0^\infty \lambda^{-1/2}\left[\frac{J_\theta^2}{\lambda+J_\theta^2},T \right]\, \xi\, d\lambda ,
        \end{equation*}
       where the integrand on the right converges in the $L_2({\mathbb R}^d)$-valued Bochner sense.
        We manipulate the integrand as follows
        \begin{align*}
            \delta_\theta(T)\xi =[J_\theta,T] \, \xi &= \frac{1}{\pi}\int_0^\infty \lambda^{1/2}(\lambda+J_\theta^2)^{-1}[J_\theta^2,T](\lambda+J_\theta^2)^{-1} \, \xi\,d\lambda\\
                    &= \frac{1}{\pi}\int_0^\infty \lambda^{1/2}\frac{J_\theta}{\lambda+J_\theta^2}L_\theta(T)(\lambda+J_\theta^2)^{-1} \, \xi\,d\lambda\\
                    &= \frac{1}{\pi}\int_0^\infty \lambda^{1/2}\frac{J_\theta}{(\lambda+J_\theta^2)^2}L_\theta(T) \, \xi\,d\lambda\\
                    &\quad + \frac{1}{\pi}\int_0^\infty \lambda^{1/2}\frac{J_\theta}{\lambda+J_\theta^2}[L_\theta(T),(\lambda+J_\theta^2)^{-1}] \, \xi\,d\lambda\\
                    &= \frac{1}{\pi}\int_0^\infty \lambda^{1/2}\frac{J_\theta}{(\lambda+J_\theta^2)^2}\,d\lambda \cdot L_\theta(T) \, \xi\\
                    &\quad + \frac{1}{\pi}\int_0^\infty \lambda^{1/2}\frac{J_\theta^2}{(\lambda+J_\theta^2)^2}L_\theta^2(T)\frac{1}{\lambda+J_\theta^2}\, \xi\,d\lambda .\\
                    &= \frac{1}{2}L_\theta(T) \, \xi+\frac{1}{\pi}\int_0^\infty \lambda^{1/2}\frac{J_\theta^2}{(\lambda+J_\theta^2)^2}L_\theta^2(T)\frac{1}{\lambda+J_\theta^2} \, \xi\,d\lambda.
        \end{align*}
        In the last equality above, we have used the fact that
        \begin{equation*}
            \int_0^\infty \lambda^{1/2}\frac{J_\theta}{(\lambda+J_\theta^2)^2}\,d\lambda = \frac{\pi}{2}\, 1_{{\mathcal B}(L_2({\mathbb R}^d))},
        \end{equation*}
        which is deduced from \eqref{best_integral} by taking $\zeta = -1/2$ and $\eta=2$. Also note that all the integrands above converge in $L_2({\mathbb R}^d)$.
                
        Now if $L_\theta^2(T)$ has bounded extension, we have
        \begin{equation*}
            \left\|\frac{J_\theta^2}{(\lambda+J_\theta^2)^2}L_\theta^2(T)\frac{1}{\lambda+J_\theta^2}\right\|_{\infty}  \leq \|L_\theta^2(T)\|_{\infty} \frac{1}{4\lambda(\lambda+1)}.
        \end{equation*}
        Hence,
        \begin{equation*}
            \|\delta_\theta(T)\|_{\infty} \leq \frac{1}{2}\|L_\theta(T)\|_{\infty} + C \|L_\theta^2(T)\|_\infty ,
        \end{equation*}
        where $C>0$ is a certain constant.
        So if $L_\theta(T)$ and $L_\theta^2(T)$ have bounded extension, then $\delta_\theta(T)$ has bounded extension. Inductively, if $L_\theta^k(T)$ has bounded extension for all $k\geq 0$, then $\delta_\theta^k(T)$ has bounded extension for all $k\geq 0$.
        This completes the proof of part \eqref{L_to_delta_bdd}.
        
        We turn to the proof of part \eqref{L_to_delta_cpt}. If $p > 1$, then ${\mathcal L}_{p,\infty}$ can be given an equivalent norm making it a Banach ideal. Then we may give the same argument as part \eqref{L_to_delta_bdd}, but with the operator norm
        replaced by a norm for ${\mathcal L}_{p,\infty}$. On the other hand, if $p \leq 1$, then ${\mathcal L}_{p,\infty}$ cannot be given a Banach norm and therefore a more delicate argument is needed. Taking yet more commutators, for all $\xi \in {\mathcal S}({\mathbb R}^d)$ we have:
        \begin{align*}
            \delta_\theta(T)\xi &= \frac{1}{2}L_\theta(T)\xi + \frac{1}{\pi}\int_0^\infty \lambda^{1/2} \frac{J_\theta^2}{(\lambda+J_\theta^2)^3} L_\theta^2(T)\xi\,d\lambda \\
                        &\quad + \frac{1}{\pi}\int_0^\infty \lambda^{1/2}\frac{J_\theta^2}{(\lambda+J_\theta^2)^2}[L_\theta^2(T),(\lambda+J_\theta^2)^{-1}]\xi\,d\lambda\\
                        &= \frac{1}{2}L_\theta(T)\xi+\frac{{\mathrm B}(3/2,3/2)}{\pi}J_\theta^{-1}L_\theta^2(T)\xi+\frac{1}{\pi}\int_0^\infty \lambda^{1/2}\frac{J_\theta^3}{(\lambda+J_\theta^2)^3}L_\theta^3(T)\frac{1}{\lambda+J_\theta^2}\xi\,d\lambda.
        \end{align*}
        
        Iterating this process ultimately leads to the expansion, for each $n\geq 1$,
        \begin{equation}\label{deltaT}
            \delta_\theta(T) = \sum_{j=1}^{n-1} \frac{1}{\pi}{\mathrm B}(j-1/2,3/2)J_\theta^{1-j}L_\theta^j(T) + \frac{1}{\pi}\int_0^\infty \lambda^{1/2}\frac{J_\theta^n}{(\lambda+J_\theta^2)^n}L_\theta^n(T)\frac{1}{\lambda+J_\theta^2}\,d\lambda.
        \end{equation}
        The coefficients above are obtained by a choice of $\eta = j+1$ and $\zeta = -1/2$ in \eqref{best_integral} yielding
        \begin{equation*}
            \int_0^\infty \lambda^{1/2}\frac{J_\theta^j}{(\lambda+J_\theta^2)^{j+1}}\,d\lambda = {\mathrm B}(j-1/2,3/2)\, J_\theta^{1-j},
        \end{equation*}
        which are understood in the same meaning as \eqref{best_integral_op}.
                
        To complete the proof of \eqref{L_to_delta_cpt}, we will show that for any $p > 0$ we can choose $n$ large enough that the integral remainder term in \eqref{deltaT} can be proved to be in ${\mathcal L}_{p,\infty}$. To this end we use the non-convex integration theory of \cite{LeSZ2018}. Let $n > 1$, and define:
        \begin{equation*}
            \mathcal I_n(\lambda) = \frac{J_\theta^n}{(\lambda+J_\theta^2)^n},\quad \mathcal J (\lambda) = \frac{1}{\lambda+J_\theta^2}.
        \end{equation*}
        Let us show that we can choose $n$ sufficiently large such that if $X \in {\mathcal L}_{p,\infty}$, then $\int_0^\infty \lambda^{1/2}\mathcal I_n(\lambda)X \mathcal  J(\lambda)\,d\lambda$ is in ${\mathcal L}_{p,\infty}$.
        Specifically, we use \cite[Corollary 3.7]{LeSZ2018} combined with \cite[Proposition 3.8]{LeSZ2018}, which together imply that it suffices to have
        \begin{equation}\label{LeSZ_condition}
            \sum_{j\in {\mathbb N}_0} ((j+1)^{1/2}\|\mathcal I_n\|_{C^{2n}([j,j+1],{\mathcal B}(L_2({\mathbb R}^d)))}\|\mathcal J\|_{C^{2n}([j,j+1],{\mathcal B}(L_2({\mathbb R}^d)))})^{\frac{p}{p+1}} < \infty.
        \end{equation}
        and $n > \frac{1}{2p}$. 
                
        Now let us check \eqref{LeSZ_condition}. For $0\leq k\leq 2n$, we have
        \begin{align*}
            \left\|\frac{\partial^k}{\partial \lambda^k}\mathcal I_n(\lambda)\right\|  &= C_{k,n}\left\|\frac{J_\theta^n}{(\lambda+J_\theta^2)^{n+k}}\right\| \\
                                                                                           &\leq C_{k,n}\left\|\frac{J_\theta^n}{(\lambda+J_\theta^2)^n}\right\| \\
                                                                                           &= C_{k,n}\left\|\frac{1}{(\lambda J_\theta^{-1}+J_\theta)^n}\right\| .
        \end{align*}
        Since
        \begin{equation*}
            \lambda J_\theta^{-1}+J_\theta \geq \max\{{\lambda+1,2\lambda^{1/2}}\}  \,1_{{\mathcal B}(L_2({\mathbb R}^d))} ,
        \end{equation*}
        it follows that
        \begin{equation*}
            \left\|\frac{\partial^k}{\partial\lambda^k}\mathcal I_n(\lambda)\right\|  \leq C_{k,n}\min\{1,\lambda^{-n/2}\}.
        \end{equation*}
        For $\mathcal J(\la )$ the estimates are easier
        \begin{equation*}
            \left\|\frac{\partial^k}{\partial \lambda^k}\mathcal J(\la )\right\|  \leq C_k\frac{1}{\lambda+1}.
        \end{equation*}
        So if we choose $n$ large enough, \eqref{LeSZ_condition} is satisfied. Thus, if $L_\theta^n(T) \in {\mathcal L}_{p,\infty}$ then
        \begin{equation*}
            \frac{1}{\pi}\int_0^\infty \lambda^{1/2}\frac{J_\theta^n}{(\lambda+J_\theta^2)^n}L_\theta^n(T)\frac{1}{\lambda+J_\theta^2}\,d\lambda \in {\mathcal L}_{p,\infty}.
        \end{equation*}
        So, if all of $L_\theta(T),L_\theta^2(T),\cdots,L_\theta^n(T)$ are in ${\mathcal L}_{p,\infty}$ then \eqref{deltaT} implies that $\delta_\theta(T) \in {\mathcal L}_{p,\infty}$. Thus by induction, if $L_\theta^k(T) \in {\mathcal L}_{p,\infty}$ for every $k\geq 0$, then $\delta^k_\theta(T) \in {\mathcal L}_{p,\infty}$ for every $k\geq 0$.
    \end{proof}

\subsection{The case $\al=1$} 
    Now we commence the proof of Theorem \ref{main-commutator} by first proving the case $\al=1$, which is the easiest case since we can directly apply Theorem \ref{L_to_delta} and the Cwikel type estimate \cite{LeSZ2017}.
    \begin{lem}\label{L_cwikel_estimate}
        Let $x \in \qcS$. The operators $L_\theta^k( x)$ have bounded extension for all $k \geq 1$. Moreover, we have
        \begin{equation*}
            L_\theta^k( x)J_\theta^{-d} \in {\mathcal L}_{1,\infty}
        \end{equation*}
        for all $k\geq 0$.
    \end{lem}   
    \begin{proof}
        We have
        \begin{align*}
            L_\theta( x) &= J_\theta^{-1}\sum_{j=1}^d [{\mathcal D}_j^2, x]\\
                      &= J_\theta^{-1}\sum_{j=1}^d 2{\mathcal D}_j[{\mathcal D}_j, x]-[{\mathcal D}_j,[{\mathcal D}_j, x]]\\
                      &= \sum_{j=1}^d 2J_\theta^{-1}{\mathcal D}_j\, \partial_j x  -J_\theta^{-1} \partial^2_jx .
        \end{align*}
        Since $J_\theta^{-1}{\mathcal D}_j$ has bounded extension, it follows that $L_\theta(x)$ also has bounded extension.
        
        Since $L_\theta$ commutes with $J_\theta$ and each ${\mathcal D}_j$, for $k\geq 2$ we have
        \begin{equation*}
            L_\theta^k( x) = \sum_{j=1}^d 2J_\theta^{-1}{\mathcal D}_j   L_\theta^{k-1}( \partial_j x)-\sum_{j=1}^d J_\theta^{-1}L_\theta^{k-1}( \partial^2_j x).
        \end{equation*}
        So by induction on $k$, all $L_\theta^k(x)$ are bounded. Moreover, by convention $L^0(T) = T$, then for all $k\geq 1$ we get
        \begin{equation*}
            L_\theta^k(x)J_\theta^{-d} = \sum_{j=1}^d 2J_\theta^{-1}{\mathcal D}_jL_\theta^{k-1}( \partial_jx)J_\theta^{-d} - \sum_{j=1}^d J_\theta^{-1}L_\theta^{k-1}( \partial^2_jx)J_\theta^{-d}.
        \end{equation*}
        Hence Theorem \ref{Cwikel-type}(\ref{cwikel_estimate_Sob}) ensures $L_\theta^k(x)J_\theta^{-d} \in {\mathcal L}_{1,\infty}$.
    \end{proof}
    
    An immediate corollary of Lemma \ref{L_cwikel_estimate} together with Theorem \ref{L_to_delta}\eqref{L_to_delta_bdd} yields
    \begin{cor}\label{QC_infty}
        For all $x \in \qcS$ and $k\geq 0$, the operator $\delta_\theta^k(x)$ has bounded extension.
    \end{cor}
%

    The main technical underpinning of Theorem \ref{main-commutator} is the following Lemma:
    \begin{lem}\label{delta_cwikel_estimate}
        Let $x \in \cA$. Then for all $\beta > 0$ and all $k\geq 0$ we have
        \begin{equation*}
            \delta_\theta^k( x)J_\theta^{-\beta} \in {\mathcal L}_{d/\beta,\infty}.
        \end{equation*}
    \end{lem}
    \begin{proof}
        Let $T =  x J_\theta^{-d}$. Then from Lemma \ref{L_cwikel_estimate} and the fact that $J_{\theta}^{-1}$ commutes with $L_{\theta}$, we have that $L_\theta^k(T) \in {\mathcal L}_{1,\infty}$ for all $k\geq 0$. Thus, it follows from Theorem \ref{L_to_delta}\eqref{L_to_delta_cpt} that $\delta_\theta^k(T) = \delta_\theta^k( x)J_\theta^{-d}$
        is in ${\mathcal L}_{1,\infty}$, and this proves the result for $\beta = d$.
        
        Now if $\beta < d$, we can apply \eqref{ALT_inequality} with $r = d/\beta$, $A = \delta_{\theta}^k(x)$ and $B = J_{\theta}^{-\beta}$ to obtain:
        \begin{equation*}
            \delta_\theta^k( x)J_\theta^{-\beta} \in {\mathcal L}_{d/\beta,\infty}
        \end{equation*}
        thus the result is proved for for $0 < \beta \leq d$. 
        
        We will now complete the proof by an inductive argument, specifically by showing that if the result holds for $\beta$ then it holds for $\beta+1$.
            
        Suppose that the result is true for some $\beta > 0$. Then we write
        \begin{align*}
            \delta_\theta^k( x)J_\theta^{-\beta-1} &= [\delta_\theta^k( x),J_\theta^{-1}]J_\theta^{-\beta} + J_\theta^{-1}\delta_\theta^k( x)J_\theta^{-\beta}\\
                                        &= J_\theta^{-1}[J_\theta,\delta_\theta^k( x)]J_\theta^{-\beta-1}+J_\theta^{-1}\delta_\theta^k( x)J_\theta^{-\beta}\\
                                        &= J_\theta^{-1}\delta_\theta^{k+1}( x)J_\theta^{-\beta-1}+J_\theta^{-1}\delta_\theta^k( x)J_\theta^{-\beta},
        \end{align*}
        By the factorisation property of $\cA$ (see Proposition \ref{factorisation}), we can write $x$ as a finite linear combination of products, $x = \sum_{j=1}^n y_jz_j$, where each $y_j, z_j  \in \cA$. Using the Leibniz rule on the $j$th summand, we deduce
        \begin{align*}
            \delta_\theta^k(y_jz_j)J_\theta^{-\beta-1} &=\sum_{j=0}^{k+1} \binom{k+1}{j} J_\theta^{-1}\delta_\theta^j( y_j)\delta_\theta^{k+1-j}( z_j )J_\theta^{-\beta }J_\theta^{ -1}\\
                                        &\quad + \sum_{j=0}^k \binom{k}{j} J_\theta^{-1}\delta_\theta^j( y_j)\delta_\theta^{k-j}( z_j )J_\theta^{-\beta}.
        \end{align*}
        Hence by the H\"older inequality and the fact that $J_\theta^{ -1}$ is bounded, 
        \begin{equation*}
            \delta_\theta^k(x)J_\theta^{-\beta-1} \in {\mathcal L}_{d/\beta,\infty}\cdot{\mathcal L}_{d,\infty}\subseteq {\mathcal L}_{d/(\beta+1),\infty}.
        \end{equation*}
        Thus the result holds for $\beta+1$, and this completes the proof.
    \end{proof}

 Observing that $L_\theta(xy) =  L_\theta(x) y  + x L_\theta(y) - J_\theta ^{-1} \delta_\theta(x) L_\theta(y) $, the above proof works for $L_\theta^k (x) J_\theta^{-\beta} \in {\mathcal L}_{d/\beta , \infty}$ as well. Moreover, using Proposition \ref{factorisation} and the H\"older inequality, we easily obtain the following ``two-sided" variant of Lemma \ref{delta_cwikel_estimate}:
    \begin{cor}\label{two_sided_cwikel}
        Let $x \in \cA$ and $k\geq 0$. Then for all $\gamma,\beta > 0$ we have:
        \begin{equation*}
            J_\theta^{-\gamma}  L_\theta^k( x)J_\theta^{-\beta} \in {\mathcal L}_{\frac{d}{\beta+\gamma},\infty}  ,\quad   J_\theta^{-\gamma}\delta_\theta^k( x)J_\theta^{-\beta} \in {\mathcal L}_{\frac{d}{\beta+\gamma},\infty}.
        \end{equation*}
    \end{cor}

\subsection{The case $0\leq \alpha \leq \beta+1$}
    For $\zeta \in (0,1)$, taking $\eta = 1$ in \eqref{best_integral} yields
    \begin{equation*}
        s^{-\zeta} = \frac{1}{{\mathrm B}(\zeta,1-\zeta)}\int_0^\infty \frac{1}{\lambda^{\zeta}(\lambda+s)}\,d\lambda.
    \end{equation*}
    If $\alpha = 1-\zeta$, we get the useful identity for $\xi \in {\mathcal S}({\mathbb R}^d)$
    \begin{equation}\label{fractional_power_integral}
        J_\theta^{\alpha}\,\xi = \frac{1}{{\mathrm B} (1-\alpha,\alpha)}\int_0^\infty \lambda^{\alpha-1}\frac{J_\theta}{\lambda+J_\theta}\,\xi\,d\lambda,
    \end{equation}
    where the integrand on the right converges in $L_2({\mathbb R}^d)$, as in the proof of Theorem \ref{L_to_delta}.
    
    The following is the $\alpha \in [0,1)$ and $\beta\geq 0$ case of Theorem \ref{main-commutator}:
    \begin{thm}\label{alpha_lt_one}
        Let $x \in \cA$. Let $\alpha \in [0,1)$ and $\beta \geq 0$ then for all $k\geq 0$
        \begin{equation*}
            [J_\theta^{\alpha},\delta_\theta^k(x)]J_\theta^{-\beta} \in {\mathcal L}_{\frac{d}{\beta-\alpha+1},\infty}.
        \end{equation*}
    \end{thm}
    \begin{proof}
       It follows from \eqref{fractional_power_integral} that for $\xi \in {\mathcal S}({\mathbb R}^d)$,
        \begin{align*}
            [J_\theta^{\alpha},\delta_\theta^k( x)]\,\xi &= \frac{1}{{\mathrm B}(1-\alpha,\alpha)}\int_0^\infty \lambda^{\alpha-1}\left[\frac{J_\theta}{\lambda+J_\theta},\delta_\theta^k( x)\right]\,\xi\,d\lambda\\
                                         &= \frac{1}{{\mathrm B}(1-\alpha,\alpha)}\int_0^\infty \lambda^{\alpha}(\lambda+J_\theta)^{-1}[J_\theta,\delta_\theta^k( x)](\lambda+J_\theta)^{-1}\,\xi\,d\lambda\\
                                         &= \frac{1}{{\mathrm B}(1-\alpha,\alpha)}\int_0^\infty \lambda^{\alpha}(\lambda+J_\theta)^{-1}\delta_\theta^{k+1}( x)(\lambda+J_\theta)^{-1}\,\xi\,d\lambda\\
                                         &= \frac{1}{{\mathrm B}(1-\alpha,\alpha)}\int_0^\infty \lambda^{\alpha}(\lambda+J_\theta)^{-2}\delta_\theta^{k+1}( x)\,\xi\,d\lambda\\
                                         &\quad - \frac{1}{{\mathrm B}(1-\alpha,\alpha)}\int_0^\infty \lambda^{\alpha}(\lambda+J_\theta)^{-1}  [(\lambda+J_\theta)^{-1},\delta_\theta^{k+1}( x)]\,\xi\,d\lambda\\
                                         &= \frac{1}{{\mathrm B}(1-\alpha,\alpha)}\int_0^\infty \lambda^{\alpha}(\lambda+J_\theta)^{-2}\delta_\theta^{k+1}( x)\xi\,d\lambda\\
                                         &\quad + \frac{1}{{\mathrm B}(1-\alpha,\alpha)}\int_0^\infty \lambda^{\alpha}(\lambda+J_\theta)^{-2}\delta_\theta^{k+2}( x)(\lambda+J_\theta)^{-1}\,\xi\,d\lambda.
        \end{align*}       
Since $J_\theta^{-\beta}$ maps ${\mathcal S}({\mathbb R}^d)$ into ${\mathcal S}({\mathbb R}^d)$, using the identity $\int _0^\infty \la^\al\frac{t^{1-\al}}{(\la +t )^2} d\la = {\mathrm B}(1-\al , 1+\al )$ which is easily deduced from \eqref{best_integral} again, we have
        \begin{align*}
            [J_\theta^{\alpha},\delta_\theta^k( x)]J_\theta^{-\beta} \,\xi &= \al\, J_\theta^{\alpha-1}\delta_\theta ^{k+1}( x)J_\theta^{-\beta }\,\xi\\
                                                   &\quad + \frac{1}{{\mathrm B}(1-\alpha,\alpha)}\int_0^\infty \lambda^{\alpha}(\lambda+J_\theta)^{-2}\delta_0^{k+2}( x)J_\theta^{-\beta}(\lambda+J_\theta)^{-1}\,\xi\,d\lambda.\\
                                                   &= \al\, J_\theta^{\alpha-1}\delta_\theta^{k+1}( x)J_\theta^{-\beta} \,\xi \\
                                                   &\quad + \frac{1}{{\mathrm B}(1-\alpha,\alpha)}\int_0^\infty \lambda^{\alpha}\frac{J_\theta^{1-\alpha}}{(\lambda+J_\theta)^2}J_\theta^{\alpha-1}\delta_\theta^{k+2}( x)J_\theta^{-\beta}\frac{1}{\lambda+J_\theta}\,\xi\,d\lambda.
        \end{align*}
        The operator $J_\theta^{\alpha-1}\delta_\theta^{k+1}( x)J_\theta^{-\beta} $ is in ${\mathcal L}_{\frac{d}{\beta-\alpha+1},\infty}$ due to Corollary \ref{two_sided_cwikel}. The second summand is treated in the following.
        
        Assume initially that $\frac{d}{\beta-\alpha+1} > 1$, or equivalently $\alpha < \beta+1< d+\al $. Under this condition, 
        the ideal ${\mathcal L}_{\frac{d}{\beta-\alpha+1},\infty}$ can be given a norm and we can estimate the second summand using the triangle inequality.
        We have
        \begin{equation}\label{alpha_estimate_1}
            \left\|\frac{J_\theta^{1-\alpha}}{(\lambda+J_\theta)^2}\right\|_{\infty} \leq \sup_{t \geq 1} \frac{t^{1-\alpha}}{(t+\lambda)^2}= \begin{cases}
                                                                                \frac{1}{(1+\lambda)^2},\quad \lambda \leq \frac{\alpha+1}{1-\alpha}\\
                                                                                 \frac {C_{\alpha}}{\lambda^{ \alpha+1}}\quad \lambda > \frac{\alpha+1}{1-\alpha}\,,
                                                                        \end{cases}
        \end{equation}
        for a certain constant $C_{\alpha}$. Thus,
        \begin{equation}\label{alpha_estimate_2}
            \left\| \lambda^{\alpha}\frac{J_\theta^{1-\alpha}}{(J_\theta+\lambda)^2}\right\|_\infty  \left\|\frac{1}{ \lambda+J_\theta }\right\|_{\infty}  \leq  \begin{cases}
                                                                                \frac{1}{(1+\lambda)^3},\quad \lambda \leq \frac{\alpha+1}{1-\alpha}\\
                                                                                 \frac{C_{\alpha}}{\lambda(1+\lambda)}\quad \lambda > \frac{\alpha+1}{1-\alpha}\,,
                                                                        \end{cases}
        \end{equation}
which is integrable. If $\alpha < \beta+1$, we get from the triangle inequality that
        \begin{equation*}
            [J_\theta^{\alpha},\delta_\theta^k( x)]J_\theta^{-\beta}\in {\mathcal L}_{\frac{d}{\beta-\alpha+1},\infty}.
        \end{equation*}
        Thus the result is proved if $\beta < d+\alpha-1$. In particular, since $d\geq 2$ we have proved the result for $0 < \beta \leq 1$.
        
        To complete the proof, we need an induction argument as in the proof of Lemma \ref{delta_cwikel_estimate}. Note first that by the assumed factorisation property of $\cA$, for any $x \in \cA$ we can write $x$ as a linear combination of products, $x = \sum_{j=1}^n y_jz_j$ where each $y_j,z_j  \in\cA$. Suppose that $\beta > 0$ is such that $[J_\theta^{\alpha},\delta_\theta^k( x)]J_\theta^{-\beta} \in {\mathcal L}_{\frac{d}{\beta-\alpha+1},\infty}$ for all $k\geq 0$ and all $x \in \cA$. Then applying the Leibniz rule to the $j$th summand, we have:
        \begin{align*}
            J_\theta^{-1}[J_\theta^{\alpha},\delta_\theta^k( y_jz_j)]J_\theta^{-\beta} &=  \sum_{l=0}^k \binom{k}{l}J_\theta^{-1}[J_\theta^{\alpha},\delta_\theta^{k-l}(y_j)\delta_\theta ^l(z_j)]J_\theta^{-\beta}\\
                                                         &=  \sum_{l=0}^k \binom{k}{l}J_\theta^{-1}\delta_\theta^{k-1}(y_j)[J_\theta^{\alpha},\delta_\theta^l(z_j)]J_\theta^{-\beta}\\
                                                         &\quad +  \sum_{l=0}^k \binom{k}{l}J_\theta^{-1}[J_\theta^{\alpha},\delta_\theta^{l}(y_j)]\delta_\theta^{k-l}(z_j)J_\theta^{-\beta}.
        \end{align*}
        Then applying the the H\"older inequality, we have
        \begin{equation}\label{midway_result}
            J_\theta^{-1}[J_\theta^{\alpha},\delta_\theta^k( x)]J_\theta^{-\beta} \in {\mathcal L}_{\frac{d}{1+\beta-\alpha+1},\infty}.
        \end{equation}
        Now we complete the proof by showing that if the required assertion holds for $\beta$, then it holds for $\beta+1$. Indeed,
        \begin{align*}
            [J_\theta^{\alpha},\delta_\theta^k( x)]J_\theta^{-\beta-1} &= [J_\theta^{-1},[J_\theta^{\alpha},\delta_\theta^k( x)]]J_\theta^{-\beta} + J_\theta^{-1}[J_\theta^{\alpha},\delta_\theta^k( x)]J_\theta^{-\beta}\\
                                                   &= -J_\theta^{-1}[J_\theta^{\alpha},\delta_\theta^{k+1}( x)]J^{-\beta}J^{-1}  +J_\theta^{-1}[J_\theta^{\alpha},\delta_\theta^k( x)]J_\theta^{-\beta}.
        \end{align*}
        From \eqref{midway_result}, we conclude that
        \begin{align*}
            [J_\theta^{\alpha},\delta_\theta^k( x)]J_\theta^{-\beta-1} \in {\mathcal L}_{\frac{d}{\beta+1-\alpha+1},\infty}.
        \end{align*}
        Hence the assertion holds for all $\beta > 0$.
    \end{proof}
    
 The cases where $\alpha \geq 1$ are handled by induction on $\al$:
 
    \begin{cor}\label{inducted_alpha}
        Let $x \in \cA$. Let $\alpha \geq 0$, $\beta \geq 0$ satisfy $\alpha < \beta+1$. Then for all $k\geq 0$ we have 
        \begin{equation*}
            [J_\theta^{\alpha},\delta_\theta^k( x)]J_\theta^{-\beta} \in {\mathcal L}_{\frac{d}{\beta-\alpha+1},\infty}.
        \end{equation*}
    \end{cor}    
    \begin{proof}
        The case $\alpha\leq 1$ is provided by Theorem \ref{alpha_lt_one}. We proceed by induction.
        Fix $\alpha\geq 0$ Suppose that the claim is true for all $k\geq 0$ and $\beta > \alpha-1$. Now let $\beta > \alpha$. Then using the Leibniz rule and Lemma \ref{delta_cwikel_estimate}
        \begin{align*}
            [J_\theta^{\alpha+1},\delta_\theta^k(x)]J_\theta^{-\beta} &= J_\theta^{\alpha}[J_\theta,\delta_\theta^k(x)]J_\theta^{-\beta}+[J_\theta^{\alpha},\delta_\theta^k( x)]J_\theta^{1-\beta}\\
                                                    &=[J_\theta^{\alpha},[J_\theta,\delta_\theta^k( x)]]J_\theta^{-\beta}+[J_\theta,\delta_\theta^k( x)]J_\theta^{\alpha-\beta}+[J_\theta^{\alpha},\delta_\theta^k( x)]J_\theta^{1-\beta}\\
                                                    &= [J_\theta^{\alpha},\delta_\theta^{k+1}( x)]J_\theta^{-\beta}+\delta_\theta^{k+1}( x)J_\theta^{\alpha-\beta}+[J_\theta^{\alpha},\delta_\theta^k( x)]J_\theta^{1-\beta}\\
                                                    &\in {\mathcal L}_{\frac{d}{\beta-\alpha+1},\infty} + {\mathcal L}_{\frac{d}{\beta-\alpha},\infty} + {\mathcal L}_{\frac{d}{\beta-1-\alpha+1},\infty}\\
                                                    &= {\mathcal L}_{\frac{d}{\beta-\alpha},\infty},
        \end{align*}
        thus proving the claim for $\alpha+1$.
    \end{proof}
    
    Using the triangle inequality holds for the operator norm in place of the ${\mathcal L}_{\frac{d}{\beta-\alpha+1},\infty}$ norm, the first part of the proof of Theorem \ref{alpha_lt_one} can easily be adapted to the case $0\leq \al = \beta+1 $.
    \begin{thm}\label{bounded_case}
        Let $x \in \qcS$, and $\alpha \geq 0$. Then for all $k\geq 0$ the operator:
        \begin{equation*}
            [J_\theta^{\alpha},\delta_\theta^k( x)]J_\theta^{-\alpha+1}
        \end{equation*}
        has bounded extension.
    \end{thm}
    \begin{proof}
        Beginning with the integral formula from the proof of Theorem \ref{alpha_lt_one}, we have
        \begin{align*}
            J_\theta^{1-\alpha}[J_\theta^{\alpha},\delta_\theta^k( x)]    &= \al \delta_\theta^{k+1}( x)  + \frac{1}{{\mathrm B}(1-\alpha,\alpha)}\int_0^\infty \lambda^{\alpha}\frac{J_\theta^{1-\alpha}}{(J_\theta+\lambda)^2}\delta_\theta^{k+2}( x)\frac{1}{ \lambda+J_\theta  }\,d\lambda.
        \end{align*}
        Thus since $\delta_\theta^{k+1}(x)$ and $\delta_\theta^{k+2}(x)$ are bounded (Corollary \ref{QC_infty}), we can use the triangle inequality for operator norm and the estimates \eqref{alpha_estimate_1} and \eqref{alpha_estimate_2} from the proof
        of Theorem \ref{alpha_lt_one} to conclude that
        \begin{equation*}
            J_\theta^{1-\alpha}[J_\theta^{\alpha},\delta_\theta^k( x)]
        \end{equation*}
        has bounded extension. Taking the adjoint yields the result.
    \end{proof}

\subsection{Proof of Theorem \ref{main-commutator}}

    So far, we have established that Theorem \ref{main-commutator} holds in the following cases
    \begin{equation*}
        0 \leq \alpha \leq \beta+1.
    \end{equation*}
    Indeed, Corollary \ref{inducted_alpha} and Theorem \ref{bounded_case} imply an even stronger statement: for all $k\geq 0$, we have that
    \begin{equation}\label{main_statement}
        \begin{cases}
                                                                                [J_\theta^{\alpha},\delta_\theta^{k}( x)]J_\theta^{-\beta} \in  {\mathcal L}_{\frac{d}{\beta-\alpha+1},\infty} ,\quad \mbox{if}\quad 0\leq \al<\beta +1,\\
                                                                                [J_\theta^{\alpha},\delta_\theta^{k}( x)]J_\theta^{-\beta}\quad \mbox{has bounded extension, if}\quad  0\leq \al=\beta +1\,.
                                                                        \end{cases}
    \end{equation}

    We can conclude the proof by showing that if \eqref{main_statement} holds for $(\alpha,\beta)$ and all $k\geq 0$ then it holds for $(\alpha-1,\beta-1)$ and all $k\geq 0$. This will complete the proof, since for any $\alpha<\beta+1$ we can find $n$ large enough such that $0\leq \alpha+n \leq \beta+n+1$ and hence \eqref{main_statement} holds for $(\alpha+n,\beta+n)$ and all $k\geq 0$.
    
    To this end, suppose that \eqref{main_statement} holds for some $(\alpha,\beta)$ where $  \alpha\leq \beta+1$ and for all $k\geq 0$. From the Leibniz rule, we derive
    \begin{align*}
        [J_\theta^{\alpha-1},\delta^k_\theta( x)]J_\theta^{1-\beta} &= [J_\theta^{\alpha},\delta_\theta^k( x)]J_\theta^{-\beta} + J_\theta^{\alpha}[J_\theta^{-1},\delta_\theta^k( x)]J_\theta^{1-\beta}\\
                                                  &= [J_\theta^{\alpha},\delta_\theta^k( x)]J_\theta^{-\beta} - J_\theta^{\alpha-1}\delta_\theta^{k+1}( x)J_\theta^{-\beta}\\
                                                  &= [J_\theta^{\alpha},\delta_\theta^k( x)]J_\theta^{-\beta} - J_\theta^{-1}[J_\theta^{\alpha},\delta_\theta^{k+1}( x)]J_\theta^{-\beta}\\
                                                  &\quad -J_\theta^{-1}\delta_\theta^{k+1}( x)J_\theta^{\alpha-\beta}\\
                                                  &= [J_\theta^{\alpha},\delta_\theta^k( x)]J_\theta^{-\beta} - J_\theta^{-1}[J_\theta^{\alpha},\delta_\theta^{k+1}( x)]J_\theta^{-\beta}\\
                                                  &\quad -[J_\theta^{-1},\delta_\theta^{k+1}( x)]J_\theta^{\alpha-\beta}-\delta_\theta^{k+1}( x)J_\theta^{\alpha-\beta-1}\\
                                                  &= [J_\theta^{\alpha},\delta_\theta^k( x)]J_\theta^{-\beta} - J_\theta^{-1}[J_\theta^{\alpha},\delta_\theta^{k+1}( x)]J_\theta^{-\beta}\\
                                                  &\quad +J_\theta^{-1}\delta_\theta^{k+2}( x)J_\theta^{\alpha-\beta-1}-\delta_\theta^{k+1}( x)J_\theta^{\alpha-\beta-1}.
    \end{align*}
    Since $\alpha\leq \beta+1$, it follows from Lemma \ref{delta_cwikel_estimate} that $[J_\theta^{\alpha-1},\delta^k_\theta( x)]J_\theta^{1-\beta}$ is in ${\mathcal L}_{\frac{d}{\beta-\alpha+1},\infty}$ if $\alpha < \beta+1$
    or ${\mathcal B}(L_{2}({\mathbb R}^d))$ if $\alpha = \beta+1$.

\begin{rk}
We close this section by some useful remarks.
\begin{enumerate}
\item    It is worth noting that if one continues the expansion in the proof of Theorem \ref{alpha_lt_one} we have the following expansion: for all $n\geq 1$ and $\alpha \in [0,1]$,
    \begin{align*}
        [J_\theta^{\alpha},\delta_\theta^k( x)] &= \sum_{j=1}^n \frac{{\mathrm B}(j-\alpha,1+\alpha)}{{\mathrm B}(1-\alpha,\alpha)}J_\theta^{\alpha-j}\delta_\theta^{k+j}( x)\\
                                    &\quad +\frac{1}{{\mathrm B}(1-\alpha,\alpha)}\int_0^\infty \lambda^{\alpha}(\lambda+J_\theta)^{-(n+1)}\delta_\theta^{k+n+1}( x)(\lambda+J_\theta)^{-1}\,d\lambda.
    \end{align*}
    Here the coefficients come from the choice of $\zeta = -\alpha$ and $\eta = j+1$ in \eqref{best_integral}.
\item Moreover one can easily deduce the ``two-sided" result that:
    \begin{equation}\label{triple_version}
        J_\theta^{-\gamma}[J_\theta^{\alpha},\delta_\theta^k( x)]J_\theta^{-\beta} \in {\mathcal L}_{\frac{d}{\beta+\gamma-\alpha+1},\infty}
    \end{equation}
    whenever $\alpha < \beta+\gamma+1$, and that the above operator has bounded extension whenever $\alpha = \beta+\gamma+1$. An easy way to see how \eqref{triple_version} follows from Theorem \ref{main-commutator} is to use the identity:
    \begin{equation*}
        J_\theta^{-\gamma}[J_\theta^{\alpha},\delta_\theta^k( x)]J_\theta^{-\beta} = [J_\theta^{\alpha-\gamma},\delta_\theta^k( x)]J_\theta^{-\beta}-[J_\theta^{-\gamma},\delta_\theta( x)]J_\theta^{\alpha-\beta}.
    \end{equation*}
\item     The generalisation to $\alpha,\beta \in {\mathbb C} $ with $\Re(\alpha)\leq \Re(\beta)+1$ is immediate.
\end{enumerate} 
\end{rk}

\section{Proofs of Theorems \ref{trace formula} and \ref{necessity}}\label{proofs of tf n}

As in Section \ref{commutator-estimates}, we consider the dense subalgebra $\cA \subset \qcS$ constructed in Proposition \ref{factorisation}.

Using Theorem \ref{main-commutator} and the commutator estimates developed in Section \ref{commutator-estimates}, we are able to establish the trace formula in Theorem \ref{trace formula}, and finally prove Theorem \ref{necessity}. This will be done by showing that for all $x \in \cA$
$$|\qd x |^d   - |A|^d   (1+{\mathcal D}^2)  ^{-d/2} \in {\mathcal L}_1$$
for a certain bounded operator $A$ on ${\mathbb C}^N\otimes L_2({\mathbb R}^d)$ (depending on $x$), and then applying the trace formula given by \cite[Theorem~6.15]{MSZ2018} to $ |A|^d   (1+{\mathcal D}^2)  ^{-d/2} $.

\subsection{Operator difference estimates}
We begin with the construction of the above mentioned operator $A$.
For $1\leq j, k \leq d$, denote $g_{j,k} (t)= \frac{t_j t_k }{|t|^2}$ on ${\mathbb R}^d$.
 Let $x\in \qcS$. Define the operator $A_j$ on $L_2({\mathbb R}^d)$ as
    \begin{equation}\label{def-Aj}
        A_j\xi :=  (\partial_j x)\xi   - \sum_{k=1}^d (M_{g_{j,k}}   \partial_k x)\xi   = (\partial_j x)\xi   - \sum_{k=1}^d  g_{j,k}(  {\mathcal D}_1,\cdots,   {\mathcal D}_d  )(\partial_k x)\xi,\quad \xi \in L_2({\mathbb R}^d)
    \end{equation}
 and define the operator $A$ on ${\mathbb C} ^N \ot L_2({\mathbb R}^d)$
 \be
 A:=   \sum_{j =1} ^d   \gamma_j  \ot A_j,
 \ee
 where $N$ and $\gamma_j$ are the same as in Definition \ref{Dirac}.
    

    The main result in this subsection is the following theorem:    
    \begin{thm}\label{qdx-A}
        Let $x \in \cA$. Then we have:
        \begin{equation*}
            |\qd x|^d-|A|^d(1+ {\mathcal D}^2)^{-d/2} \in {\mathcal L}_1.
        \end{equation*}
    \end{thm}
    
    Recall that ${\mathcal D} = \sum_{j=1}^d  \gamma_j \ot {\mathcal D}_j$, and 
    $\qd x =  \ri [\sgn({\mathcal D}), 1\ot x ]$. Let $g(t) = t(1+t^2) ^{-1/2}$ and write
    $$\qd x  = \ri [\sgn({\mathcal D}) - g({\mathcal D}) , 1\ot x ]+ \ri \sum_{j=1}^d     \gamma_j \ot [{\mathcal D}_j J_\theta^{-1}, x]  .$$
    By Lemma \ref{Cwikel-xp}, $[\sgn({\mathcal D}) - g({\mathcal D}) , 1\ot x ]$ belongs to ${\mathcal L}_{p}$ when $ p > \frac d 2 $. Define the auxiliary operator $\widetilde A_j$ for $1\leq j\leq d$ on $L_2({\mathbb R}^d)$ as
    \begin{equation}\label{def-Aj-aux}
        \widetilde  A_j :=  \partial_j x   - \sum_{k=1}^d {\mathcal D}_j {\mathcal D}_k J_\theta^{-2}  \partial_k x  \,.
    \end{equation}
    The following proposition connects the commutator $ [{\mathcal D}_j J_\theta^{-1}, x] $ with $\widetilde A_j$.

    \begin{prop}\label{initial_operator_difference}
        Let $1\leq j \leq d$, and $x \in \cA$. Then,
        \begin{equation*}
            [{\mathcal D}_j J_\theta^{-1}, x]- \widetilde  A_jJ_\theta^{-1} \in {\mathcal L}_{\frac{d}{2},\infty}.
        \end{equation*}
    \end{prop}    
    \begin{proof}
        From the Leibniz rule, we have
        \begin{equation*}
             [{\mathcal D}_j J_\theta^{-1},x] = \partial_j  x   J_\theta^{-1}+  {\mathcal D}_j[J_\theta^{-1},  x] =  \partial_j x J_\theta^{-1} -   {\mathcal D}_j J_\theta^{-1}\delta_\theta( x)J_\theta^{-1}.
        \end{equation*}
        Using the integral formula \eqref{deltaT} from Theorem \ref{L_to_delta}, we have for all $n\geq 0$,
        \begin{align*}
            \delta_\theta(x)J_\theta ^{-1} & = \sum_{j=1}^{n-1} \frac{1}{\pi}{\mathrm B}(j-1/2,3/2)J_\theta ^{1-j}L_\theta^j(x)J_\theta^{-1} \\
                                &\quad + \frac{1}{\pi}\int_0^\infty \lambda^{1/2}\frac{J_\theta^n}{(\lambda+J_\theta^2)^n}L_\theta^n(x)J_\theta^{-2}\frac{J_\theta}{\lambda+J_\theta^2}\,d\lambda.
        \end{align*}
        From Corollary \ref{two_sided_cwikel}, we have that $J_\theta^{1-j}L_\theta^j( x)J_\theta^{-1} \in {\mathcal L}_{d/j,\infty}$ for every $j\geq 1$. Due to a similar argument to the proof of Lemma \ref{L_to_delta}, we have that
        \begin{equation*}
            \int_0^\infty \lambda^{1/2} \frac{J_\theta^n}{(\lambda+J_\theta^2)^n}L_\theta^n(x)J_\theta^{-2}\frac{J_\theta}{\lambda+J_\theta^2}\,d\lambda \in {\mathcal L}_{\frac{d}{2},\infty}
        \end{equation*}
        provided $n$ is sufficiently large.
        So (recalling that ${\mathrm B}(\frac{1}{2},\frac{3}{2}) = \frac{\pi}{2}$) we obtain
        \begin{equation}\label{partial_difference}
            [{\mathcal D}_j  J_\theta^{-1},x] \in    \partial_j x  J_\theta^{-1}-\frac{1}{2}{\mathcal D}_jJ_\theta^{-1}L_\theta(x)J_\theta^{-1} + {\mathcal L}_{\frac{d}{2},\infty}.
        \end{equation}
        By the definition of $L_\theta$, we have:
        \begin{align*}
            {\mathcal D}_j  J_\theta ^{-1} L_\theta(x)J_\theta ^{-1} &= {\mathcal D}_j  J_\theta^{-2}[J_\theta^2,  x]J_\theta^{-1}\\
                                  &= {\mathcal D}_j  J_\theta^{-2}\sum_{k=1}^d [{\mathcal D}_k ^2,  x]J_\theta^{-1}\\
                                  &=  \sum_{k=1}^d {\mathcal D}_j J_\theta ^{-2}({\mathcal D}_k  \partial_kx +  \partial_k x \,{\mathcal D}_k) J_\theta^{-1}\\                                  
                                  &=  \sum_{k=1}^d {\mathcal D}_j J_\theta ^{-2}(2{\mathcal D}_k  \partial_kx -  \partial_{k}^2 x  ) J_\theta^{-1}
        \end{align*}
        From Corollary \ref{two_sided_cwikel}, we have $ {\mathcal D}_j J_\theta ^{-2}   \partial_{k}^2 x    J_\theta^{-1} \in {\mathcal L}_{d/2,\infty}$,
        and therefore 
        \begin{equation}\label{second_partial_difference}
            {\mathcal D}_j  J_\theta^{-1}L_\theta(x)J _\theta ^{-1} \in 2 \sum_{k=1}^d  {\mathcal D}_j {\mathcal D}_k J_\theta^{-2}  \partial_k x J_\theta^{-1} + {\mathcal L}_{d/2,\infty}.
        \end{equation}
        Combining \eqref{partial_difference} and \eqref{second_partial_difference} yields:
        \begin{equation*}
            [{\mathcal D}_jJ_{\theta}^{-1},x] \in \partial_j x J_{\theta}^{-1}-\sum_{k=1}^d {\mathcal D}_j{\mathcal D}_k J_{\theta}^{-2}\partial_k x J_{\theta}^{-1} + {\mathcal L}_{d/2,\infty} = \widetilde{A}_jJ_{\theta}^{-1}+{\mathcal L}_{d/2,\infty}
        \end{equation*}
        as was claimed.
    \end{proof}

    Let us also compare $ \widetilde  A_jJ_\theta^{-1}$ with $   A_jJ_\theta^{-1}$.
    
       \begin{prop}\label{initial_operator_difference-aux}
        Let $1\leq j \leq d$, and $x \in \cA$. Then,
        \begin{equation*}
            A_jJ_\theta^{-1}- \widetilde  A_jJ_\theta^{-1} \in {\mathcal L}_{\frac{d}{2},\infty}.
        \end{equation*}
    \end{prop}
    
    \begin{proof}
By definition, $  A_j =   \sum_{k=1}^d M_{ g_{j,k}} \partial_k x $ and $\widetilde  A_j =   \sum_{k=1}^d M_{\widetilde g_{j,k}} \partial_k x $ with $\widetilde g_{j,k} (t)= \frac{t_j t_k }{1+|t|^2}$. So we are reduced to estimating $M_{ g_{j,k}} \partial_k x J_\theta^{-1}  -  M_{\widetilde g_{j,k}} \partial_k x J_\theta^{-1}$ for every $k$. Using the factorisation of $x$ as a linear combination of products $yz$, $y,z \in \cA$ (Proposition \ref{factorisation}) and the Leibniz rule, we have
\be
M_{ g_{j,k}} \partial_k (yz) J_\theta^{-1}  -  M_{\widetilde g_{j,k}} \partial_k (yz) J_\theta^{-1}= (M_{ g_{j,k}}  -  M_{\widetilde g_{j,k}}) \partial_k y\, z  J_\theta^{-1}+ (M_{ g_{j,k}}  -  M_{\widetilde g_{j,k}}) y\,  \partial_k z  J_\theta^{-1}.
\ee
From Lemma \ref{delta_cwikel_estimate}, both $  z  J_\theta^{-1}$ and $ \partial_k z  J_\theta^{-1}$ belong to ${\mathcal L}_{d, \infty}$. On the other hand, one can easily check that $ g_{j,k}-\widetilde g_{j,k} \in L_p({\mathbb R}^d)$ as $p> \frac d 2$, which yields by Theorem \ref{Cwikel-type}(\ref{L_p cwikel}) that 
$$(M_{ g_{j,k}}  -  M_{\widetilde g_{j,k}})  y \in {\mathcal L}_p\subset {\mathcal L}_{d,\infty}, \quad (M_{ g_{j,k}}  - M_{\widetilde g_{j,k}}) \partial_k y \in {\mathcal L}_p\subset {\mathcal L}_{d,\infty} .$$
Thus it follows from the H\"older inequality that 
$$M_{ g_{j,k}} \partial_k x J_\theta^{-1}  - M_{\widetilde g_{j,k}} \partial_k x J_\theta^{-1} \in {\mathcal L}_{d/2, \infty} ,$$
 whence the proposition.
    \end{proof}

   For $g(t) = t(1+t^2) ^{-1/2}$ on ${\mathbb R}$, Propositions \ref{initial_operator_difference} and \ref{initial_operator_difference-aux} imply that
    \begin{equation}\label{gD-estimate}
        \ri[g({\mathcal D}),1\otimes  x] - A(1+{\mathcal D}^2)^{-1/2} \in {\mathcal L}_{\frac{d}{2},\infty}.
    \end{equation}
    This -- combined with Lemma \ref{Cwikel-xp} -- yields:
    \begin{equation*}
        \qd x - A(1+{\mathcal D}^2)^{-1/2} \in {\mathcal L}_{\frac{d}{2},\infty}
    \end{equation*}
    for all $x \in \cA$.
    
    \begin{lem}\label{symmetrised_difference}
        Let $x\in \cA$. We have
        \begin{equation*}
            |\qd x|^d-((1+{\mathcal D}^2)^{-1/2}|A|^2(1+{\mathcal D}^2)^{-1/2})^{d/2} \in {\mathcal L}_1.
        \end{equation*}
    \end{lem}    
    \begin{proof}
        We already know from Lemma \ref{Cwikel-xp} that $\ri[g({\mathcal D}),1\otimes x]-\qd x \in {\mathcal L}_{\frac{d}{2}}$, which together with \eqref{gD-estimate} ensures that
        \begin{equation*}
            \qd x-A(1+{\mathcal D}^2)^{-1/2} \in {\mathcal L}_{\frac{d}{2},\infty}.
        \end{equation*}
        Taking the adjoint:
        \begin{equation*}
            \qd x^*-(1+{\mathcal D}^2)^{-1/2}A^* \in {\mathcal L}_{\frac{d}{2},\infty}.
        \end{equation*}
        Recall that $\qd x\in {\mathcal L}_{d,\infty}$ by Theorem \ref{sufficiency} (as has been proved in Section \ref{section-sufficiency}), so it follows that $A(1+{\mathcal D}^2)^{-1/2} \in {\mathcal L}_{d,\infty}$.
        Using the H\"older inequality, we have
        \begin{align*}
            |\qd x|^2 - (1+ {\mathcal D}^2)^{-1/2}|A|^2(1+ {\mathcal D}^2)^{-1/2} &= \qd x^*\big(\qd x - A(1+  {\mathcal D}^2)^{-1/2}\big)\\
            &\quad +\big(\qd x^*-(1+ {\mathcal D}^2)^{-1/2}A^*\big) A(1+{\mathcal D}^2)^{-1/2}\\
                                                    &\in {\mathcal L}_{\frac{d}{3},\infty}  \subset {\mathcal L}_{\frac{5d}{12}} \,.
        \end{align*} 
        If $d=2$, then we are done. 
        
        Now assume that $d > 2$. We appeal to a recent result from E. Ricard \cite[Theorem 3.4]{Ricard2018}, which says that we can take a power $1/2$ to each term of the preceding inclusion to get
        \begin{equation*}
            |\qd x| - \Big((1+{\mathcal D}^2)^{-1/2}|A|^2(1+ {\mathcal D}^2)^{-1/2}\Big)^{1/2} \in {\mathcal L}_{\frac{5d}{6},\infty}.
        \end{equation*}
        Next we introduce a power $d$:
        \begin{align*}
            &|\qd x|^d-\Big((1+{\mathcal D}^2)^{-1/2}|A|^2(1+{\mathcal D}^2)^{-1/2}\Big)^{d/2} \\
                    &= \sum_{k=0}^{d-1} |\qd x|^{d-k-1}\Big(|\qd x|-\big((1+ {\mathcal D}^2)^{-1/2}|A|^2(1+ {\mathcal D}^2)^{-1/2}\big)^{1/2}\Big)\Big((1+{\mathcal D}^2)^{-1/2}|A|^2(1+{\mathcal D}^2)^{-1/2}\Big)^{\frac{k}{2}}\\
                                                                &\in \sum_{k=0}^{d-1} {\mathcal L}_{\frac{d}{d-k-1},\infty}\cdot {\mathcal L}_{\frac{5d}{6}}\cdot{\mathcal L}_{\frac{d}{k},\infty} \subset {\mathcal L}_{\frac{5d}{5d+1},\infty} \subset {\mathcal L}_1. \qedhere
        \end{align*} 
    \end{proof}

    By definition, $|A|^2 = A^*A$, so we can write $|A|^2$ as a polynomial in elements of $\cA$ and functions of $D_j$, $j=1,\ldots,d$. It then follows from Theorem \ref{main-commutator} that
    \begin{equation}\label{main_A_estimate}
        [|A|^2,(1+ {\mathcal D}^2)^{\alpha/2}](1+{\mathcal D}^2)^{-\beta/2} \in {\mathcal L}_{\frac{d}{\beta-\alpha+1},\infty}
    \end{equation}
    for all $\beta > 0$ and $\alpha < 1$. Therefore, if $d = 2$, letting $\al= -1$ and $\beta=1 $ in \eqref{main_A_estimate}, we have
    \begin{equation*}
        [|A|^2,(1+{\mathcal D}^2)^{-1/2}](1+{\mathcal D}^2)^{-1/2} \in {\mathcal L}_{2/3,\infty} \subset {\mathcal L}_1
    \end{equation*}
    This inclusion can be combined with Lemma \ref{symmetrised_difference} to arrive at
    \begin{equation*}
        |\qd x|^2-|A|^2(1+ {\mathcal D}^2)^{-1} \in {\mathcal L}_1
    \end{equation*}
    which completes the proof of Theorem \ref{qdx-A} for the $d=2$ case.       
    
    For $d>2$, we need 
    
    \begin{prop}\label{difference-d>2}
        Let $d > 2$. Then
        \begin{equation*}
            |A|^d(1+  {\mathcal D}^2)^{-d/2}-((1+ {\mathcal D}^2)^{-1/2}|A|^2(1+{\mathcal D}^2)^{-1/2})^{d/2} \in {\mathcal L}_1.
        \end{equation*}
    \end{prop}    
    \begin{proof}
        From \cite[Theorem B.1]{CLMSZ2018}, it suffices to show the following four conditions:
        \begin{enumerate}[{\rm (i)}]
            \item{}\label{con1} $|A|^{d-2}(1+  {\mathcal D}^2)^{1-\frac{d}{2}} \in {\mathcal L}_{\frac{d}{d-2},\infty}.$
            \item{}\label{con2} $(1+{\mathcal D}^2)^{-1/2}|A|^2(1+ {\mathcal D}^2)^{-1/2} \in {\mathcal L}_{\frac{d}{2},\infty}.$
            \item{}\label{con3} $[|A|^2(1+ {\mathcal D}^2)^{-1/2},(1+ {\mathcal D}^2)^{-1/2}] \in {\mathcal L}_{\frac{d}{2},1}.$
            \item{}\label{con4} $|A|^{d-2}[|A|^2,(1+{\mathcal D}^2)^{1-\frac{d}{2}}](1+ {\mathcal D}^2)^{-1} \in {\mathcal L}_1.$
        \end{enumerate}
        Since $d > 2$, we have that $|A|^{d-2} = |A|^{d-3}\sgn(A)A$, so \eqref{con1} follows immediately from Lemma \ref{delta_cwikel_estimate}. 
        Similarly using $|A|^2 = A^*A$, we get also get \eqref{con2} immediately from the H\"older inequality and the fact that $A(1+{\mathcal D}^2)^{-1/2}$ and its adjoint operator belong to $ {\mathcal L}_{d,\infty}$.
        
        For \eqref{con3}, we write:
        \begin{align*}
            [|A|^2(1+ {\mathcal D}^2)^{-1/2},(1+{\mathcal D}^2)^{-1/2}] = [|A|^2,(1+{\mathcal D}^2)^{-1/2}](1+{\mathcal D}^2)^{-1/2}
        \end{align*}
        which is in ${\mathcal L}_{\frac{2d}{5},\infty}$ due to \eqref{main_A_estimate} (with $\alpha=-1$ and $\beta=1$). Since $\frac{2d}{5} < \frac{d}{2}$, it follows that ${\mathcal L}_{2d/5,\infty}\subset {\mathcal L}_{d/2,1}$
        and this proves \eqref{con3}.
        Finally, \eqref{con4} immediately follows from \eqref{main_A_estimate} with $\alpha = 2-d$ and $\beta = 2$.
    \end{proof}

Lemma \ref{symmetrised_difference} and Proposition \ref{difference-d>2} yield Theorem \ref{qdx-A} for the case $d>2$, and thus complete the proof of Theorem \ref{qdx-A}.

\subsection{Proof of Theorem \ref{trace formula}}
Let us quote \cite[Theorem~6.15]{MSZ2018} in the following. Let $C_0 (\qe) $ be the norm closure of $\qcS$ in ${\mathcal B}(L_2({\mathbb R}^d))$. For every $g\in C(\mathbb{S}^{d-1})$, as defined in \eqref{multiplication}, $g\big(   \frac{\ri \nabla_\theta}{(-{\mathcal D}elta_\theta)^{1/2}})$ is the multiplication operator $\xi(t)\mapsto g(\frac{t}{|t|}) \xi(t)$ in ${\mathcal B}(L_2({\mathbb R}^d))$. Moreover,  all $g\big(   \frac{\ri \nabla_\theta}{(-{\mathcal D}elta_\theta)^{1/2}})$ with $g\in C(\mathbb{S}^{d-1})$ form a commutative $C^*$-subalgebra of ${\mathcal B}(L_2({\mathbb R}^d))$. Set $\Pi(C_0 (\qe)+{\mathbb C},  C(\mathbb{S}^{d-1}))$ to be the $C^*$-subalgebra of ${\mathcal B}(L_2({\mathbb R}^d))$ generated by $C_0 (\qe)+{\mathbb C}$ and all those $g\big(   \frac{\ri \nabla_\theta}{(-{\mathcal D}elta_\theta)^{1/2}})$'s. Theorem 3.3 of \cite{MSZ2018} implies that there exists a unique norm-continuous $*$-homomorphism
$${\rm{sym}}: \Pi(C_0 (\qe)+{\mathbb C},  C(\mathbb{S}^{d-1}))\longrightarrow  \big(C_0 (\qe)+{\mathbb C}\big) \ot_{\min}  C(\mathbb{S}^{d-1})$$
 which maps $x \in C_0(\qe)$ to $x\otimes 1$ and $g\big(\frac{\ri \nabla_{\theta}}{(-{\mathcal D}elta_{\theta})}\big)$ to $1\otimes g$. Then \cite[Theorem~6.15]{MSZ2018} says that for every continuous normalised trace $\vf$ on ${\mathcal L}_{1,\infty}$, every $x\in W^d_1(\qe)$, and every $T \in \Pi(C_0 (\qe)+{\mathbb C},  C(\mathbb{S}^{d-1}))$, we have
\beq\label{trace-formula-MSX}
\vf (T x (1-{\mathcal D}elta_\theta)   ^{-d/2}  )  = C_d   \Big(   \tau_\theta \ot \int_{\mathbb{S}^{d-1}}      \Big)   \big(  {\rm{sym}}(T) (x \ot 1)   \big)
 \eeq
where $C_d$ is a certain constant depending only on the dimension $d$.

Now we are able to give the proof of Theorem \ref{trace formula}.

\begin{proof}[Proof of Theorem \ref{trace formula}]
We will assume initially that $x\in \cA$.
For a continuous normalised trace $\vf $ on ${\mathcal L}_{1, \infty}$, Theorem \ref{qdx-A} ensures that 
$$\vf (|\qd x|^d  )  =  \vf \big(  |A|^d   (1+ {\mathcal D}^2) ^{-d/2} \big) . $$
But since $A= \sum_j \gamma_j \ot A_j$ self-adjoint unitary matrices $\gamma_j$, the only part that contributes to the trace on the right hand side above is $ (1\ot  \sum_j A_j^* A_j  )^{d/2} (1+ {\mathcal D}^2) ^{-d/2}$. Hence,
$$\vf (|\qd x|^d  )  =  \vf \big(  ( \sum_j A_j^* A_j  )^{d/2}  (1-{\mathcal D}elta_\theta ) ^{-d/2} \big) . $$

However, note that each $A_j$ is a linear combination of operators of multiplication by a function $x \in \qcS$ and Fourier multiplication by a function $g \in C(\mathbb{S}^{d-1})$, and so is in the algebra $\Pi(C_0(\qe)+{\mathbb C},C(\mathbb{S}^{d-1}))$, with symbol:
\begin{equation*}
    {\rm sym}(A_j) = \partial_j x\otimes 1- \sum_{k=1}^d s_js_k\otimes \partial_k x.
\end{equation*}
Since ${\rm{sym}}$ is a norm-continuous $*$-homomorphism, we have 
$${\rm{sym}}  ( \sum_j A_j^* A_j  )^{d/2}  = \Big( \sum_{j=1}^d\big| \partial_j x - s_j \sum_{k=1}^d  s_k \partial_k x\big|^2 \Big)^{d/2}.$$
Since $d\geq 2$, we can write:
\begin{equation*}
    \left(\sum_j A_j^*A_j\right)^{d/2} = \left(\sum_j A_j^*A_j\right)^{(d-2)/2}(\sum_{j} A_j^*A_j).
\end{equation*}

Recalling the definition of $A_j$,
\begin{equation*}
    A_j = \partial_j x + \sum_{k=1}^d\frac{{\mathcal D}_j{\mathcal D}_k}{-{\mathcal D}elta_{\theta}}\partial_k x
\end{equation*}
We arrive at:
\begin{equation*}
    \left(\sum_j A_j^*A_j\right)^{d/2} = \left(\sum_j A_j^*A_j\right)^{(d-2)/2}\sum_{j=1}^d A_j^*(\partial_j x - \sum_{k=1}^d\frac{{\mathcal D}_j{\mathcal D}_k}{-{\mathcal D}elta_\theta} \partial_k x).
\end{equation*}
Since each $\partial_j x$ is in $W^{d}_1(\qe)$, we can apply \eqref{trace-formula-MSX} to arrive finally at:
\begin{align*} 
    \vf \big(   &( \sum_j A_j^* A_j  )^{d/2} (1-{\mathcal D}elta_\theta )^{-d/2} \big)\\ 
                 &= C_d  \Big(  \tau _\theta  \ot \int_{\mathbb{S}^{d-1}}ds\Big) ({\rm{sym}} ( \sum_j A_j^* A_j  )^{(d-2)/2})(\sum_{j=1}^d {\rm sym}(A_j)^*(\partial_j x-s_j\sum_{k=1}^d s_k\partial_k x)))\\
                &= C_d\int_{\mathbb{S}^{d-1}} \tau_\theta(\left(\sum_{j=1}^d\big| \partial_j x - s_j \sum_{k=1}^d  s_k \partial_k x\big|^2 \right)^{d/2})\,ds. 
\end{align*}


By virtue of Corollary \ref{final_approximation_corollary}, the general case of Theorem \ref{trace formula} is done via an approximation argument, identically to the proof of \cite[Theorem 1.2]{MSX2018}.
\end{proof}

\subsection{Proof of Theorem \ref{necessity}}

Finally, we prove Theorem \ref{necessity}.
%
%
%
%

%

Recall from Theorem \ref{sufficiency} that when $y \in \qcS$ we have $\qd y \in {\mathcal L}_{d,\infty}$. Then if $x \in L_\infty(\qe)$, we have $(\qd y)x \in {\mathcal L}_{d,\infty}$. 
The following lemma shows that $(\qd y)x \in {\mathcal L}_{d,\infty}$ for certain unbounded $x\in L_d(\qe)$. Note that in the strictly noncommutative case of $\det(\theta)\neq 0$, the following lemma
is unnecessary as then we would have $L_d(\qe)\subset L_\infty(\qe)$. 
\begin{lem}\label{qdy_x_lemma}
    Let $d > 2$, and take $x \in L_d(\qe)$ and $y \in \qcS$. Then $(\qd y)x$ has extension in the ideal ${\mathcal L}_{d,\infty}$, with a quasi-norm bound
    \begin{equation*}
        \|(\qd y)x\|_{d,\infty} \lesssim_{d} \|x\|_d\|y\|_{W^{1}_\infty}.
    \end{equation*}
\end{lem}
\begin{proof}
    On the dense subspace $C^\infty_c({\mathbb R}^d)$, the operator of multiplication by $x$ is meaningful, and since $\qd y$ is bounded, the operator $(\qd y)x$ is well-defined on the subspace $C^\infty_c({\mathbb R}^d)$. Let us show
    that there is a bounded extension in ${\mathcal L}_{d,\infty}$. Applying the Leibniz rule:
    \begin{align*}
        -\ri(\qd y)x &= [\sgn({\mathcal D})-{\mathcal D} J_{\theta}^{-1},y]x+ [{\mathcal D} J_{\theta}^{-1},y]x\\
                     &= (\sgn({\mathcal D})-{\mathcal D} J_{\theta}^{-1})yx-y(\sgn({\mathcal D})-{\mathcal D} J_{\theta}^{-1})x + [{\mathcal D},y]J_{\theta}^{-1}x + {\mathcal D}[J_{\theta}^{-1},y]x\\
                     &= (\sgn({\mathcal D})-{\mathcal D} J_{\theta}^{-1})yx-y(\sgn({\mathcal D})-{\mathcal D} J_{\theta}^{-1})x + [{\mathcal D},y]J_{\theta}^{-1}x- {\mathcal D} J_{\theta}^{-1} [J_{\theta},y]J_{\theta}^{-1}x.
    \end{align*}
    We know from Corollary \ref{QC_infty} that $[J_{\theta},y]$ has bounded extension, and since $[{\mathcal D},y] = \sum_{j=1}^d -\ri \gamma_j\otimes \partial_jy$,
    the commutator $[{\mathcal D},y]$ has bounded extension.  
    
    Let us first bound the terms $[{\mathcal D},y]J_{\theta}^{-1}x$ and $[J_\theta,y]J_{\theta}^{-1}x$. Since $d > 2$, we may apply Lemma \ref{Cwikel-type-lem} to obtain:
    \begin{equation*}
        \|[{\mathcal D},y]J_{\theta}^{-1}x\|_{d,\infty} \leq \|[{\mathcal D},y]\|\|J_\theta^{-1}x\|_{d,\infty} \lesssim_d \|y\|_{\dot{W}^1_\infty}\|x\|_d.
    \end{equation*}
    
    To bound $[J_\theta,y]J_\theta^{-1}x$, we use the fact that:
    \begin{equation*}
        J_{\theta}-{\mathcal D} = \frac{1}{J_{\theta}+{\mathcal D}}
    \end{equation*}
    is bounded, so again applying Lemma \ref{Cwikel-type-lem}, it follows that:
    \begin{equation}\label{J_T,y}
        \|[J_{\theta},y]\| \lesssim_{d} \|y\|_\infty+\|[{\mathcal D},y]\| \leq \|y\|_{W^1_\infty}.
    \end{equation}
    Thus,
    \begin{equation*}
        \|[J_\theta,y]J_\theta^{-1}x\|_{d,\infty}\lesssim_{d} \|y\|_{W^1_\infty}\|x\|_d.
    \end{equation*}
        
    Denoting $h({\mathcal D}) := \sgn({\mathcal D})- {\mathcal D} J_{\theta}^{-1}$, we have so far:
    \begin{equation}\label{qdy_x}
         \|(\qd y)x\|_{d,\infty} \lesssim_{d}  \|h({\mathcal D})yx\|_{d,\infty}+\|yh({\mathcal D})x\|_{d,\infty} + \|y\|_{W^1_\infty}\|x\|_d.
    \end{equation} 
    As was already noted in the proof of Lemma \ref{Cwikel-xp}, we can write $h({\mathcal D}) := \sum_{j=1}^d \gamma_j\otimes h_j(\ri \nabla_\theta)$
    where:
    \begin{equation*}
        h_j(t) = \frac{t_j}{|t|(1+|t|^2)^{1/2}(|t|+(1+|t|^2)^{1/2})}, \quad 1\leq j\leq d.
    \end{equation*}
    Thus,
    \begin{equation*}
        \sup_{t \in {\mathbb R}^d} |h_j(t)|(1+|t|^2) < \infty.
    \end{equation*}
    It follows that $h({\mathcal D})J_\theta$ has bounded extension. Lemma \ref{Cwikel-type-lem} then yields
    \begin{align}
        \|yh({\mathcal D})x\|_{d,\infty} &\leq \|y\|_\infty\|h({\mathcal D})J_\theta\|_\infty\|J_{\theta}^{-1}x\|_{d,\infty}\nonumber\\
                               &\lesssim_d \|y\|_\infty\|x\|_d\label{yhDx}.
    \end{align}
    Similarly,
    \begin{equation*}
        \|h({\mathcal D})yx\|_{d,\infty} = \|h({\mathcal D})J_{\theta}J_\theta^{-1}yJ_{\theta}J_\theta^{-1}x\|_{d,\infty} \lesssim_{d} \|J_\theta^{-1}yJ_\theta\|\|J_\theta^{-1}x\|_{d,\infty} 
    \end{equation*}
    We can write $J_{\theta}^{-1}yJ_{\theta}$ as:
    \begin{equation*}
        J_{\theta}^{-1}yJ_{\theta} = -J_{\theta}^{-1}[J_\theta,y]+y
    \end{equation*}
    Applying \eqref{J_T,y} again allows us to bound the norm of the above by $\|y\|_{W^1_\infty}$, so we arrive at
    the quasinorm bound:
    \begin{equation}\label{hDyx}
        \|h({\mathcal D})yx\|_{d,\infty} \lesssim_{d} \|y\|_{W^1_\infty}\|x\|_d.
    \end{equation}
    
%
%
    Combining \eqref{yhDx}, \eqref{hDyx} and \eqref{J_T,y} with \eqref{qdy_x} yields $\|(\qd y)x\|_{d,\infty} \lesssim_{d} \|x\|_d\|y\|_{W^1_\infty}$ as desired.
\end{proof}


Before proceeding to the proof of Theorem \ref{necessity}, we make the following remark concerning integration of operator-valued functions. Let $\psi \in {\mathcal S}({\mathbb R}^d)$, and let $x \in W^{1}_d(\qe)$. Then (formally), one has: 
\begin{equation}\label{formal_bochner_computation}
    \|\qd (\psi\ast x)\|_{d,\infty} = \left\|\int_{{\mathbb R}^d} \psi(t)\qd(T_{-t}(x))\,dt\right\|_{d,\infty} \leq \|\psi\|_1\sup_{t \in {\mathbb R}^d} \|\qd (T_{-t}(x))\|_{d,\infty}.
\end{equation}
This formal computation is justified by the continuity of the mapping $t\mapsto T_{-t}(x)$ in the $W^{1}_d(\qe)$ norm (Theorem \ref{continuity_of_translation}), which combines with 
Theorem \ref{sufficiency} to imply that the mapping $t\mapsto \qd(T_{-t}x)$ is continuous in the ${\mathcal L}_{d,\infty}$ topology. Since $d > 1$, the ideal ${\mathcal L}_{d,\infty}$
can be equipped with an equivalent Banach norm, and so the functions:
\begin{equation*}
    t\mapsto \psi(t)(T_{-t}x),\quad t\mapsto \psi(t)\qd (T_{-t}(x))
\end{equation*}
are both Bochner measurable in the Banach spaces $W^1_d(\qe)$ and ${\mathcal L}_{d,\infty}$ respectively. Theorem \ref{sufficiency} implies that $x\mapsto \qd x$ is a bounded linear map
from $W^1_d(\qe)$ to ${\mathcal L}_{d,\infty}$, and hence:
\begin{equation*}
    \qd\left(\int_{{\mathbb R}^d} \psi(t)T_{-t}(x)\,dt\right) = \int_{{\mathbb R}^d} \psi(t)\qd (T_{-t}(x))\,dt
\end{equation*}
where both integrals are Bochner integrals. This justifies \eqref{formal_bochner_computation}.

Noting that $T_{-t}$ both commutes with Fourier multipliers and is unitary on $L_2(\qe)$, it follows that:
\begin{equation*}
    \|\qd (T_{-t}x)\|_{d,\infty} = \|\qd x\|_{d,\infty},\quad t \in {\mathbb R}^d
\end{equation*}
and hence \eqref{formal_bochner_computation} implies:
\begin{equation}\label{convolution_upper_bound}
    \|\qd (\psi\ast x)\|_{d,\infty} \lesssim_d \|\psi\|_1\|\qd x\|_{d,\infty},\quad x \in W^1_d(\qe)
\end{equation}
(the constant which appears results from the necessity of switching to an equivalent norm for ${\mathcal L}_{d,\infty}$).

 We now proceed to the proof of Theorem \ref{necessity}. 
 

\begin{proof}[Proof of Theorem \ref{necessity}]
     We assume that $d > 2$ and $x\in L_d(\qe)+L_\infty(\qe)$. Suppose that $\qd x \in {\mathcal L}_{d,\infty}$. 
    
    From Corollary \ref{space_frequency_approximation} and Lemma \ref{Schwartz_approximation_is_possible}, we may select $\{\psi_{\varepsilon}\}_{\varepsilon>0}$, $\{\phi_{\varepsilon}\}_{\varepsilon> 0}$ and $\{\chi_{\varepsilon}\}_{\varepsilon>0}$ such that $\psi_{\varepsilon}\ast (U(\phi_{\varepsilon})U(\chi_{\varepsilon})x) \in \qcS$.

    The upper bound \eqref{convolution_upper_bound} implies:
    \begin{equation*}
        \|\qd (\psi_{\varepsilon}\ast (U(\phi_{\varepsilon})U(\chi_{\varepsilon})x))\|_{d,\infty} \lesssim_d \|\psi_{\varepsilon}\|_1\|\qd (U(\phi_{\varepsilon})U(\chi_{\varepsilon})x)\|_{d,\infty}.
    \end{equation*}
    Expanding the commutator using the Leibniz rule, the quasi-triangle inequality and Theorem \ref{sufficiency}:
    \begin{align*} 
        \|\qd (\psi_{\varepsilon}\ast (U(\phi_{\varepsilon})U(\chi_{\varepsilon})x))\|_{d,\infty} &\lesssim_d \|\psi_{\varepsilon}\|_1\|\qd(U(\phi_{\varepsilon})U(\chi_{\varepsilon})x)\|_{d,\infty}\\
                                                                                                  &\lesssim_d \|\psi_{\varepsilon}\|_1(\|(\qd U(\phi_{\varepsilon}))U(\chi_{\varepsilon})x\|_{d,\infty} + \|U(\phi_{\varepsilon})\qd(U(\chi_{\varepsilon}))x\|_{d,\infty}\\
                                                                                                  &\quad + \|U(\phi_{\varepsilon})U(\chi_{\varepsilon})\qd x\|_{d,\infty})\\
                                                                                                  &\lesssim_d \|\psi_{\varepsilon}\|_1(\|(\qd U(\phi_{\varepsilon}))U(\chi_{\varepsilon})x\|_{d,\infty} + \|U(\phi_{\varepsilon})\qd(U(\chi_{\varepsilon}))x\|_{d,\infty}\\
                                                                                                  &\quad + \|U(\phi_{\varepsilon})\|_\infty\|U(\chi_{\varepsilon})\|_\infty\|\qd x\|_{d,\infty}).
    \end{align*}

    By construction $\|\psi_{\varepsilon}\|_1$ is constant as $\varepsilon\to 0$, and applying Proposition \ref{HY-ineq}, we also have that $\|U(\phi_{\varepsilon})\|_\infty$
    and $\|U(\chi_{\varepsilon})\|_{\infty}$ are uniformly bounded as $\varepsilon\to 0$.
    We now argue that $\|(\qd U(\phi_{\varepsilon}))U(\chi_{\varepsilon})x\|_{d,\infty}$ and $\|U(\phi_{\varepsilon})\qd(U(\chi_{\varepsilon}))x\|_{d,\infty}$ are also uniformly bounded as $\varepsilon\to 0$. To see this,
    write $x$ as $x_0+x_1$, where $x_0 \in L_\infty(\qe)$ and $x_1 \in L_d(\qe)$. Then Theorem \ref{sufficiency} and Lemma \ref{qdy_x_lemma} yield the bound:
    \begin{equation*}
        \|U(\phi_{\varepsilon})\qd(U(\chi_{\varepsilon}))x\|_{d,\infty} \lesssim_d \|\phi_{\varepsilon}\|_1\|U(\chi_{\varepsilon})\|_{\dot{W}^1_d}\|x_0\|_\infty+\|\phi_{\varepsilon}\|_1\|U(\chi_{\varepsilon})\|_{W^1_\infty}\|x_1\|_d
    \end{equation*}
    and a similar bound for $\|(\qd U(\phi_{\varepsilon}))U(\chi_{\varepsilon})x\|_{d,\infty}$.
    
    Due to Lemma \ref{cancellation_lemma}, the seminorms $\|U(\phi_{\varepsilon})\|_{\dot{W}^1_d}$ and $\|U(\chi_{\varepsilon})\|_{\dot{W}^1_d}$ are uniformly bounded as $\varepsilon\to 0$. Similarly, the $W^1_\infty$-norms of $U(\chi_{\varepsilon})$
    and $U(\psi_{\varepsilon})$ are uniformly bounded as $\varepsilon\to 0$.     
    
    It follows that
    $\{\qd\big(\psi_{\varepsilon}\ast (U(\phi_\varepsilon)U(\chi_\varepsilon)x)\big) \}_{\varepsilon>0}$ is uniformly bounded in ${\mathcal L}_{d,\infty}$ as $\varepsilon\to 0$.
    Now applying Corollary \ref{trace formula-bound} to $\qd\big(\psi_{\varepsilon}\ast(U(\phi_{\varepsilon})U(\chi_{\varepsilon})x)\big)$, it follows that $\{\psi_{\varepsilon}\ast(U(\phi_{\varepsilon})U(\chi_{\varepsilon})x)\}_{\varepsilon>0}$ is uniformly bounded in $\dot{W}_d^1(\qe)$, so for every $1\leq j \leq d$, $\{\partial_j \big(\psi_{\varepsilon}\ast(U(\phi_{\varepsilon})U(\chi_{\varepsilon})x)\big)\}_{\varepsilon>0}$ is uniformly bounded in $L_d(\qe)$. Since $d\geq 2$, the space $L_d(\qe)$
    is reflexive and therefore $\{\partial_j(\psi_{\varepsilon}\ast (U(\phi_{\varepsilon})U(\chi_{\varepsilon})x))  \}_{\varepsilon> 0}$ has a weak limit point in $L_d(\qe)$. But we know from Theorem \ref{space_frequency_approximation} that if $y \in L_{d/(d-1)}(\qe)$
    or $y \in L_1(\qe)$, then $U(\chi_{\varepsilon})U(\phi_{\varepsilon})(\psi_{\varepsilon}\ast y)\rightarrow y$ in the $L_{d/(d-1)}(\qe)$ sense or in the $L_1(\qe)$ sense respectively; hence that $\psi_{\varepsilon}\ast(U(\phi_{\varepsilon})U(\chi_{\varepsilon})x) \ra  x $ in the the distributional sense. It follows that the weak limit point of $\{\partial_j(\psi_{\varepsilon}\ast (U(\phi_{\varepsilon})U(\chi_{\varepsilon})x))\}_{\varepsilon>0}$ in $L_d(\qe)$ must also be $\partial_j x$.  
    
    Therefore, $\partial_j x \in L_d(\qe)$ for every $1\leq j \leq d $. That is, $x \in W^{1}_{d}(\qe)$.
    
    Finally, we obtain the bound on the norm using Corollary \ref{trace formula-bound}. That result implies that there exists a constant $c_d >0 $ such that for all continuous normalised traces $\vf$ on ${\mathcal L}_{1,\infty}$,
    $$\|x\|_{\dot{W}_d^1} \lesssim_{d}  \vf(|\qd x|^d )  ^{\frac 1 d }.$$
    Since $\vf$ is continuous,
    \begin{equation*}
        \|x\|_{\dot{W}_d^1} \lesssim_{d}\|\vf\|_{({\mathcal L}_{1,\infty})^*}\|\qd x\|_{d,\infty}.
    \end{equation*}
    Selecting a continuous normalised trace $\vf$ of norm $1$ completes the proof  for $d > 2$.

    For $d = 2$, we make the stronger assumption that $x \in L_\infty(\qe)$. This permits us to carry out the same proof, but instead we use the bounds:
    \begin{align*}
        &\|U(\phi_{\varepsilon})\qd(U(\chi_{\varepsilon}))x\|_{2,\infty} \lesssim_d \|\phi\|_1\|\chi_{\varepsilon}\|_{W^1_2}\|x\|_\infty,\\
        &\|\qd(U(\phi_{\varepsilon}))U(\chi_{\varepsilon})x\|_{2,\infty} \lesssim_d \|\chi\|_1\|\phi_{\varepsilon}\|_{W^1_2}\|x\|_\infty
    \end{align*}
    to prove that $\{\qd\big(\psi_{\varepsilon}\ast (U(\phi_\varepsilon)U(\chi_\varepsilon)x)\}$ is uniformly bounded in $L_{2,\infty}$.
    
%
%

\end{proof}

\medskip

\noindent{\bf Acknowledgements.} 
The authors would like to thank to anonymous referees for numerous helpful comments and corrections to Lemma \ref{Schwartz_approximation_is_possible}. We also
extend our gratitude to Galina Levitina for noticing a gap in our original proof of Theorem \ref{necessity}.
This work was done when the third author was visiting the University of New South Wales; he wishes to express his gratitude to the first two authors for their kind hospitality. The authors are supported by Australian Research Council (grant no. FL170100052); X. Xiong is also partially supported by the National Natural Science Foundation of China (grant no. 11301401).

\end{document}